\documentclass[12pt]{amsart}
\usepackage[utf8]{inputenc}
\usepackage{enumerate}
\usepackage{amsmath}
\usepackage{melanie}
\usepackage{brandon}
\usepackage{graphicx}
\usepackage{amsfonts}
\usepackage{amsmath}
\usepackage{amssymb}
\usepackage{fancyhdr}
\usepackage{indentfirst}
\usepackage{tikz}
\usepackage{tikz-cd}
\usepackage{capt-of}
\usepackage{booktabs}
\usepackage{verbatim}
\usepackage{color}
\usepackage{amsthm}
\usepackage{subfigure}
\usepackage{multirow}
\usepackage{longtable}
\usepackage{xcolor}
\usepackage{hyperref}

\newtheorem{theorem}{Theorem}[section]
\newtheorem{lemma}[theorem]{Lemma}

\newtheorem{corollary}[theorem]{Corollary}
\newtheorem{proposition}[theorem]{Proposition}

\newtheorem{conjecture}{Conjecture}
\theoremstyle{plain}
\newtheorem{data}[theorem]{Data Analysis}
\theoremstyle{remark} \newtheorem{remark}[theorem]{Remark}

\usepackage[margin=1in]{geometry}

\numberwithin{equation}{section}

\title{Inductive Methods for Counting Number Fields}
\author[B. Alberts]{Brandon Alberts}
\author[R.J. Lemke Oliver]{Robert J. Lemke Oliver}
\author[J. Wang]{Jiuya Wang}
\author[M.M. Wood]{Melanie Matchett Wood}

\begin{document}

\begin{abstract}
	We give a new method for counting extensions of a number field asymptotically by discriminant, which we employ to prove many new cases of Malle's Conjecture and counterexamples to Malle's Conjecture.  
	We consider families of extensions whose Galois closure is a fixed permutation group $G$.
	Our method relies on having asymptotic counts for $T$-extensions for some normal subgroup $T$ of $G$, uniform bounds for the number of such $T$-extensions, and possibly weak bounds on the asymptotic number of $G/T$-extensions.  However, we do not require that most $T$-extensions of a $G/T$-extension are $G$-extensions.  Our new results use $T$ either abelian or $S_3^m$, though our framework is general.
\end{abstract}

\maketitle

\section{Introduction}

	Let $k$ be a number field, $\bar{k}$ a fixed choice of its algebraic closure, and $G$ a permutation group of degree $n$ (i.e.\ transitive subgroup of the symmetric group $S_n$). We call a field extension $L/k$ a \emph{$G$-extension} if
	the Galois closure $\widetilde{L}$ of $L$ over $k$ has Galois group $\Gal(\widetilde{L}/k)$ which, acting on the embeddings $L\hookrightarrow \widetilde{L}$,
	is isomorphic to $G$ as a permutation group. Define a set of field extensions
	\[
	\mathcal{F}_{n,k}(G;X) = \#\{L/k :L\sub \bar{k},\ [L:k]=n,\  \Gal(\tilde{L}/k) \cong G,\ |\disc(L/k)|\le X\},
	\]
	where $ \cong $ denotes isomorphism as permutation groups and $|\cdot|$ denotes the norm down to $\Q$. The subscript $n$ is redundant since $G$ is a permutation group, but we include it because the degree is often a convenient reminder of which permutation representation we are considering for a particular abstract group.

	Number fields are ``counted'' by studying the asymptotic growth of $\#\mathcal{F}_{n,k}(G;X)$ as $X\to \infty$. Malle \cite{malle2002,malle2004} was the first to make general predictions for this rate of growth, leading to the following conjecture.

	\begin{conjecture}[Number Field Counting Conjecture]\label{conj:number_field_counting}
		Let $k$ be a number field and $G$ a  transitive permutation group of degree $n$. Then there exist positive constants $a,b,c > 0$ depending on $k$ and $G$ such that
		\[
		\#\mathcal{F}_{n,k}(G;X)
			\sim c X^{1/a} (\log X)^{b-1}
		\]
		as $X\to \infty$.
	\end{conjecture}

	In this paper we prove many new cases of Conjecture~\ref{conj:number_field_counting} for a class of permutation groups we call \emph{concentrated}.  We say a transitive permutation group $G$ is concentrated in a proper normal subgroup $N$ if $N$ contains all minimum index elements of $G$, and we say that $G$ is concentrated if this holds for some proper normal subgroup $N$.  Equivalently, $G$ is concentrated if
		\[
			\langle g \in G : \ind(g) = a(G) \rangle \ne G,
		\]
	where $\ind\colon S_n\to \Z$ is the index function defined by $\ind(g) = n - \#\{\text{orbits of }g\}$ and $a(G) = \min_{g\in G-\{1\}} \ind(g)$.
	While we do not prove Conjecture \ref{conj:number_field_counting} for any non-concentrated groups, our methods are able to improve the known upper bounds for many non-concentrated groups as well.

	Our strategy requires, as input, fairly weak information about the number of $G/N$-extensions, at the expense of requiring ``uniform'' information about relative $N$-extensions of number fields. It is crucial to note that our strategy does not require that $G$ is an imprimitive permutation group.  In other words, we do not require that the extensions we count have subextensions.  Moreover, in significant contrast to previous work that has a similar flavor, we do not require that ``most'' $G/N$ extensions of an $N$-extension are $G$-extensions.  In particular, this means that we are not limited to the case where $G$ is a wreath product.  These new aspects of our approach allow us to count number fields asymptotically for many more groups than all prior approaches.

	We specifically show how this strategy can be executed when $N$ is abelian (where the uniformity required is closely related to the sizes of torsion subgroups of class groups) and, in many cases, when $N \cong S_3^r$ for some $r \geq 1$. This leads to a proof of Conjecture~\ref{conj:number_field_counting} for many new, infinite families of transitive 
	groups---for example, for many new nilpotent groups $G$.

	Our main results are Theorem \ref{thm:main_S3_wreath} and Theorem \ref{thm:main_abelian_on_top}. These theorems take as input upper bounds on the number of $G/N$-extensions and on the average size of certain torsion subgroups of the class group of such extensions. We convert these bounds into an upper bound for $\#\mathcal{F}_{n,k}(G;X)$, and if these ``input bounds'' are sufficiently small, then we are able to prove Conjecture~\ref{conj:number_field_counting} for $\#\mathcal{F}_{n,k}(G;X)$ with explicit expressions for $a$ and $b$.  The inductive nature of these results means that each time we prove Conjecture~\ref{conj:number_field_counting} (or even obtain an improved upper bound) for one group $G$, we can take that and use it as input in our main theorems to prove further cases of Conjecture~\ref{conj:number_field_counting}. This has a compounding effect on the number of new results we are able to prove.

	In all cases where we prove Conjecture~\ref{conj:number_field_counting} (and in all other cases where it is known), the constant $a$ agrees with Malle's predicted value, which is $a(G)$ \cite{malle2002}.  The correctness of this value is referred to as the Weak Form of Malle's Conjecture, 
	usually expressed in the softer form
	\[
		X^{1/a(G)} \ll_{k,G} \#\mathcal{F}_{n,k}(G;X) \ll_{k,G,\epsilon} X^{1/a(G) + \epsilon}.
	\]

	Malle also proposed a value for $b$ in \cite{malle2004}, which he denotes $b(k,G)$. This number is given by the number of orbits of minimal index elements in $G$ under the cyclotomic action, i.e. the Galois action $x.g = g^{\chi(x)^{-1}}$ for $\chi\colon G_K\to \hat{\Z}$ the cyclotomic character. Conjecture~\ref{conj:number_field_counting} with the values $a=a(G)$ and $b=b(k,G)$ is referred to as the Strong Form of Malle's Conjecture. While the strong form is known to be true in a number of cases, Kl\"uners  \cite{kluners2005} gave a counterexample by proving that
	\[
	\#\mathcal{F}_{6,\mathbb{Q}}(C_3\wr C_2;X) \gg X^{1/2}\log X
	\]
	despite the fact that $b(\Q,C_3\wr C_2) = 1$. 
	Our results include proofs of Conjecture~\ref{conj:number_field_counting} for infinitely many groups where $b=b(k,G)$ agrees with Malle's prediction (thus proving the strong form of his conjecture in these cases), and infinitely many groups where $b\ne b(k,G)$ contradicts Malle's prediction.
	See Corollary \ref{cor:Malle_counter_example} where we verify Conjecture \ref{conj:number_field_counting} for infinitely many specific counterexamples to Malle's prediction as proposed by Kl\"uners. However, in general the expression we give for $b$ can be difficult to evaluate and depends on the existence of solutions to certain embedding problems.
	\begin{remark}
		T\"urkelli has proposed a corrected value of $b$ \cite{turkelli2015} in Malle's Conjecture. Wang \cite{wang2024} has evaluated the expression for the correct $b$ proved in this paper for a certain infinite family of examples to show that T\"urkelli's modified prediction is incorrect.
	\end{remark}

\subsection{Example Corollaries}

	It is not feasible to provide an exhaustive list of the types of groups for which we prove Conjecture~\ref{conj:number_field_counting} because our results are flexible enough to be applied in many different situations. Thus, before we state our main technical results, we present several representative cases and families of groups for which our main results give a proof of Conjecture \ref{conj:number_field_counting}.
	These results will be proved in Section~\ref{sec:examples}.

	\begin{corollary}\label{cor:nilpotent}
		Let $G$ be a finite nilpotent transitive  permtuation group for which $\langle g\in G-\{1\} : \ind(g) = a(G)\rangle$ is abelian, i.e. all the elements of minimal index commute with each other.

		Then Conjecture \ref{conj:number_field_counting} holds for $G$ over any number field.
	\end{corollary}

	Previously, the most far reaching result proving Conjecture \ref{conj:number_field_counting} in this direction was the work of Koymans and Pagano \cite[Theorem 1.1]{koymans_pagano2021}, which proves Conjecture \ref{conj:number_field_counting} when $G$ is nilpotent in the regular representation with $\langle g\in G-\{1\}: \ind(g) = a(G)\rangle$ contained in the center of $G$.  We allow any permutation representation and even in the regular representation case have a weaker hypothesis on the minimal index elements.  For example, we now know
	Conjecture \ref{conj:number_field_counting} for the holomorph group ${\operatorname{Hol}}(D_4)=D_4\rtimes \Aut(D_4)$ in its degree 8 affine transformation action on the elements of $D_4$. As a benchmark, to the best of our knowledge, previous methods, including \cite{kluners2012,koymans_pagano2021,kluners-wang2023idelic}, for proving Conjecture~\ref{conj:number_field_counting} are amenable for at most $1538$ of the $2{,}739{,}294$ nilpotent transitive groups of degree up to $32$. A computation in \verb^Magma^ \cite{Magma} shows that Corollary~\ref{cor:nilpotent} proves Conjecture~\ref{conj:number_field_counting} for at least $2{,}686{,}926$ of the nilpotent transitive groups of degree up to $32$.

	When $G$ and $H$ are permutation groups, we write $G\wr H$ for their wreath product. We always take $G\wr H$ to be a permutation group in the wreath representation.
	For a positive integer $n$, we write $C_n$ for the cyclic group of order $n$ in its regular permutation representation.

	\begin{corollary}\label{cor:cyclic_wreath}
	Let $n$ be a positive integer, $\ell$ be the smallest prime dividing $n$, $B$ a transitive permutation group of degree $m$, and $k$ a number field.  
	If there exists at least one $B$-extension of $k$ and
	\[
		\#\mathcal{F}_{m,k}(B;X)\ll_{k} 
		X^{\frac{1}{2}+\frac{1}{\ell-1}-\delta}
	\]
	for some $\delta > 0$,  then Conjecture \ref{conj:number_field_counting} holds for $G=C_n\wr B$ over $k$.

	In particular, Conjecture \ref{conj:number_field_counting} holds for 
	$C_n\wr B$ in each of the following situations:
	\begin{itemize}
		\item $B$ is in its regular permutation representation, with $|B|>2$ if $n$ is odd, and there is at least one $B$-extension of $k$;
		\item $B$ is a nilpotent group, not containing a transposition if $n$ is odd; or 
		\item $B$ is a finite simple group of Lie type over $\mathbb{F}_q$ with rank $r$ with $q \geq q_0(r)$ for some absolute constant $q_0(r)$ depending only on $r$, $B$ occurs as a Galois group over $k$, and $B$ is in any primitive permutation representation of non-minimal degree, other than $\mathrm{PSU}_6(\mathbb{F}_q)$ in its non-minimal action on the parabolic subgroup $P_3$.
	\end{itemize}
	\end{corollary}

	In the case $G=C_2\wr B$, this strengthens Kl\"uners' result that Conjecture \ref{conj:number_field_counting} is true for $G$ under the assumption that $1\leq \#\mathcal{F}_{m,k}(B;X)\ll_\epsilon X^{1+\epsilon}$ \cite[Corollary 5.10]{kluners2012}.
	For wreath products $C_n \wr B$ with $n>2$, these are the first results of this form.

	By taking advantage of the inductive nature of our results, we can iterate these examples to prove the following:

	\begin{corollary}\label{cor:iterated_cyclic_wreath}
	Let $k$ be a number field, let $n_1, \dots, n_r \geq 2$ be integers, and let $G=C_{n_1}\wr C_{n_2} \wr \cdots \wr C_{n_r}$. Suppose that any one of the following holds:
		\begin{enumerate}[(a)]
			\item $n_2 > 2$,
			\item $n_1,n_2,...,n_r$ are all powers of $2$,
			\item $n_1=2^d$ and $n_2=2$, and if $r \ge 3$ we also have $1 \le d < 6 - 4/n_3$,
			\item $n_1 = 2^d 3$ and $n_2=2$, and if $r\ge 3$ we also have $0 \le d < 13/3 - 4/n_3$.
		\end{enumerate}
		Then Conjecture \ref{conj:number_field_counting} holds for $G$ over $k$.
		
		In particular, this includes all the Sylow $p$-subgroups $C_p^{\wr r}$ of $S_{p^r}$ as well as Kl\"uners' \cite{kluners2005} original counterexample to Malle's Conjecture, $C_3\wr C_2$ in degree $6$.
	\end{corollary}

	The family of wreath products of two cyclic groups is rich with important behavior. The above result contains the first groups for which Conjecture \ref{conj:number_field_counting} has been proven with an asymptotic that disagrees with Malle's prediction, including Kl\"uners' counterexample $C_3\wr C_2$. Kl\"uners proposed a larger family of wreath products for which his arguments show that Malle's conjecture is incorrect \cite[Page 413]{kluners2005}, many of which fall under Corollary \ref{cor:iterated_cyclic_wreath}:

	\begin{corollary}[Kl\"uners' counterexamples]\label{cor:Malle_counter_example}
		Let $G = C_\ell \wr C_d$ for $d\mid \ell-1$ and $k\cap \Q(\zeta_\ell) = \Q$. If $d > 2$ then Conjecture \ref{conj:number_field_counting} is true for $G$ over $k$, but $b\ne b(k,G)$ disagrees with Malle's prediction.
	\end{corollary}

	We similarly prove results for wreath products with $S_3$ in place of a cyclic group.
	\begin{corollary}\label{cor:S3_wreath}
		Let $G=S_3\wr B$ be the wreath product of $S_3$ in degree $3$ with a transitive permutation group $B$ of degree $m$. Let $k$ be a number field for which there exists at least one $B$-extension of $k$.

		\begin{enumerate}[(i)]
			\item Suppose that
			\[
				\#\mathcal{F}_{m,k}(B;X) \ll_k X^{\frac{5}{3}+\frac{1}{3m[k:\Q]}-\delta}
			\]
			for some $\delta > 0$. Then Conjecture \ref{conj:number_field_counting} holds for $G$ over $k$.
			\item Furthermore, if $B$ is primitive and
			\[
				\#\mathcal{F}_{m,k}(B;X) \ll_k X^{\frac{5}{3}+\frac{10}{18m-15}-\delta}
			\]
			for some $\delta > 0$, then Conjecture \ref{conj:number_field_counting} holds for $G$ over $k$.
		\end{enumerate}
		In particular, Conjecture \ref{conj:number_field_counting} holds for $S_3\wr B$ over any number field when $B$ is any of the examples in the bulleted list of Corollary~\ref{cor:cyclic_wreath}. Conjecture \ref{conj:number_field_counting} also holds for the iterated wreath products $S_3^{\wr r}$ in degree $3^r$.
	\end{corollary}

	Conjecture \ref{conj:number_field_counting} was not previously known for any groups in this family except the trivial case $S_3\wr 1 \cong S_3$.

	Besides wreath products, our results also allow us to access many other groups expressible as semidirect products such as the following:

	\begin{corollary}\label{cor:trace_0_semidirect}
		Let $k$ be a number field, let $B$ be a transitive permutation group of degree $m$, let $p$ be a prime, and let $W \leq \mathbb{F}_p^m$ be the trace $0$ subspace.  Let $G = W \rtimes B$, where $B$ acts on $W$ by means of its degree $m$ permutation representation.  If  there is at least one $B$-extension of $k$, and there is some $\delta>0$ so that
			\[
				\#\mathcal{F}_{m,k}(B;X) \ll_{k,m} X^{\frac{1}{2(p-1)} - \delta},
			\]
		then Conjecture~\ref{conj:number_field_counting} holds for $G$ in its degree $pm$ permutation representation on the cosets of $W_1 \rtimes B$, where $W_1 < W$ is the subspace with first coordinate $0$.

		In particular, Conjecture~\ref{conj:number_field_counting} holds for $G$ when $B$ is:
			\begin{itemize}
				\item in its regular representation, and $|B| \geq 38 \cdot (p-1)^2$; or
				\item nilpotent, with $a(B) > 2(p-1)$.
				\item $B$ is a finite simple group of Lie type over $\mathbb{F}_q$ with rank $r$ with $q \geq q_1(r,p)$ for some absolute constant $q_1(r,p)$ depending only on $r$ and $p$, $B$ occurs as a Galois group over $k$, and $B$ is in any primitive permutation representation of non-minimal degree, other than $\mathrm{PSU}_6(\mathbb{F}_q)$ in its non-minimal action on the parabolic subgroup $P_3$.
			\end{itemize}
	\end{corollary}

	As one final benchmark, to the best of our knowledge, previous methods for proving Conjecture~\ref{conj:number_field_counting} are amenable for at most $237$ transitive permutation groups of degree up to 23, including $167$ non-nilpotent groups.

	\begin{corollary}\label{cor:compute}
		Conjecture~\ref{conj:number_field_counting} holds for at least $1665$
		transitive permutation groups of degree up to 23 over $\Q$, and for at least $339$  such groups that are not nilpotent. 
	\end{corollary}

	This follows from a computation in \verb^Magma^ that is explained in more detail in Section~\ref{sec:examples}.  The code for this computation is available at \cite{codeurl}.

\subsection{Main Results}

	We state first our result for wreath products of the form $S_3\wr B$.

	\begin{theorem}\label{thm:main_S3_wreath}
		Let $k$ be a number field, let $B$ be a transitive permutation group of degree $m$ such that there is at least one $B$ extension of $k$, and let $G = S_3 \wr B$.
		
		Suppose $\theta \ge 0$ is such that
		\begin{equation}\label{eqn:S3_wreath_class_hypothesis}
			\sum_{F \in \mathcal{F}_{m,k}(B;X)} |\Cl_{F}[2]|^{2/3} \ll_{m,k} X^{\theta}.
		\end{equation}
		
		Then the following hold:
		\begin{enumerate}[(i)]
			\item If $\theta < 2$ then there exists a positive constant $c(k,G)>0$ such that
			\[
				\#\mathcal{F}_{3m,k}(G;X) \sim c(k,G)X.
			\]
			\item If $\theta \ge 2$ then
			\[
				\#\mathcal{F}_{3m,k}(G;X) \ll_{m,k,\epsilon} X^{\frac{\theta+1}{3}+\epsilon}.
			\]
		\end{enumerate}
	\end{theorem}

	For $n\geq 2$, we have $a(S_n\wr B) = 1$ and $b(k,S_n\wr B) = 1$ \cite[Lemma 2.2]{malle2004}.  Hence,
	Theorem~\ref{thm:main_S3_wreath} (i) proves cases of the Strong Form of Malle's Conjecture.

	In practice, the bound $X^\theta$ in \eqref{eqn:S3_wreath_class_hypothesis} is often proven as a hybrid bound $\theta = \beta + t$ where $\beta \geq 0$ is such that
	\begin{equation} \label{eqn:B-bound}
		\#\mathcal{F}_{m,k}(B;X) \ll_{m,k} X^\beta,
	\end{equation}
	and where $t \geq 0$ is such that
	\[
		|\Cl_F[2]|^{2/3} \ll_{m,k} |\disc(F/\Q)|^t
	\]
	for each $B$ extension $F/k$. For example, the first part of Corollary \ref{cor:S3_wreath} will follow from the bound $|\Cl_F[2]|\ll |\disc(F/\Q)|^{\frac{1}{2} - \frac{1}{2[F:\Q]}+\epsilon}$ \cite{bhargava_shankar_taniguchi_thorne_tsimerman_zhao_2020}. It is often possible to do better bounding $t$ on average, for example \cite[Corollary 1.14]{LemkeOliver2024} gives a bound for the average size of $|\Cl_F[2]|$ when varying $F$ over primitive extensions. This is used to prove the second part of Corollary \ref{cor:S3_wreath}.  

	The groups $S_3\wr B$ are examples of imprimitive groups. An imprimitive extension of fields is one that has some intermediate subfield. In the case of $S_3\wr B$, any $L\in \mathcal{F}_{3m,k}(S_3\wr B;X)$ is necessarily a cubic extension of some $F\in \mathcal{F}_{m,k}(B;X)$. In general, an imprimitive group $G$ can always be realized as a subgroup of $H\wr B$ for some permutation groups $H,B$ where the projection of $H^m \cap G$ onto each coordinate of $H^m$ is surjective and $G$ surjects onto $B$. In the language of field extensions, any $L\in \mathcal{F}_{nm}(G;X)$ is a tower of field extensions $L/F/k$ for $L\in \mathcal{F}_{n,F}(H;X)$ and $F\in \mathcal{F}_{m,k}(B;X)$. Given an imprimitive group $G$ realized in this way, we say that $(H,B)$ is a \emph{tower type} for $G$, as defined in \cite[page 12]{LO-uniform}.

	The majority of new cases we prove for Conjecture \ref{conj:number_field_counting} follow from Theorem \ref{thm:main_abelian_on_top} below, which applies when $G$ has an abelian normal subgroup. This theorem is structured similarly to Theorem \ref{thm:main_S3_wreath}, but with the added benefit that it does not require $G$ to have a certain imprimitive structure. 
	The statement of Theorem \ref{thm:main_abelian_on_top} requires us to develop some terminology. To give the reader an idea of what to expect, we state a corollary for some imprimitive permutation groups. We will prove that Corollary~\ref{cor:main_abelian_on_top_imprimitive} follows from Theorem \ref{thm:main_abelian_on_top} in Subsection \ref{subsec:proving_abelian_on_top_imprimitive}.
	Let $a(U) = \min_{g\in U-\{1\}} \ind(g)$ for any subset $U$ of a permutation group.

	\begin{corollary}\label{cor:main_abelian_on_top_imprimitive}
		Let $k$ be a number field and $G$ be an imprimitive
		transitive permutation group with tower type $(A,B)$ for which $A$ is a finite abelian group and $B$ is a transitive permutation group of degree $m$ such that there is at least one $B$ extension of $k$.
		
		Suppose $\theta \ge 0$ is such that
		\begin{equation}\label{eqn:Abelian_imprimitive_class_hypothesis}
			\sum_{F \in \mathcal{F}_{m,k}(B;X)} |\Hom(\Cl_F, A)| \ll_{m,|A|,k} X^{\theta}.
		\end{equation}
		
		Then the following hold:
		\begin{enumerate}[(i)]
			\item If $\theta < \frac{|A|}{a(A^m \cap G)}$ then there exists a positive constant $c(k,G)>0$ such that
			\[
				\#\mathcal{F}_{m|A|,k}(G;X) \sim c(k,G)X^{1/a(A^m \cap G)}(\log X)^{b-1},
			\]
			where $b\ge 0$ is some integer (given explicitly in Theorem \ref{thm:main_abelian_on_top}).
			\item If $\theta \ge \frac{|A|}{a(A^m \cap G)}$ then
			\[
				\#\mathcal{F}_{m|A|,k}(G;X) \ll_{m,|A|,k,\epsilon} X^{\theta/|A|+\epsilon}.
			\]
		\end{enumerate}
	\end{corollary}

	Equation (\ref{eqn:Abelian_imprimitive_class_hypothesis}) is very similar to (\ref{eqn:S3_wreath_class_hypothesis}). Indeed, if the abelian group is given by
	\[
		A = \prod_{\ell\text{ prime}} \prod_{i=0}^{r_\ell} \Z/\ell^{n_{\ell,i}}\Z,
	\]
	it  follows that
	\[
		|\Hom(\Cl_F,A)| \le \prod_{\ell\text{ prime}} \prod_{i=0}^{r_\ell}|\Cl_F[\ell^{n_{\ell,i}}]| \le \prod_{\ell\text{ prime}}|\Cl_F[\ell]|^{\sum_{i=0}^{r_\ell} n_{\ell,i}}.
	\]

	In order to move away from imprimitive structures for the full statement of Theorem \ref{thm:main_abelian_on_top}, it is convenient to structure the result as counting elements of the set of continuous surjective group homomorphisms
	\[
		\Sur(G_k,G;X) = \{\pi\in \Sur(G_k,G) : |\disc_G(\pi)|\le X\},
	\]
	where $G_k$ is the absolute Galois group of $k$, $\disc_G(\pi)$ is the relative discriminant $\disc(F/k)$, and $F$ is the field fixed by $\pi^{-1}(\Stab_G(1))$. The Galois correspondence 
	implies that the size of this set is $\#\mathcal{F}_{n,k}(G;X)$ times a constant depending only on $G$.  We make this precise in Lemma~\ref{lem:sur-to-fields} below.

	Our method of proof will involve partitioning the set of $G$-extensions by the subfield of the Galois closure fixed by a particular normal subgroup $T\normal G$. This is very naturally described in terms of the surjections. If $T\normal G$ has canonical quotient map $q\colon G \to G/T$, we consider the pushforward
	\[
	q_*\colon \Sur(G_k,G) \to \Sur(G_k,G/T).
	\]
	The elements of the image $q_*\Sur(G_k,G)\subseteq \Sur(G_k,G/T)$ correspond via Galois theory to
	Galois $G/T$-extensions $M/k$ for which there exists a $G$-extension $F/k$ whose Galois closure $\widetilde{F}$ has fixed field $\widetilde{F}^T = M$. The elements of a fiber of $q_*$ correspond to such $G$-extensions $F$. 
	Even though $G$ is a permutation group, we are forgetting this structure when we take the quotient $G/T$, which we consider as a permutation group in its regular representation. 

	Our main results take counting results for certain $G/T$-extensions, which we express as counting elements of the set
	\[
	q_*\Sur(G_k,G;X) = \{\pi\in \Sur(G_k,G/T) : \pi = q\circ\widetilde{\pi}\text{ for some }\widetilde{\pi}\in \Sur(G_k,G;X)\},
	\]
	as an input towards counting $G$-extensions. In particular, the elements of $q_*\Sur(G_k,G;X) \subseteq \Sur(G_k,G/T)$ correspond to Galois $G/T$-extensions which are equal to the fixed field $\widetilde{F}^T$ of the Galois closure $\widetilde{F}$ of some $G$-extension $F/k$ with $|\disc(F/k)|\le X$.

	In order to state the explicit values for $a$ and $b$ in our proven cases of Conjecture \ref{conj:number_field_counting}, we give the definitions for certain invariants appearing in the Twisted Malle Conjecture \cite[Conjecture 3.10]{alberts2021}, stated as Conjecture \ref{conj:twisted_number_field_counting} below without explicit values for $a,b,c$. Let $T\normal G $ be a normal subgroup of a finite transitive permutation group.
	\begin{enumerate}[(i)]
		\item When $\pi\colon G_k\to G$ is a continuous homomorphism, we define $T(\pi)$ to be the group $T$ together with the Galois action $x.t = \pi(x) t \pi(x)^{-1}$. When $T$ is abelian, the Galois module $T(\pi)$ depends only on the pushforward $q_*\pi$. For this reason, we often abuse notation and write $T(q_*\pi)$ for $T(\pi)$ in this case.
		\item The cohomology group $H^1_{ur}(k,T(\pi))$ is the subgroup of unramified classes,
		\[
			H^1_{ur}(k,T(\pi)) = \{f\in H^1(k,T(\pi)) : \forall \fp,\ {\rm res}_{I_{\fp}}(f) = 0\},
		\]
		where $\fp$ ranges over all finite and infinite places of $k$, and $I_\fp$ is the inertia group of $k$ at each finite place $\fp$ and the decomposition group at each infinite place.
		\item $a(T) = \underset{t\in T-\{1\}}{\min}\ind(t)$ is the minimum index of elements in $T$, where $T$ is viewed as a subset of the permutation group $G$, and
		\item $b(k,T(\pi))$ is the number of orbits of elements $\{t\in T: \ind(t) = a(T)\}$ of minimal index  under the Galois action $x\colon t\mapsto (\pi(x)t\pi(x)^{-1})^{\chi^{-1}(x)}$ for $\chi:G_k\to \hat{\Z}^{\times}$ the cyclotomic character. This is the action induced from $\pi$ twisted by the cyclotomic character $\chi\colon G_k\to \hat{\Z}^{\times}$.
	\end{enumerate}
	These invariants occur naturally in our method of proof.

	\begin{theorem}\label{thm:main_abelian_on_top}
		Let $k$ be a number field and $G$ a finite transitive permutation group of degree $n$ for which there exists at least one $G$-extension of $k$. Suppose that $T\normal G$ is a proper normal subgroup and that $T$ is abelian. 
		
		Suppose $\theta \ge 0$ is such that
			\begin{equation}\label{eqn:abelian_H1ur_hypothesis}
				\sum_{\pi\in q_*\Sur(G_k,G;X)} |H_{ur}^1(k,T(\pi))| \ll_{n,k} X^{\theta}.
			\end{equation}
		Then the following hold:
		\begin{enumerate}[(i)]
			\item If $\theta < 1/a(T)$ then there exists a positive constant $c(K,G)>0$ such that
			\[
			\#\mathcal{F}_{n,k}(G;X) \sim c(k,G)X^{1/a(T)}(\log X)^{\max_\pi b(k,T(\pi)) - 1},
			\]
			where the maximum is taken over $\pi\in q_*\Sur(G_k,G)$.
			\item If $\theta\ge 1/a(T)$ then
			\[
				\#\mathcal{F}_{n,k}(G;X) \ll_{n,k,\epsilon} X^{\theta+\epsilon}.
			\]
		\end{enumerate}
	\end{theorem}

	Theorem \ref{thm:main_abelian_on_top}(i) proves Conjecture \ref{conj:number_field_counting} when the hypotheses apply, and is the source of the majority of our examples in the introduction.

	\begin{remark}
		Corollary \ref{cor:main_abelian_on_top_imprimitive} is the case of Theorem \ref{thm:main_abelian_on_top} where $G\subseteq A\wr B$ is an imprimitive group with tower type $(A,B)$ for some abelian group $A$ and for which we specifically take $T = A^m \cap G$. As we will see in Section \ref{sec:examples}, different choices for $T$ can provide different quality results even for imprimitive groups. In this sense, Theorem \ref{thm:main_abelian_on_top} is significantly more flexible that Corollary \ref{cor:main_abelian_on_top_imprimitive}. For example, the full strength of Theorem \ref{thm:main_abelian_on_top} is required to prove Corollary \ref{cor:nilpotent}.
	\end{remark}

	The constants $a,b$ from Conjecture \ref{conj:number_field_counting} are made explicit in Theorem \ref{thm:main_abelian_on_top}(i), and $c$ is made explicit in the proof of Theorem \ref{thm:main_abelian_on_top}. It is not guaranteed that these constants agree with Malle's predictions. However, often they do.

	In all cases in which we apply Theorem \ref{thm:main_abelian_on_top}, we have
	$\theta \ge 1/a(G\setminus T):= \min_{g\in G\setminus T} \ind(g)$.
	Indeed, a generalized version of Malle's weak lower bound predicts that $\#q_*\Sur(G_k,G;X) \gg X^{1/a(G\setminus T)-\epsilon}$. There are no known counterexamples to this prediction. Clearly $1/a(G\setminus T)<1/a(T)$ implies that $a(T)=a(G)$, and hence all current applications, as well as predicted future applications, of 
	Theorem \ref{thm:main_abelian_on_top}(i) prove Conjecture \ref{conj:number_field_counting} with the $a$-value as predicted by Malle.

	The value for $c$ will be expressed as a convergent sum of Euler products. This is reminiscent of the leading coefficient for $\#\mathcal{F}_{4,\Q}(D_4;X)$ \cite{cohen-diaz-y-diaz-olivier2002}, and in fact occurs for the same reason. All cases of Conjecture \ref{conj:number_field_counting} that we are able to prove with Theorem \ref{thm:main_abelian_on_top} have an \emph{accumulating subfield}, that is a nontrivial extension $L/k$ that lies inside a positive proportion of (the Galois closures of) $G$-extensions of $k$. Accumulating subfields are widely expected to prevent the leading coefficient from being equal to an arithmetically significant Euler product in the spirit of Bhargava's predictions for $S_n$ \cite{Bhargava2010}, and are known to cause a failure of independence when counting with restricted local conditions. See \cite{wood2009,Altug_Shankar_Varma_Wilson_2017,Shankar_Thorne_2022,alberts2024restricting} for further examples and discussion of this phenomenon.

	\begin{remark}
		During the preparation of this paper, Loughran--Santens released a preprint \cite{loughran2024malle} making predictions for the leading coefficient in Malle's conjecture (after removing a ``thin subset'' of fields).  Lougran--Santens \cite[Conjecture 1.3(2)]{loughran2024malle} predicts that the leading coefficient for concentrated groups is given in precisely the same way as our proof: as a sum of leading coefficients of fibers. However, they only count extensions that are linearly disjoint from $k(\mu_{\exp(G)})$, which makes their setup slightly different from ours. We expect that our methods will still apply to this setting, in particular noting that a version of Theorem \ref{thm:main_pointwise} still holds by the same proof if each extension which is not linearly disjoint from $k(\mu_{\exp(G)})$ is removed.
	\end{remark}

	As in Theorem \ref{thm:main_S3_wreath}, the bound $X^\theta$ is often computed as a hybrid bound $\theta = \beta + t$. Here, $\beta\ge 0$ is such that
	\[
		\# q_*\Sur(G_k,G;X) \ll_{n,k} X^\beta
	\]
	is an upper bound on the number of $G/T$-extensions parametrized by $q_*\Sur(G_k,G;X)$. The constant $t$ can be taken to satisfy
	\[
		|H^1_{ur}(k,T(\pi))| \ll_{n,k} |\disc(F(\pi)/k)|^t
	\]
	for each $\pi\in q_*\Sur(G_k,G;X)$, where $F(\pi)$ is the field fixed by $\pi^{-1}(\Stab_G(1))$. A bound for the average size $|H^1_{ur}(k,T(\pi))|$ might also be used here.

	The size of $q_*\Sur(G_k,G;X)$ has not been previously studied, although it can be bounded by sets of $G/T$-extensions with bounded invariants. We will later define a pushforward discriminant  $q_*\disc_G$ on $G/T$-extensions for which
	\[
		q_*\Sur(G_k,G;X)\subseteq \{\pi\in \Sur(G_k,G/T) : |q_*\disc_G(\pi)|\le X\}.
	\]
	See \eqref{E:push_forward} and \eqref{eq:pushforward_discriminant}. We can then use upper bounds for counting $G/T$-extensions ordered with respect to this pushforward discriminant as input in Theorem \ref{thm:main_abelian_on_top}. When $G$ is an imprimitive group, we will see in Proposition \ref{prop:imprim_beta} that the pushforward discriminant agrees with the discriminant of some transitive representation of $G/T$, which is the key observation implying Corollary \ref{cor:main_abelian_on_top_imprimitive}. In general, the quantity $\beta$ is about upper bounds for counting $G/T$-extensions.

	The object $H^1_{ur}(k,T(\pi))$ behaves like torsion in a class group. Indeed, if $T(\pi)$ carries the trivial action then $H^1_{ur}(k,T(\pi)) = \Hom(\Cl_k,T)$.
	In all cases, we can use Minkowski's bound on the size of the class group to give an initial bound
	\[
		|H_{ur}^1(k,T(\pi))| \ll_{k,\epsilon} |\disc(\pi)|^{d(\hat{T})/2+\epsilon},
	\]
	for $d(\hat{T})$ the minimum number of generators for $\hat{T}=\Hom(T,\Q/\Z)$ as a Galois module (see Lemma \ref{lem:H1_first_bound}).

	In order to take full advantage of the inductive nature of our results, we prove pointwise inductive bounds for $|H_{ur}^1(k,T(\pi))|$ that improve over this initial bound. Our main result in this direction is Lemma \ref{lem:inductive_H1ur_bound}, which bounds the size of $H^1_{ur}(k,T(\pi))$ in terms of $M$ and $T/M$ for some subgroup $M\le T$. In order to use Lemma \ref{lem:inductive_H1ur_bound} optimally, one must make a strategic choice for the subgroup $M$.

	It is not clear in general how to make such strategic choices, so we leave the complete discussion for Section \ref{sec:classgrp}. In certain cases, particular choices give us strong upper bounds for $|H^1_{ur}(k,T(\pi))|$ which we give below. 
	We say a $G_k$-module is \emph{constant over} a number field $F/k$ if 
	the Galois action factors through $\Gal(F/k)$.  
	The \emph{field of definition} of a $G_k$-module $A$ is the smallest Galois field extension $F/k$ such that $A$ is constant over $F$.
	The Galois module $T(\pi)$ has field of definition given by the fixed field of $\kappa\circ\pi\colon G_k\to \Aut(T)$, where $\kappa:G\to \Aut(T)$ is the action by conjugation. Equivalently, if $\pi$ corresponds to the $G$-extension $F/k$ then the field of definition for $T(\pi)$ is the subfield of $F$ fixed by the centralizer of $T$ in $G$.

	In the following result, a \emph{nilpotent $G$-module} \cite[Definition 1.2]{hilton1982nilpotency} is defined to be a $G$-module $M$ for which the lower central $G$-series terminates, where the lower central $G$-series is defined recursively by $\Gamma_G^1(M) = M$ and $\Gamma_G^j(M) = \langle m - x.m | x\in G, m\in \Gamma_G^{j-1}(M)\rangle$.

	\begin{corollary}\label{cor:inductive_H1ur_bounds}
	Let $A$ be a finite $G_k$-module that is constant over $F.$ 
	\begin{itemize}
		\item[(i)] If $A$ is a nilpotent $G_k$-module, 
		then
		\[
		|H^1_{ur}(k,A)| \ll_{k,|A|,\epsilon} |\disc(F/\Q)|^{\epsilon}.
		\]
		\item[(ii)] If $A$ is a simple $G_k$-module with exponent $e$, then
		\[
		|H^1_{ur}(k,A)| \ll_{|A|,\epsilon} |\Cl_F[e]| \cdot |\disc(F/\Q)|^{\epsilon}.
		\]
		\item[(iii)] If there is an injective homomorphism of $G_k$ modules $A' \ra{\rm Ind}_{F}^k(A) := \Z[G_k]\otimes_{\Z[G_F]} A$, then
		\[
		|H^1_{ur}(k,A')| \ll_{|A|,[F:k]} |\Hom(\Cl_F,A)|.
		\] (Here we only use $A$ as a constant $G_F$ module and not any $G_k$-module structure on $A$.)
	\end{itemize}
	\end{corollary}

	If $G$ is a nilpotent group in any representation, then for any abelian $T\normal G$ and any $\pi\colon G_k\to G$ the induced Galois module $T(\pi)$ is necessarily a nilpotent $G_k$-module. Corollary \ref{cor:inductive_H1ur_bounds}(i) is used in the proof Corollary \ref{cor:nilpotent}, showing that $|H^1_{ur}(k,T(\pi))|\ll |\disc(F/\Q)|^{\epsilon}$ has essentially no contribution to the bound in (\ref{eqn:abelian_H1ur_hypothesis}).

	The case of induced modules is particularly relevant to Corollary \ref{cor:cyclic_wreath}. The subgroup $T=C_n^m\normal C_n\wr B=G$ is given by the induced module $T = {\rm Ind}_1^B(C_n)$. This implies that $T(\pi) = {\rm Ind}_{F(\pi)}^{k}(C_n)$ as $G_k$-modules, so that Corollary \ref{cor:inductive_H1ur_bounds}(iii) together with Minkowski's bound allows us to take $|H^1_{ur}(k,T(\pi))|\ll |\disc(F(\pi)/\Q)|^{1/2+\epsilon}$ in (\ref{eqn:abelian_H1ur_hypothesis}), greatly improving our results.

\subsection{Method of Proof}\label{SS:method}

	Broadly speaking, we prove our main results
	by first considering for some normal subgroup $T\normal G$ with canonical quotient map $q\colon G\to G/T$ the fibers of the pushforward
	\[
	q_*\colon \Sur(G_k,G) \to \Sur(G_k,G/T)
	\]
	separately, then adding the results for each fiber together. Alberts constructed a bijection between the fiber $q_*^{-1}(\pi)$ with a certain set of crossed homomorphisms in \cite[Lemma 1.3]{alberts2021}, generalizing the Galois correspondence between $G$-extensions and (surjective) homomorphisms. We discuss this correspondence further in Subsection \ref{subsec:upperboundMB}. This naturally suggests a ``twisted'' version of Conjecture \ref{conj:number_field_counting}.

	\begin{conjecture}[Twisted Number Field Counting Conjecture]\label{conj:twisted_number_field_counting}
		Let $k$ be a number field, and $G$ a transitive permutation group of degree $n$, and $T\normal G$ a proper normal subgroup with canonical quotient map $q:G\to G/T$, and $\pi \in q_*\Sur(G_k,G)$. Then there exist positive constants $a,b,c > 0$ depending on $k$, $G$, $T$, and $\pi$ such that
		\[
		\#\{\psi\in q_*^{-1}(\pi) : |\disc_G(\psi)|\le X\} \sim c X^{1/a}(\log X)^{b-1}
		\]
		as $X\to \infty$.
	\end{conjecture}
	Alberts formulates this conjecture with explicit predictions for $a$ and $b$ in \cite[Conjecture 3.10]{alberts2021} in analogy to Malle's predictions, with the caveat that the prediction for the value of $b$ has similar issues to Malle's predictions. Alberts' counting function
	\[
		N(L/k,T\normal G;X) = \#\{F\in \mathcal{F}_{n,G}(k;X) : (\widetilde{F})^T = L \}
	\]
	is equal to a constant multiple of $\#\{\psi\in q_*^{-1}(\pi) : |\disc_G(\psi)|\le X\}$ via the Galois correspondence. The correctness of these predictions is referred to as the Twisted Malle Conjecture (with a corresponding weak and strong form). This conjecture was proven for $T$ abelian by Alberts and O'Dorney in \cite[Corollary 1.2]{alberts-odorney2021} as long as there exists at least one $G$-extension. We take this result and use it to evaluate the sum
	\begin{align}\label{eq:fibering_counting_function}
	\sum_{\pi\in \Sur(G_k,G/T)} \#\{ \psi\in q_*^{-1}(\pi) : |\disc_G(\psi)|\le X\}.
	\end{align}
	The asymptotic growth rate of each individual summand is given by 
	\cite[Corollary 1.2]{alberts-odorney2021}. We give a bound for the dependence of each summand on $\pi$, allowing us to take the sum of these growth rates. The set $q_*\Sur(G_k,G)$ appearing in Theorem \ref{thm:main_abelian_on_top} is precisely the subset of $\pi\in \Sur(G_k,G/T)$ for which the fiber $q_*^{-1}(\pi)$ is nonempty, that is
	\[
		q_*\Sur(G_k,G) = \{\pi\in \Sur(G_k,G/T) : q_*^{-1}(\pi) \ne \emptyset\}.
	\]

	A necessary, but not sufficient, criterion for this method to yield a proof of Conjecture \ref{conj:number_field_counting} in this paper is that one of the fibers $q_*^{-1}(\pi)$ contributes a \emph{positive proportion} of $G$-extensions. The Twisted Malle Conjecture, with invariants predicted by Alberts, suggests that there exists a fiber of positive density if and only if $G$ is concentrated in $T$.

	We expect our method of proof to apply to concentrated groups in general: if one knew enough about counting ``twisted'' $T$-extensions along the lines of Conjecture \ref{conj:twisted_number_field_counting} for some subgroup $T$ with enough control of how the error depends on the action induced by $\pi$, 
	then our method could be applied to prove an analogous result to Theorem \ref{thm:main_abelian_on_top} with $T$ being such a group. 
	Theorem \ref{thm:main_abelian_on_top} results from our extensive understanding of counting abelian extensions.
	The first step in the proof of Theorem \ref{thm:main_S3_wreath} is to prove new cases of Conjecture \ref{conj:twisted_number_field_counting} for $T = S_3^m \normal S_3\wr B$ (see Theorem \ref{thm:twisted_for_S3r_wreath}). We are able to do this because wreath products are the ``generic imprimitive structure'' for a group and because we can count $S_3$-extensions very well.
	In this case, the twisted counting function $\#\{\psi\in q_*^{-1}(\pi) : |\disc_{S_3\wr B}(\psi)| \le X\}$ can be related to an untwisted counting function $\#\Sur(G_F,S_3;X)$ for an extension $F/k$ determined by $\pi$. Since we already know Conjecture \ref{conj:number_field_counting} holds for $S_3$, we can use this relation to count the size of the fibers.

	We make this method explicit in Section \ref{sec:analysis}, detailing exactly what kind of information we need about Conjecture \ref{conj:twisted_number_field_counting} in order to prove new cases of Conjecture \ref{conj:number_field_counting}. We expect that, with any results proving new cases of Conjecture \ref{conj:twisted_number_field_counting} with the dependence on $\pi$ made explicit, our methods will convert these to proofs of Conjecture \ref{conj:number_field_counting} for new concentrated groups.

	\begin{remark}
		The methods in this paper are readily generalizable. There is potential for our methods to be applied in the following more general situations.
		\begin{itemize}
			\item The upper bounds in Theorem \ref{thm:main_S3_wreath}(ii) and Theorem \ref{thm:main_abelian_on_top}(ii) currently depend on $k$. If the dependence of the inputs on $k$ is made explicit, we expect that the dependence of the result on $k$ can be made explicit as well.
			\item It would be interesting to apply our methods to ``concentrated normal subgroups'' $T\normal G$ to prove new cases of Conjecture \ref{conj:twisted_number_field_counting}. 
			We call the normal subgroup $T$ concentrated if the minimal index elements in $T$, namely $\{ t\in T : \ind_G(t) = a(T)\}$, generate a proper subgroup of $T$. This is done by partitioning $q_*^{-1}(\pi)$ over the second quotient map $T\to T/N$ for some abelian normal subgroup $N\normal G$ contained in $T$ that contains all the elements of minimal index. (Note that, if $G$ were itself concentrated in $T$, then we could already apply Theorem \ref{thm:main_abelian_on_top} to $G$ with the proper normal subgroup $N$).
			\item It would be interesting to generalize Theorem \ref{thm:main_abelian_on_top} to other admissible orderings. Indeed, Alberts--O'Dorney work at this level of generality in \cite{alberts-odorney2021}.
		\end{itemize}
	   We opt to prove one extension of our methods in Theorem \ref{thm:cute} that involves alternate orderings as a demonstration of the generality and power of this technique. This result will show that, for any finite group with a nontrivial abelian normal subgroup, there exists some ordering for which we can give the asymptotic number of $G$-extensions.
	\end{remark}

\subsection{History of Number Field Counting Results}\label{subsec:history}

	Conjecture \ref{conj:number_field_counting} is known for several infinite families.

	We present the Table \ref{table:history} containing the previously known cases of Conjecture \ref{conj:number_field_counting}, separated according to whether the groups are concentrated or not.
	A computation in \verb^Magma^ shows, out of a total of $40{,}238$ transitive groups of degree $\le 31$,  that $39{,}770$ are concentrated.
	So it is relatively common for a transitive permutation group to be concentrated.
	These counting results are over an arbitrary number field unless otherwise specified. For each of the groups listed in this table, the value for $a$ in Conjecture \ref{conj:number_field_counting} is known to agree with Malle's prediction.

	\begin{table}[!htbp]\label{table:history}
	\begin{tabular}{|l|l|}
		\hline
		Group(s) & Reference\\\hline\hline
		\multicolumn{2}{|c|}{concentrated}\\\hline
		Abelian $\not\cong C_p^r$ & \cite{wright1989}\\
		
		$D_4\subset S_4$ & \cite{cohen-diaz-y-diaz-olivier2002}\\
		
		Generalized quaternion groups $Q_{4m}$ & \cite[Satz 7.6]{klunersHab2005}\\
		{}\hspace{1.5cm}for $m=2^\ell$ in degree $4m$, &{}\\

		12T5 & over $k=\Q$ \cite[Satz 7.7]{klunersHab2005}\\

		$C_3\wr C_2$ & over $k=\Q$ \cite{kluners2005}\\
		
		$C_2\wr H$ & when $\#\mathcal{F}_{m,k}(H;X) \ll X^{1+\epsilon}$ \cite{kluners2012}\\
		
		$S_n\times A$, $n\le 5$, $A=\text{abelian}$ & \cite{jwang2021,masri_thorne_tsai_wang2020} over $k=\Q$\\
		&(over arbitrary $k$ if $2\nmid |A|$ when $n=3$\\
		&or $\gcd(|A|,n!) = 1$ if $n\in \{4,5\}$)\\
		
		$S_n\times B$, $n\le 5$, $B=\text{nilpotent}$ in degree $|B|$,& \cite{mishra-ray2024}\\
		{}\hspace{1.5cm}$2\nmid |B|$ if $n=3$, &{}\\
		{}\hspace{1.5cm}${\rm gcd}(|B|,n!) = 1$ if $n\in \{4,5\}$ &{}\\

		$G\subsetneq C_\ell \wr C_\ell$ imprimitive in degree $\ell^2$ & \cite{fouvry-koymans2021malle,kluners-wang2023idelic}\\

		$G\subset S_{|G|}$ nilpotent,& \cite{koymans_pagano2021}\\
		{}\hspace{1.5cm}$\{g: \ind(g) = a(G)\} \subseteq Z(G)$ &{}\\
		
		\hline
		\multicolumn{2}{|c|}{non-concentrated}\\\hline
		$C_p^r$ & \cite{wright1989}\\
		
		$S_n$, $n\le 5$ & \cite{DW88,bhargava-shankar-wang2015}\\
		
		$S_3\subset S_6$ & \cite{bhargava-wood2007,belabas2010discriminants}\\
		\hline
	\end{tabular}
	\caption{Previously known cases of Conjecture \ref{conj:number_field_counting}}
	\end{table}

	Our main results expand the list for concentrated groups many times over, to the point that it is no longer feasible to give an exhaustive list of the types of such groups. In particular, Theorem \ref{thm:main_abelian_on_top} subsumes many previously known results for concentrated groups. We also expand several of these families:
	\begin{itemize}
		\item Kl\"uners' results for $C_2\wr H$ are expanded to include any $H$ for which $\#\mathcal{F}_{m,k}(H;X) \ll X^{3/2-\delta}$ for some $\delta > 0$, as well as analogous families $C_n\wr H$. See Corollary \ref{cor:cyclic_wreath}.
		\item Koymans--Pagano's results for nilpotent groups in the regular representation with $\{g: \ind(g) = a(G) \} \subseteq Z(G)$ are expanded to nilpotent groups in any representation with $\langle g: \ind(g) = a(G)\rangle$ abelian. See Corollary \ref{cor:nilpotent}.
		This family also includes the generalized quaternion groups $Q_{4m}$ for $m=2^\ell$ in the regular representation proven by Kl\"uners \cite{klunersHab2005} and 
		the imprimitive groups $G\subsetneq C_\ell\wr C_\ell$, special cases of which are proven by Fouvry--Koymans \cite{fouvry-koymans2021malle} while the general family is proven by Kl\"uners--Wang \cite{kluners-wang2023idelic}. We complete the latter family to include the wreath product $G=C_\ell\wr C_\ell$.
	\end{itemize}

	Aside from abelian groups, there are only finitely many non-concentrated groups for which Conjecture \ref{conj:number_field_counting} is known to hold. Counting $D_4$-extensions ordered by conductor can also naturally be interpreted as a non-concentrated result \cite{Altug_Shankar_Varma_Wilson_2017}. 

	Conjecture \ref{conj:twisted_number_field_counting} was previously known only in one case, with the values for $a$ and $b$ as predicted by Alberts: $T(\pi)$ for $T\normal G\subseteq S_n$ an abelian normal subgroup and $\pi\in q_*\Sur(G_k,G)$ \cite{alberts-odorney2021}. We prove new nonabelian cases of Conjecture \ref{conj:twisted_number_field_counting} in Theorem \ref{thm:twisted_for_S3r_wreath}.

	Also, as Theorems \ref{thm:main_S3_wreath} and \ref{thm:main_abelian_on_top} both implicitly use and produce upper bounds 
	of the form $\#\mathcal{F}_{n,k}(G;X) \ll_{k,n} X^\beta$, we summarize some of what is currently known for upper bounds. 
	Schmidt \cite{schmidt1995} proved that it suffices to take $\beta = \frac{n+2}{4}$ when $G$ is transitive of degree $n$, showing that such a bound always exists; it is now known we may take $\beta = 1.5 (\log n)^2$ \cite{lemke_oliver_thorne2022,LO-uniform} (though see also \cite{ellenberg_venkatesh2006} and \cite{couveignes} for earlier improvements depending only on the degree of $G$).  
	There are a number of other techniques that more substantially leverage the group structure of $G$ \cite{kluners-malle2004,dummit2018,mehta2020,alberts2020,kluners2022nilpotent,bhargava2024,LO-uniform,LO-Galois}, and by taking advantage of these results as inputs to our main theorems, one can produce a plethora of new examples for which Conjecture~\ref{conj:number_field_counting} holds.
	The known upper bounds we specifically leveraged in stating our corollaries are:
	\begin{itemize}
		\item If $G$ is nilpotent, then $\#\mathcal{F}_{n,k}(G;X) \ll_{k,n,\epsilon} X^{\frac{1}{a(G)} + \epsilon}$.  This follows from \cite{kluners-malle2004} if $G$ is in its regular representation, and from \cite{alberts2020} in general.  (See also \cite{Kluners2022}.) An analogous upper bound was proven for alternate orderings in \cite{alberts2020}, which includes the pushforward discriminant.
		\item If $G$ is in the regular representation and $|G|>4$, then $\#\mathcal{F}_{|G|,k}(G;X)\ll_{k,G,\epsilon} X^{3/8+\epsilon}$ \cite{ellenberg_venkatesh2006}, and we also have $\#\mathcal{F}_{|G|,k}(G;X) \ll_{k,G,\epsilon} X^{\frac{c_0}{\sqrt{|G|}} + \epsilon}$ where $c_0 = \frac{863441}{2880 \sqrt{9690}} \approx 3.045$ \cite{LO-Galois}.
		\item If $G$ is a finite simple group of Lie type, say over $\mathbb{F}_q$ with rank $r$, let $G_0$ denote its minimal degree primitive permutation representation, say in degree $n$.  From \cite[Theorem~1.1]{LO-uniform}, we have $\#\mathcal{F}_{n,k}(G_0;X) \ll_{n,k} X^{C r}$ for some absolute constant $C>0$.  Recall that $\Stab_{G_0} 1 =: M_0$ is a maximal subgroup.  If $M < G$ is any maximal subgroup with $|M| < |M_0|$ and $G_1$ the corresponding primitive representation of $G$ on the cosets of $M$, then it follows from \cite[Lemma~6.3]{LO-uniform} that 
			\[
				\#\mathcal{F}_{[G:M],k}(G_1;X) \ll_{G,k} X^{\frac{Cr}{ |M_0| / |M|}}.
			\]
		On combining works of Liebeck and Saxl \cite{Liebeck,LiebeckSaxl} with explicit case work involving the parabolic and other geometric subgroups (aided, e.g., by \cite{BrayHoltRoneyDougal,KleidmanLiebeck}) that, unless $G$ is of type $\mathrm{PSU}_6(\mathbb{F}_q)$ and $M$ is the parabolic subgroup $P_3$, we see that we always have $|M_0| / |M| \gg q^{1/2}$.  
		It follows that $\#\mathcal{F}_{[G:M],k}(G_1;X) \ll_{G,k} X^{\frac{C^\prime r}{q^{1/2}}}$ for some absolute $C^\prime > 0$.  In particular, this exponent may be made arbitrarily small on choosing $q$ sufficiently large in terms of $r$, and every non-minimal primitive permutation representation of $G$ arises in this way.
	\end{itemize}
	In particular, in all three cases, these upper bounds can be made arbitrarily small as $G$ varies, which means we can produce infinitely many examples of groups satisfying Theorem \ref{thm:main_abelian_on_top}(i). 
	This is the main source of the scale of the infinite families in the introduction.

	Theorems \ref{thm:main_S3_wreath} and \ref{thm:main_abelian_on_top} also take average bounds on class group torsion as input, where $|H^1_{ur}(k,T(\pi))|$ in Theorem \ref{thm:main_abelian_on_top} can be bounded in terms of certain class group torsion. Minkowski's bound for the size of the class group immediately implies
	\begin{align*}
		|\Cl_K[\ell]| &\ll |\disc(K/\Q)|^{1/2+\epsilon}\\
		|H^1_{ur}(k,T(\pi))| &\ll |\disc(F/\Q)|^{d(\hat{T})/2+\epsilon},
	\end{align*}
	where $F$ is the field of definition for $T(\pi)$ and $d(\hat{T})$ is the minimum number of generators for $\hat{T}=\Hom(T,\Q/\Z)$ as a Galois module (see Lemma \ref{lem:H1_first_bound}). To our best knowledge, bounds for $|H^1_{ur}(k,T(\pi))|$ have not been directly studied in the literature. We prove results in Section \ref{sec:classgrp} bounding these, in part by relating them to class group torsion.

	There are some improvements to Minkowski's bound for $|\Cl_K[\ell]|$ in the literature.  The ones that we use most often are as follows.
	\begin{itemize}
		\item $|\Cl_K[\ell]|\ll |\disc(K/\Q)|^{\epsilon}$ for $K/\Q$ with $\Gal(\tilde{K}/\Q)$ an $\ell$-group \cite{Kluners2022}
		\item $|\Cl_K[2]|\ll |\disc(K/\Q)|^{0.2784...+\epsilon}$ for $[K:\Q]\le 4$ \cite{bhargava_shankar_taniguchi_thorne_tsimerman_zhao_2020}
		\item $|\Cl_K[2]|\ll_{[K:\Q]} |\disc(K/\Q)|^{1/2-1/2[K:\Q]+\epsilon}$ for $[K:\Q] > 4$ \cite{bhargava_shankar_taniguchi_thorne_tsimerman_zhao_2020}
	\end{itemize}
	See also \cite{Pierce2005, Pierce2006a, Helfgott2006, ellenberg_venkatesh2007, Wang2021a, Wang2020} for other improvements to Minkowski's bound for $|\Cl_K[\ell]|$.
	Strictly speaking, Theorems \ref{thm:main_S3_wreath} and \ref{thm:main_abelian_on_top} only require bounds for the average size of class group torsion.
	There are a few cases in which precise asymptotics for the average of $|\Cl_K[\ell]|$
	in families are known including $3$-torsion for quadratic extensions due to Davenport and Heilbronn \cite{Davenport1971} and Datskovsky and Wright \cite{DW88}, $2$-torsion for cubic extensions due to Bhargava \cite{Bhargava2005} (see also \cite{bhargava-shankar-wang2015}), and
	$3$-torsion  for extensions $L$ for which $\Gal(\tilde{L}/k)$ is a $2$-group containing a transposition due to Lemke Oliver, Wang, and Wood \cite{LOWW}.
	In many other cases, on average improvements to Minkowski's bound for 
	$|\Cl_K[\ell]|$ are known \cite{Soundararajan2000,Heath-Brown2017,Ellenberg2017, Pierce2020, Widmer2018, Frei2018,An2020,Frei2021, Thorner2022,LemkeOliver2023,Koymans2024}.  See \cite{pierce_turnage-butterbaugh_wood2021}
	for an overview of the conjectures on bounding class group torsion pointwise and on average, and the recent paper of Lemke Oliver and Smith \cite{LemkeOliver2024} for the  state-of-the art theorems giving average improvements to Minkowski's bound for 
	$|\Cl_K[\ell]|$.

\subsection{Layout of the Paper}

	We begin with Section \ref{sec:analysis}, where we give the explicit form of our method. We will prove Theorem \ref{thm:main_pointwise} in this section, which states explicitly what we need to know in order to add the fibers together to prove an asymptotic growth rate for $\#\mathcal{F}_{n,k}(G;X)$. We state this result in a general language, so that it can be applied new cases of Conjecture \ref{conj:number_field_counting} in the future.

	We prove Theorem \ref{thm:main_S3_wreath} in Section \ref{sec:S3_wreath}. The proof is comparatively short, taking advantage of the wealth of results concerning $S_3$-extensions to quickly check the hypotheses of Theorem \ref{thm:main_pointwise}.

	Next, we prove important results concerning the ingredients of Theorem \ref{thm:main_abelian_on_top}. Section \ref{sec:classgrp} proves Corollary \ref{cor:inductive_H1ur_bounds}, along with some other upper bounds for the cohomology groups $|H^1_{ur}(k,M)|$. Section \ref{sec:pushforward_invariants} develops the notion of a pushforward discriminants, describes the relationship with imprimitive extensions, and proves that Corollary \ref{cor:main_abelian_on_top_imprimitive} follows from Theorem \ref{thm:main_abelian_on_top}.

	We will then prove Theorem \ref{thm:main_abelian_on_top} in Section \ref{sec:abelian_on_top}, building on work of Alberts--O'Dorney \cite{alberts-odorney2021} to check the hypotheses of Theorem \ref{thm:main_pointwise}.

	Section \ref{sec:examples} contains a list of examples of groups for which we prove Conjecture \ref{conj:number_field_counting}, including proofs of the corollaries listed in the introduction. This section can be read independent of the other sections in this paper, so that interested readers can jump straight in to applying our results to check cases of Conjecture \ref{conj:number_field_counting}. In addition to the corollaries listed in the introduction, we include a subsection describing the \verb^Magma^ code used to produce Corollary \ref{cor:compute} and a subsection summarizing the current best known results towards Conjecture \ref{conj:number_field_counting} for all transitive groups of degree $6$.

	Sections \ref{sec:conditional_examples} and \ref{sec:cute} present evidence that our method has the potential to apply to concentrated groups in greater generality. Section \ref{sec:conditional_examples} discusses what one needs to know about the concentrated group $G$ to apply the methods of this paper, and in particular relates these ingredients to existing conjectures in arithmetic statistics. Meanwhile, Section \ref{sec:cute} goes in a different direction. We use our method to prove Theorem \ref{thm:cute}, which states that for any group concentrated in an abelian normal subgroup, there exists some admissible ordering $\inv$ for which our methods gives the asymptotic growth rate of $\#\Sur_{\inv}(G_k,G;X)$. In particular, this shows that every solvable group has an ordering for which we can give the corresponding asymptotic growth rate.

\subsection{Notation} 
	The following is a list of conventions and notations we take throughout.
		
	\allowdisplaybreaks
	\begin{align*}
	\text{``a }&\text{permutation group of degree $m$'' implies that $m$ is finite}\\
	k & \text{ will always denote a number field, and is the base field for our extensions.}\\
	G_F &= \text{ the absolute Galois group of a field }F\\
	\widetilde{L} &=\text{the Galois closure of a field }L\text{ over the base field }k\\
	G & \text{ will always denote a transitive permutation group,}\\
	 &\text{and will be of degree $n$ unless otherwise specified}\\
	\text{Permutation }&\text{groups are \emph{isomorphic} if they are of the same degree $m$ and}
	 \\
	 \text{are }&\text{conjugate as subgroups of $S_m$}
	 \\
	\Stab_G(1)  & \text{ the subgroup of $G$ fixing $1$} \\
	G \wr H &=\text{the wreath product in the wreath representation, i.e.}\\
		&\phantom{=}\text{ with stabilizer }(\Stab_G(1) \times G^{m-1}) \times \Stab_H(1)\text{ when $H$ has degree $m$}\\
	\disc(F/k) &\text{ the relative discriminant ideal}\\
	|\mathfrak{a}| &=\text{ the norm of the ideal }\mathfrak{a}\text{ down to }\Q\\
	\mathcal{F}_{n,k}(G;X) &=\{L/k : [L:k] = n,\ \Gal(\widetilde{L}/k)\cong G,\ |\disc(L/k)|\le X\}\\
	\disc(\pi) &\text{ the relative discriminant of the field fixed by }\pi^{-1}(\Stab_G(1))\\
	\disc_G(\pi) &\text{ same as above, when $G$ needs to be specified}\\
	\Sur(G_k,G) &= \{\pi:G_K\to G\text{ surjective continuous homomorphism}\}\\
	\Sur(G_k,G;X) &= \{\pi\in \Sur(G_K,G) : |\disc_G(\pi)|\le X\}\\
	\ind(g) &= n - \#\{\text{orbits of }g\}\text{ for }g\in S_n\\
	\ind_n(g) &\text{ same as above, when $n$ needs to be specified}\\
	\ind_G(g) &\text{ same as above, when $G$ needs to be specified}\\
	a(U)&= \min_{g\in U-\{1\}}\ind(g)\text{ for any subset }U\subseteq S_n\\
		&\hphantom{=} a(G) \text{ is Malle's predicted value for }a\text{ in Conjecture \ref{conj:number_field_counting}}\\
		&\hphantom{=} a(T) \text{ is Alberts's predicted value for }a\text{ in Conjecture \ref{conj:twisted_number_field_counting}}\\
	\chi:G_k\to \hat{\Z}^{\times}&\text{ is the cyclotomic character, given by }\Gal(k\Q^{\rm ab}/k)\subseteq \hat{\Z}^{\times}\\
	b(k,G)&= \text{the number of orbits of conjugacy classes in }G\\
		&\hphantom{=}\text{ with respect to the Galois action }x.g=g^{\chi(x)}\\
		&\hphantom{=}b(k,G)\text{ is Malle's predicted value for }b\text{ in Conjecture \ref{conj:number_field_counting}}\\
	c(k,G)&=\text{the value for }c\text{ in Conjecture \ref{conj:number_field_counting} whenever it is known to hold}\\
	T(\pi)&=\text{the subgroup } T\normal G\text{ with Galois action }x.t=\pi(x)t\pi(x)^{-1}\\
		&\hphantom{=}\text{induced by }\pi:G_k\to G.\\
	T(\pi)^*&=\Hom(T(\pi),\mu)\text{ as a Galois module, }\mu\text{ is the group of roots of unity}\\
	&\phantom{=}\text{called the Tate dual of }T(\pi)\\
	q_*&=\text{ the pushforward along the quotient }q:G\to G/T\\
	q_*\Sur(G_k,G) &= \{\pi\in \Sur(G_k,G/T) : \pi = q\circ \widetilde{\pi}\text{ for some }\widetilde{\pi}\in\Sur(G_k,G)\}\\
	q_*\Sur(G_k,G;X) &= \{\pi\in \Sur(G_k,G/T) : \pi = q\circ \widetilde{\pi}\text{ for some }\widetilde{\pi}\in\Sur(G_k,G;X)\}\\
	q_*\disc &=\text{ the pushforward discriminant, see \eqref{eq:pushforward_discriminant}}\\
	\Sur_{\inv}(G_k,G;X) &= \{\pi\in \Sur(G_k,G) : |\inv(\pi)|\le X\}\\
	b(k,T(\pi))&= \text{the number of orbits of conjugacy classes in }T\\
		&\hphantom{=}\text{ with respect to the Galois action }x.t=\pi(x)t^{\chi(x)^{-1}}\pi(x)^{-1}\\
	c(k,T(\pi))&=\text{the value for }c\text{ in Conjecture \ref{conj:twisted_number_field_counting} whenever it is known to hold}
	\end{align*}

	Below are the conventions we use for asymptotic notation. Any implied constants are always allowed to depend on $k$ and $G$, unless otherwise specified.
	\begin{align*}
	f(X) \sim g(X) & \text{ asymptotic, i.e. }\lim_{X\to\infty} f(X)/g(X) = 1\\
	f(X) \ll g(X) & \text{ there exists a constant }C\text{ s.t. }|f(X)| \le Cg(X)\\
	f(X) \ll_P g(X) & \text{ same as above, where }C\text{ depends only on the parameters }P\\
	f(X) = O(g(X)) & \text{ same as }f(X) \ll g(X)\\
	f(X) = O_P(g(X)) & \text{ same as }f(X) \ll_P g(X)\\
	f(X) \asymp g(X) & \text{ same as }g(X) \ll f(X) \ll g(X)\\
	f(X) \asymp_P g(X) & \text{ same as }g(X) \ll_P f(X) \ll_P g(X)
	\end{align*}

	\allowdisplaybreaks[0]

\section*{Acknowledgements}
    Alberts was partially supported by an AMS-Simons travel grant, Lemke Oliver was supported by NSF grant DMS-2200760, Wang was partially supported by a Foerster-Bernstein Fellowship at Duke University and NSF grant DMS-2201346, and Wood was partially supported by a Packard Fellowship for Science and Engineering, an NSF Waterman award DMS-2140043, a MacArthur Fellowship, and the Radcliffe Institute for Advanced Study at Harvard University.

    The authors would like to thank Alexander Smith for helpful conversations and references. The authors would also like to thank Evan O'Dorney, Daniel Loughran, and Frank Thorne for helpful feedback.

\section{The Inductive Framework}\label{sec:analysis}

	In this section, we provide the general analytic framework that we use to piece the fiber-wise counts back together.  This is provided by the following theorem

	\begin{theorem}\label{thm:main_pointwise}
		Let $G$ be finite permutation group with a normal subgroup $T \normal G$, and let $q\colon G \to G/T$ be the quotient map.  Let $k$ be a number field, and for any $\pi \in \Sur(G_k,G/T)$, let $q_*^{-1}(\pi) \subseteq \Sur(G_k,G)$ be the fiber over $\pi$.

		Assume there exist real numbers $a>0$ and $b \geq 1$ such that the following three conditions are satisfied:
			\begin{enumerate}
				\item (``Precise counting of the fibers'') For each $\pi \in \Sur(G_k, G/T)$, there is some constant $c(\pi) \geq 0$ so that
					\[
						\#\{ \psi \in q_*^{-1}(\pi) : |\disc(\psi)| \leq X\}
							= (c(\pi) + o(1)) X^{1/a} (\log X)^{b-1}
					\]
				as $X \to \infty$.

				\item (``Uniform upper bounds on the fibers'') For each $\pi \in \Sur(G_k, G/T)$, there is a constant $f(\pi) \geq 0$ so that for every $X \geq 2$, we have
					\[
						 \#\{ \psi \in q_*^{-1}(\pi) : |\disc(\psi)| \leq X\}
							\leq f(\pi) X^{1/a} (\log X)^{b-1}.
					\]

				\item (``Criterion for convergence'') The series 
					\[
						\sum_{\pi \in \Sur(G_k,G/T)} f(\pi)
					\]
				converges, where $f(\pi)$ is as above.
			\end{enumerate}
		Then 
			\[
				\#\Sur(G_k,G;X)
					= (c + o(1) ) X^{1/a} (\log X)^{b-1},
			\]
		where $c$ is given by the convergent series
			$
				c := \sum_{\pi \in \Sur(G_k,G/T)} c(\pi).
			$
	\end{theorem}
	\begin{remark} \rm
		Note that we have allowed $c(\pi)$ and $f(\pi)$ to be $0$.  This is convenient for two reasons.  First, there may be $\pi \in \Sur(G_k, G/T)$ for which the fiber $q_*^{-1}(\pi)$ is empty, in which case we may take $f(\pi) = 0$.  We may equivalently restrict our attention in (1) and (2) to those $\pi$ in the subset $q_* \Sur(G_k, G) \subseteq \Sur(G_k, G/T)$, which we will often do in what follows.
		
		Second, by allowing $c(\pi) = 0$, we are not demanding that every fiber, or indeed that \emph{any} fiber, has positive density.  If every $c(\pi)=0$, then the conclusion is that $\#\Sur(G_k,G;X) = o(X^{1/a}(\log X)^{b-1})$, which is not an asymptotic formula but is a potentially nontrivial upper bound.
		This means in particular that Theorem~\ref{thm:main_pointwise} can still meaningfully apply even to non-concentrated groups.
	\end{remark}

	\begin{proof}
		Since the fibers $q_*^{-1}(\pi)$ for distinct $\pi \in \Sur(G_k,G/T)$ are disjoint, it follows that for any $X \geq 1$, we may write
			\begin{equation} \label{eqn:fiber-wise-decomposition}
				\#\Sur(G_k,G;X)
					= \sum_{\pi \in \Sur(G_k,G/T)} \#\{\psi \in q_*^{-1}(\pi) : |\disc(\psi)| \leq X\}.
			\end{equation}
		Now, let $Y \geq 1$ be arbitrary.  It follows from our assumptions that there exists a finite subset $\Pi \subset \Sur(G_k,G/T)$ and some $X_0 \geq 2$, both depending on $Y$, such that
			\begin{equation} \label{eqn:tail-of-series}
				\sum_{\substack{ \pi \in \Sur(G_k,G/T) \\ \pi \not\in \Pi}} f(\pi) 
					< \frac{1}{4Y}
			\end{equation}
		and, for every $\pi \in \Pi$ and every $X \geq X_0$, we have
			\[
				\left| \#\{ \psi \in q_*^{-1}(\pi) : |\disc(\psi)| \leq X\} - c(\pi) X^{1/a} (\log X)^{b-1}\right|
					< \frac{X^{1/a} (\log X)^{b-1}}{2Y}.
			\]
		Indeed, such a set $\Pi$ exists by the criterion for convergence, and such an $X_0$ exists by the precise counting of fibers for the finitely many $\pi \in \Pi$.

		Inserting this into \eqref{eqn:fiber-wise-decomposition} and appealing to the uniform upper bound on fibers, we readily find for any $X \geq X_0$ that
			\[
				\left| \#\Sur(G_k,G;X) - \sum_{\pi \in \Pi} c(\pi) X^{1/a} (\log X)^{b-1} \right| 
					< \frac{3X^{1/a} (\log X)^{b-1}}{4Y}.
			\]
		We next observe for any $\pi \in \Sur(G_k,G/T)$ that we evidently must have $c(\pi) \leq f(\pi)$.  From this, we deduce both that the series
			\[
				\sum_{\pi \in \Sur(G_k,G/T)} c(\pi)
			\]
		converges as a consequence of the criterion for convergence, and that
			\[
				\sum_{\substack{ \pi \in \Sur(G_k,G/T) \\ \pi \not\in \Pi}} c(\pi)
					< \frac{1}{4Y}
			\]
		on comparison with \eqref{eqn:tail-of-series}.  Pulling everything together, we find for every $X \geq X_0$,
			\[
				\left| \#\Sur(G_k,G;X) - \sum_{\pi \in \Sur(G_k,G/T)} c(\pi) X^{1/a} (\log X)^{b-1} \right|
					< \frac{ X^{1/a} (\log X)^{b-1} }{Y}.
			\]
		The result follows on taking $Y \to \infty$.
	\end{proof}

	Theorem~\ref{thm:main_pointwise} is stated as a result counting elements of $\Sur(G_k,G)$, not number fields.  Even though these problems are equivalent, this choice is deliberate and justified by the way we treat groups concentrated in an abelian normal subgroup later in the paper.  However, we close this section with a lemma that makes precise the translation between these two perspectives so that Theorem~\ref{thm:main_pointwise} may still be properly regarded as a number field counting result.

	\begin{lemma}\label{lem:sur-to-fields}
		Let $G$ be a transitive permutation group of degree $n$ and let $k$ be a number field.  Given an element $\pi \in \Sur(G_k, G)$, we may associate to $\pi$ the subfield $F$ of $\overline{k}$ fixed by $\pi^{-1}(\mathrm{Stab}_G 1)$. This subfield is a degree $n$ extension of $k$ with Galois closure group $G$, and $\disc(F/k) = \disc_G \pi$.

		Conversely, given such a field $F$, there are exactly $[N_{S_n}(G) \cap N_{S_n}(\Stab_G 1) : C_{S_n}(G) \cap N_{S_n}(\Stab_G 1)]$ elements of $\Sur(G_k,G)$ giving rise to $F$ in this manner, where for any subgroup $H \leq S_n$, we let $N_{S_n}(H)$ and $C_{S_n}(H)$ denote the normalizer and centralizer subgroups of $H$ in $S_n$.
	\end{lemma}
	\begin{proof}
		For any transitive permutation group $G$ of degree $n$, the subgroup $\Stab_G 1$ has index $n$, the conjugates of $\Stab_G 1$ are $\Stab_G 1, \dots, \Stab_G n$, whose total intersection is trivial, and the action of $G$ on the cosets of $\Stab_G 1$ is permutation isomorphic to $G$.  This implies the first claim.

		For the second, observe that if $\pi$ and $\pi^\prime$ are both associated with $F$, then there must be some $\phi \in \mathrm{Aut}(G)$ with $\phi(\Stab_G 1) = \Stab_G 1$ such that $\pi^\prime = \phi \circ \pi$.  This process may be reversed, so we may equivalently count such automorphisms $\phi$.  Any such automorphism $\phi$ must permute the cosets of $\Stab_G 1$, so arises via conjugation from an element of $S_n$.  In fact, because $G$ is a permutation group, this conjugation must be from an element of the normalizer $N_{S_n}(G)$.
		Moreover, since we must have that $\phi(\Stab_G 1) = \Stab_G 1$, it must arise from the intersection $N_{S_n}(G) \cap N_{S_n}(\Stab_G 1)$.  The kernel of the restriction to this subgroup is $C_{S_n}(G) \cap N_{S_n}(\Stab_G 1)$, and the lemma follows.
	\end{proof}

\section{Wreath Products by $S_3$}\label{sec:S3_wreath}
	In this section, we prove Theorem \ref{thm:main_S3_wreath} for groups of the form $G = S_3\wr B = (S_3^m)\rtimes B$ where $B\subset S_m$ is a transitive permutation group of degree $m$.  We begin by fixing a useful convention.  As both $G$ and $B$ are permutation groups, we may assume that the labels of the element `$1$' in $\{1,\dots,m\}$ and $\{1,\dots,3m\}$ are compatible in the sense that $\Stab_G 1 \leq q^{-1}(\Stab_B 1)$, where $q\colon G \to G/T \cong B$ is the quotient map composed with a fixed isomorphism $G/T \to B$, where $T = S_3^m$.  In this context, the minimal index elements of $G$ are the transpositions in $G$, so $G$ is concentrated in $T$.  We begin by establishing the precise counting of fibers required by Theorem~\ref{thm:main_pointwise}, that is, we establish Conjecture~\ref{conj:twisted_number_field_counting} for such groups.

	\begin{theorem}\label{thm:twisted_for_S3r_wreath}
		Let $B$ be a permutation group of degree $m$ and let $G = S_3\wr B$.  Let $q\colon G\to G/T \cong B$ be the quotient map, where $T = S_3^m$, and assume that $\Stab_G 1 \leq q^{-1}(\Stab_B 1)$.  For each $\pi\in \Sur(G_k,B)$, there exists a positive constant $c(\pi) > 0$ such that
		\[
			\#\{ \psi \in q_*^{-1}(\pi) : |\disc_G(\psi)| \leq X\} \sim c(\pi)X.
		\]
	\end{theorem}

	\begin{proof}
		Let $\pi\in \Sur(G_k,B)$ be fixed, and let $F/k$ be the subfield of $\overline{k}$ fixed by $\pi^{-1}(\Stab_B 1)$.  Similarly, given $\psi \in q_*^{-1}(\pi)$, let $E/k$ be the subfield of $\overline{k}$ fixed by $\psi^{-1}(\Stab_G 1)$.  By our assumption that $\Stab_G 1 \leq q^{-1}( \Stab_B 1)$, $E$ is a cubic extension of $F$, necessarily with Galois closure group $S_3$ over $F$.  Moreover, by the conductor-discriminant formula, we have that $\disc(E/k) = \mathrm{Nm}_{F/k} \disc(E/F) \cdot \disc_B(\pi)^3$, so $E \in \mathcal{F}_{3,F}(S_3; x)$, where we have set $x := X / |\disc_B(\pi)|^3$ for convenience.
		
		We next observe that $\Aut(G)$ acts transitively and freely on the set $\{ \psi \in \Sur(G_k, G) : \ker \psi = G_{\widetilde{E}}\}$, where $G_{\widetilde{E}}$ is the absolute Galois group of the Galois closure $\widetilde{E}$ of $E$ over $k$.  Appealing to Lemma~\ref{lem:sur-to-fields}, we therefore find that
			\[
				\#\{ \psi \in q_*^{-1}(\pi) : |\disc_G(\psi)| \leq X\}
					= c_G \cdot \#\{ E \in \mathcal{F}_{3,F}(S_3;x) : \Gal(\widetilde{E}/k) \cong G\},
			\]
		where
			\begin{equation} \label{eqn:c_G-def}
				c_G := \frac{[N_{S_{3m}}(G) \cap N_{S_{3m}}(\Stab_G 1) : C_{S_{3m}}(G) \cap N_{S_{3m}}(\Stab_G 1 ) ]}{[N_{S_m}(B) \cap N_{S_{m}}(\Stab_B 1) : C_{S_{m}}(B) \cap N_{S_{m}}(\Stab_B 1)]}.
			\end{equation}
		It follows from work of Datskovsky and Wright \cite{DW88} that there is some constant $c_F > 0$ such that
			\[
				\#\mathcal{F}_{3,F}(S_3;x)
					\sim c_F x = c_F \frac{X}{|\disc_B(\pi)|^3}
			\]
		as $X \to \infty$.  We claim that the same asymptotic holds for the subset of $\mathcal{F}_{3,F}(S_3;X)$ whose Galois closure over $k$ has Galois group $G$, so that the theorem holds with $c(\pi) = c_G c_F / |\disc_B(\pi)|^3$.

		To prove this claim, we exploit the fact that Datskovsky and Wright also prove an asymptotic for the number of fields $E \in \mathcal{F}_{3,F}(S_3;x)$ subject to finitely many local conditions.  Let $\fp$ be a prime of $k$ that splits completely in $F$, say as $\mathfrak{P}_1 \dots \mathfrak{P}_m$.  If $E \in \mathcal{F}_{3,F}(S_3;x)$ is such that there is some $1 \leq i \leq m$ for which the \'etale algebra $E \otimes_F F_{\mathfrak{P}_i}$ is the direct sum of $F_{\mathfrak{P}_i}$ with a quadratic extension, and so that $E \otimes_F F_{\mathfrak{P}_j}$ is totally split for each $j \ne i$, then we must have that $\Gal(\widetilde{E}/k) \cong G$.  In particular, any $E \in \mathcal{F}_{3,F}(S_3;x)$ whose Galois closure group is not $G$ cannot satisfy this local condition at any prime that splits completely in $F$.

		Let $S$ be a finite set of primes $\fp$ of $k$ that split completely in $F$.  For each $\fp \in S$, let $\Sigma_\fp$ be the subset of $\mathcal{F}_{3,F}(S_3;\infty)$ consisting of those $E$ that do not satisfy the local condition described above.  By \cite{DW88}, there are positive constants $\delta_\fp$, bounded uniformly away from $1$ for $\fp$ sufficiently large, so that
			\[
				\#\{ E \in \mathcal{F}_{3,F}(S_3;x) : E \in \Sigma_\fp \text{ for all } \fp \in S\}
					\sim c_F \cdot \prod_{\fp \in S} \delta_\fp \cdot x
			\]
		as $X \to \infty$.  From the discussion above, the set on the left-hand side contains all extensions $E \in \mathcal{F}_{3,F}(S_3;x)$ whose Galois closure group is not $G$, and the right-hand side may be made arbitrarily small by choosing $S$ sufficiently large.  This gives the claim and the theorem.
	\end{proof}

	Next, we need a uniform upper bound on the sizes of the fibers.  This is essentially provided by a result of Lemke Oliver, Wang, and Wood.
	\begin{lemma}\label{lem:LOWW_uniform}
		Let $G = S_3\wr B$ and $T=S_3^m$ with quotient map $q\colon G\to G/T \cong B$. Assume that $\Stab_G 1 \leq q^{-1}(\Stab_B 1)$.  For each $\pi\in \Sur(G_k,B)$, every $X \geq 1$, and every $\epsilon > 0$, we have
			\[
				\#\{ \psi \in q_*^{-1}(\pi) : |\disc_G \psi| \leq X\}
					= O_{[k:\mathbb{Q}],m,\epsilon}\left (\frac{|\disc(k)|^{m+\epsilon} |\Cl_F[2]|^{2/3}}{|\disc_B (\pi)|^{2-\epsilon}} X\right)
			\]
		where $F$ is the subfield of $\overline{k}$ fixed by $\pi^{-1}(\Stab_B 1)$.
	\end{lemma} 
	\begin{proof}
		As in the proof of Theorem~\ref{thm:twisted_for_S3r_wreath}, we have that
			\[
				\#\{ \psi \in q_*^{-1}(\pi) : |\disc_G \psi| \leq X\}
					\leq c_G \cdot \#\mathcal{F}_{3,F}(S_3; X / |\disc_B (\pi)|^3 ),
			\]
		where $c_G$ is as in \eqref{eqn:c_G-def}.  On noting that $\disc_B(\pi) = \disc(F/k)$ and $c_G = O_m(1)$, the result then follows from \cite[Corollary 3.2]{LOWW}.
	\end{proof}

	We may now prove Theorem~\ref{thm:main_S3_wreath}.

	\begin{proof}[Proof of Theorem \ref{thm:main_S3_wreath}]
		We use Theorem~\ref{thm:main_pointwise} in concert with Lemma~\ref{lem:sur-to-fields}.  The precise counting of fibers is provided by Theorem~\ref{thm:twisted_for_S3r_wreath}, while the uniform upper bounds on fibers are provided by Lemma~\ref{lem:LOWW_uniform}.  It therefore remains to check the criterion for convergence.  We first note that, with $f(\pi)$ determined by Lemma~\ref{lem:LOWW_uniform}, 
			\begin{align*}
				\sum_{\pi \in \Sur(G_k,B)} f(\pi)
					&\ll_{k,m,\epsilon} \sum_{\pi \in \Sur(G_k,B)} \frac{|\Cl_F[2]|^{2/3}}{|\disc_B(\pi)|^{2-\epsilon}} \\
					&\ll_m \sum_{F \in \mathcal{F}_{m,k}(B;\infty)} \frac{|\Cl_F[2]|^{2/3}}{|\disc(F/k)|^{2-\epsilon}},
			\end{align*}
		where we have invoked Lemma~\ref{lem:sur-to-fields} in the second line.  We now recall that in the hypotheses of Theorem~\ref{thm:main_S3_wreath}, we have assumed there is some $\theta \geq 0$ so that
			\[
				\sum_{F \in \mathcal{F}_{m,k}(B;X)} |\Cl_F[2]|^{2/3}
					\ll_{m,k} X^\theta
			\]
		for every $X \geq 1$.  If $\theta< 2$, then the criterion for convergence is satisfied by partial summation, and this yields the first claim.
		If $\theta \geq 2$, then the criterion for convergence is not satisfied, and we instead find on using Lemma~\ref{lem:LOWW_uniform} directly that
			\[
				\#\mathcal{F}_{3m,k}(G;X)
					\ll_{k,m,\epsilon} \sum_{F \in \mathcal{F}_{m,k}(B;X^{1/3})} \frac{|\Cl_F[2]|^{2/3}}{|\disc(F/k)|^{2-\epsilon}}X
					\ll_{k,m,\epsilon} X^{\frac{\theta+1}{3} + \epsilon}.
			\]
		This completes the proof of the second claim, and thus the theorem.
	\end{proof}

\section{Inductive Bounds for $H^1_{ur}$}\label{sec:classgrp}
	In this section, we prove a number of bounds for $H^1_{ur}(k,A)$ where $A$ is some $G_k$-module. 
	For this section, we will use the usual additive notation for the group operation in $A$.
	The group $H^1_{ur}(k,A)$ is closely related to class group torsion.  This is clear when $A$ carries the trivial action, as $H^1_{ur}(k,A) = \Hom(\Cl_k,A)$ in this case.  For arbitrary $G_k$-modules, one key way to understand $H^1_{ur}(k,A)$ is through the restriction map
	$$
		H^1_{ur}(k,A) \ra  H^1_{ur}(F,A)=\Hom(\Cl_F,A),
	$$
	where $F$ is the field of definition for $A$ as a Galois module.
	This gives the bound of Lemma~\ref{lem:H1_first_bound} below that was stated in the introduction.

	For an abelian group $A$, let $\hat{A}:=\Hom(A,\Q/\Z)$ denote the Pontryagin dual.
	If $A$ is a $G$-module for some group $G$, then $\hat{A}$ is naturally a $G$-module via $(g \phi)( a)=\phi(g^{-1}a)$ for $g\in G$ and $a\in A$ and $\phi\in \hat{A}.$

	\begin{lemma}\label{lem:H1_first_bound}
		Let $k$ be a number field, $F/k$ a finite extension, and $A$ a $G_k$-module constant over $F$. Then
		\[
			|H^1_{ur}(k,A)| \le |A|^{[F:k]}\cdot |\Hom_{G_k}(\Cl_F,A)|,
		\]
		and in particular
			\[
				|H^1_{ur}(k,A)|\ll_{|A|,\epsilon} |\disc(F/\Q)|^{d(A)/2+\epsilon},
			\]
		where $d(A)$ is the minimal number of generators for $\hat{A}$ as a $G$-module.
	\end{lemma}
	\begin{proof}
		Let $G=\Gal(F/k)$.
		The inflation-restriction sequence gives an exact sequence 
		\[
		\begin{tikzcd}
			0 \rar & H^1(G,A) \rar & H^1(k,A) \rar & H^1(F,A)^{G}.
		\end{tikzcd}
		\]
		Thus the kernel of $H^1_{ur}(k,A) \ra  H^1_{ur}(F,A)^{G}$ is a subgroup of 
		$H^1(G,A)$ and hence of size at most $|A|^{[F:k]}$.
		Since $G_F$ acts trivially on $A$, we have that $H^1_{ur}(F,A)=\Hom(\Cl_F,A)$, and the $G$-invariant elements are precisely the $G$-equivariant homomorphisms 
		$\Hom_G(\Cl_F,A).$ Thus
		\begin{align*}
			|H^1_{ur}(k,A)| \le |A|^{[F:k]}\cdot |\Hom_G(\Cl_F,A)|.
		\end{align*}

		If we let $F_0$ be the field of definition of $A$, since $\Gal(F_0/k)$ is a subgroup of
		$|\Aut(A)|,$ we have $|A|^{[F_0:k]}\ll_{|A|} 1.$ 
		For any two finite $G$-modules $A,B$, we have a natural bijection $\Hom_G(A,B)\ra \Hom_G(\hat{B},\hat{A})$.  Thus $|\Hom_G(\Cl_{F_0},A)|=|\Hom_G(\hat{A},\widehat{\Cl_{F_0}})|$
		is at most $|\Cl_{F_0}|^{d(A)}.$  Using Minkowski's bound gives
		\begin{align*}
			|H^1_{ur}(k,A)|\le |A|^{[F_0:k]}\cdot |\Cl_{F_0}|^{d(A)} \ll_{|A|,\epsilon} |\disc(F_0/\Q)|^{d(A)/2+\epsilon}\leq |\disc(F/\Q)|^{d(A)/2+\epsilon}.
		\end{align*}
	\end{proof}

	A weaker version of this bound was used in \cite{alberts2020}. Improved bounds for $|H^1_{ur}(k,A)|$ are the primary reason that our results beat the unconditional upper bounds proven in \cite{alberts2020} for solvable groups.

\subsection{Inductive Bounds}
	We will get better bounds than Lemma~\ref{lem:H1_first_bound} by a strategic application of the (co)induced module.

	We consider an extension of number field $E/k$.
	We recall the definition of the induced module ${\rm Ind}_E^k(A) = \Z[G_k]\otimes_{\Z[G_{E}]} A$. Shapiro's lemma states that $H^1(k,{\rm Ind}_E^k(A)) \cong H^1(E,A)$. 
	(Recall that when moving between a group and a finite-index subgroup that the induced and coinduced modules are isomorphic.)
	We will need the restriction of this to the subgroups of unramified coclasses.

	\begin{lemma}\label{lem:induced}
	Let $k$ be a number field with finite extension $E$.
		Let $A$ be a $G_E$-module. Then Shapiro's isomorphism $H^1(k,{\rm Ind}_E^k(A)) \cong H^1(E,A)$ restricts to an isomorphism
		\[
			H^1_{ur}(k,{\rm Ind}_E^k(A)) \cong H^1_{ur}(E,A).
		\]
	\end{lemma}

	\begin{proof}
	This follows from the commutative diagram \cite[Equation (3.3)]{skinner2014iwasawa} when the bottom row is restricted to inertia. In fact, Skinner--Urban state this result in words a couple of paragraphs below this diagram.
	\end{proof}

	By making strategic use of Lemma~\ref{lem:induced}, we can prove the following bounds for $|H^1_{ur}(k,A)|$ that are useful in inductive arguments.

	\begin{lemma}\label{lem:inductive_H1ur_bound}
	Let $k$ be a number field and $F/k$ a finite extension. Let $A$ be a finite $G_k$-module constant over $F$. Let $G=\Gal(F/k)$. Suppose 
	\begin{itemize}
		\item $H\subset G$ is a subgroup,
		\item $M\subset A$ is a sub $H$-module.
	\end{itemize}
	Then
	\[
	|H^1_{ur}(k,A)| \le |H^1_{ur}(k,{\rm Core}(M))|\cdot |H^1_{ur}(F^H,A/M)| \cdot [A:M]^{[F^H:k]}[A:{\rm Core}(M)]^{\omega(F/k)-1},
	\]
	where ${\rm Core}(M) = \bigcap_{g\in G_k} gM$ is the $G_k$-core of $M$ and $\omega(F/k)$ equals the number of places ramified in $F/k$ (including infinite places). In particular, we also have
	\[
		|H^1_{ur}(k,A)| \ll_{|A|,\epsilon}  |H^1_{ur}(k,{\rm Core}(M))|\cdot |H^1_{ur}(F^H,A/M)| \cdot |\disc(F/\Q)|^{\epsilon}.
	\]
	\end{lemma}

	Before proving Lemma~\ref{lem:inductive_H1ur_bound}, we discuss how it may be used to improve upper bounds for $|H^1_{ur}(k,A)|$. Savings occur in essentially two ways:
	\begin{itemize}
		\item Moving from $A$ to the pair ${\rm Core}(M), A/M$ reduces the size of the modules being considered by a factor of $[M:{\rm Core(M)}]$. This translates into savings which are potentially significant for large modules that have few indecomposable factors.
		\item The presence of $F^H$ in place $F$ in the second factor introduces additional savings. This piece can then be bounded in terms of torsion in $\Cl_{F^H}$ instead of $\Cl_F$ using Lemma~\ref{lem:H1_first_bound}, which is typically smaller.
	\end{itemize}

	\begin{proof}
	The proof is via using exact sequences to bound the size of various terms. Consider the homomorphism
	\[
	\phi:A \to \text{CoInd}^G_H(A/M):=\Hom_{\Z[H]}(\Z[\Gal(F/k)],A/M)
	\]
	defined by $a\mapsto (f_a: r\mapsto raM)$.
	The kernel of this map is
	\[
	\{a\in A : ga \in M \text{ for all }g\in \Gal(F/k)\} = {\rm Core}(M).
	\]
	Thus we have an exact sequence
	\begin{equation}\label{E:conindex}
	\begin{tikzcd}
	0 \rar & {\rm Core}(M) \rar & A \rar{\phi} & \text{CoInd}^G_H(A/M).
	\end{tikzcd}
	\end{equation}
	By taking the corresponding long exact sequence of cohomology, we have an exact sequence
	\[
	\begin{tikzcd}
	H^0(k,\im\phi) \rar{\delta} & H^1(k,{\rm Core}(M)) \rar{\iota_*} & H^1(k,A) \rar{\phi_*} & H^1(k,\im\phi),
	\end{tikzcd}
	\]
	as well as
	the analogous sequence when $k$ is replaced by any inertia group $I_v$ of $\Gal(F/k)$.
	Next, we need to consider the unramified parts. We apply the snake lemma to
	\[
	\begin{tikzcd}
	{}&H^1(k,{\rm Core}(M))\dar \rar{\iota_*} & H^1(k,A)\dar \rar{\phi_*} & H^1(k,\im\phi)\dar\\
	0\rar & \prod_\nu H^1(I_v,{\rm Core}(M))/\delta(H^0(I_\nu,\im\phi)) \rar & \prod_\nu H^1(I_\nu,A) \rar & \prod_{\nu} H^1(I_\nu,\im\phi).
	\end{tikzcd}
	\]
	Just looking at the kernels from the snake lemma, this implies
	\[
	\begin{tikzcd}
	{\rm Sel}(k,{\rm Core}(M)) \rar& H^1_{ur}(k,A) \rar & H^1_{ur}(k,\im\phi)
	\end{tikzcd}
	\]
	is exact, where ${\rm Sel}(k,{\rm Core}(M))$ fits into the exact sequence
	\[
	\begin{tikzcd}
	0\rar & H^1_{ur}(k,{\rm Core}(M)) \rar &{\rm Sel}(k,{\rm Core}(M)) \rar& \prod_{\nu} \delta(H^0(I_\nu,\im\phi)).
	\end{tikzcd}
	\]
	Thus,
	\begin{equation}\label{E:H1main}
		|H^1_{ur}(k,A)| \le |{\rm Sel}(k,{\rm Core}(M))|\cdot |H^1_{ur}(k,\im\phi)|.
	\end{equation}
	We now bound $|{\rm Sel}(k,{\rm Core}(M))|$, starting with
	\begin{align*}
		|{\rm Sel}(k,{\rm Core}(M))| 
		&\le |H^1_{ur}(k,{\rm Core}(M))|\cdot \prod_{\nu} |\delta(H^0(I_\nu,\im\phi))|.
	\end{align*}
	If $\nu$ is unramified in $F/k$ then the connecting homomorphism $\delta$ is trivial on $H^0(I_\nu,\im\phi)$, and otherwise
	\[
		|\delta(H^0(I_v,\im\phi))|\le |\im\phi| = [A:{\rm Core}(M)].
	\]
	Thus we have bounded
	\begin{equation}\label{E:Selbound}
	|{\rm Sel}(k,{\rm Core}(M))| \le |H^1_{ur}(k,{\rm Core}(M))| [A:{\rm Core}(M)]^{\omega(F/k)}.
	\end{equation}

	In order to bound $H^1_{ur}(k,\im\phi)$, we obtain another exact sequence from \eqref{E:conindex},
	\[
	\begin{tikzcd}
	H^0(k,\coker\phi) \rar &H^1(k,\im\phi) \rar & H^1(k,\text{CoInd}^G_H(A/M)).
	\end{tikzcd}
	\]
	The analogous exact sequence for $I_v$ in place of $k$, and commuting restriction maps, imply we have a map
	\[
	\begin{tikzcd}
	 &H^1_{ur}(k,\im\phi) \rar & H^1_{ur}(k,\text{CoInd}^G_H(A/M)).
	\end{tikzcd}
	\]
	whose kernel has size at most $|H^0(k,\coker\phi)|$.
	Thus,
	\[
		|H^1_{ur}(k,\im\phi)| \le |H^0(k,\coker\phi)| \cdot |H^1_{ur}(k,\Hom_{\Z[H]}(\Z[\Gal(F/k)],A/M))|.
	\]
	We have
	\[
		|H^0(k,\coker\phi)|\le |\coker\phi| = \frac{|\Hom_{\Z[H]}(\Z[\Gal(F/k)],A/M)|}{|\im\phi|} = \frac{[A:M]^{[F^H:k]}}{[A:{\rm Core}(M)]}.
	\] Shapiro's Lemma restricted to the unramified classes (Lemma \ref{lem:induced}) implies 
	\[
		H^1_{ur}(k,\Hom_{\Z[H]}(\Z[\Gal(F/k)],A/M)) \cong H^1_{ur}(F^H,A/M).
	\]
	Thus we have shown
	\begin{equation}\label{E:H1unbound}
		|H^1_{ur}(k,\im\phi)| \le \frac{[A:M]^{[F^H:k]}}{[A:{\rm Core}(M)]}\cdot |H^1_{ur}(F^H,A/M)| .
	\end{equation}
	Multiplying \eqref{E:Selbound} and \eqref{E:H1unbound}, and applying \eqref{E:H1main},  gives
	\begin{align*}
		|H^1_{ur}(k,A)| &\le |H^1_{ur}(k,{\rm Core}(M))\cdot|[A:{\rm Core}(M)]^{\omega(F/k)-1}\cdot [A:M]^{[F^H:k]}\cdot |H^1_{ur}(F^H,A/M)| ,
	\end{align*}
	concluding the proof of the upper bound.

	For the $\ll$ upper bound, we use the fact that $G$ acts faithfully on $A$, and thus $[F^H:k]=|G|/|H|$ is bounded in terms of $|A|$.
	Also $c^{\omega(F/k)} \ll_{c,\epsilon} \disc(F/\Q)^{\epsilon}$ for any fixed constant $c$. This gives the $\ll$ upper bound.
	\end{proof}

\subsection{Applications}

	In several of the examples detailed in Section \ref{sec:examples}, we reference Lemma \ref{lem:inductive_H1ur_bound} directly so that we can choose $M$ and $H$ optimally for the given situation. However, in practice it can be difficult to determine which pairs $M$, $H$ are optimal for using Theorem \ref{thm:main_abelian_on_top}.

	We first give a simple lemma to let us compare between a module and its submodules.
	\begin{lemma}\label{lem:embedding_H1ur}
		Let $\iota:A_1\hookrightarrow A_2$ be an injective homomorphism of $G_k$-modules. Then
		\[
			|H^1_{ur}(k,A_1)| \le [A_2:A_1]\cdot |H^1_{ur}(k,A_2)|.
		\]
		In particular,
		\[
			|H^1_{ur}(k,A_1)| \ll_{|A_2|} |H^1_{ur}(k,A_2)|.
		\]
	\end{lemma}

	\begin{proof}
		The long exact sequence of cohomology gives an exact sequence
		$$
		H^0(k,A_2/A_1)\ra H^1(k,A_1) \ra H^1(k,A_2),
		$$
		and thus the kernel $N$ of $H^1_{ur}(k,A_1) \ra H^1_{ur}(k,A_2)$
		is surjected onto by a subgroup of $H^0(k,A_2/A_1)$. Thus,
		\begin{align*}
			|H^1_{ur}(k,A_1)| &\le |N|\cdot |H^1_{ur}(k,A_2)|\\
			&\le |H^0(k,A_2/A_1)|\cdot |H^1_{ur}(k,A_2)|\\
			&\le |A_2/A_1|\cdot |H^1_{ur}(k,A_2)|.
		\end{align*}
	\end{proof}

	Corollary \ref{cor:inductive_H1ur_bounds} in the introduction includes some special cases for which we know how to make an optimal choice for $M$ and $H$ in Lemma \ref{lem:inductive_H1ur_bound}.

	\begin{proof}[Proof of Corollary \ref{cor:inductive_H1ur_bounds}]
	Part (i): Let $M_0=A$.  We define $G_k$-modules $M_i$ and $N_i$ recursively, such that $N_i$ is the $G_k$-module generated by $g.m-m$ for $g\in G_k$ and $m\in M_i$ and $M_{i+1}={\rm Core}(N_i).$
	We then apply Lemma~\ref{lem:inductive_H1ur_bound} with $A=M_i$, and $H=G=\Gal(F/k)$, and the $M$ from Lemma~\ref{lem:inductive_H1ur_bound} being our $N_i$.
	Since $G_k$ acts trivially on $M_i/N_i$, we have $|H^1_{ur}(k,M_i/N_i)|=|\Hom(\Cl_k,M_i/N_i)|\ll_{k,|A|} 1$.  Thus we obtain
	\begin{align*}
	|H^1_{ur}(k,M_i)|&\ll_{|A|,\epsilon} |H^1_{ur}(k,{\rm Core}(N_i))||H^1_{ur}(k,M_i/N_i)| |\disc(F/\Q)|^\epsilon\\
	&\ll_{k,|A|,\epsilon} |H^1_{ur}(k,M_{i+1})||\disc(F/\Q)|^\epsilon.
	\end{align*}
	If we let $A_1=A$ and let $\Gamma^{j}_G(A)$ be the $G_k$-module generated by $g.m-m$ for $g\in G_k$ and $m\in \Gamma^{j-1}_G(A)$, then we can see inductively that $M_j\subset \Gamma^{j}_G(A).$  That $A$ is nilpotent means that $\Gamma^{j}_G(A)$ for some $j$, and thus $M_j=0$.  In particular this $j$ is bounded in terms of $|A|,$ and so we can apply the above inequality inductively to obtain the statement of part (i).

	Part (ii): Every proper subgroup of a simple module is necessarily core-free. Given a simple module $A$ of exponent $e$, choose some element $a\in A$ of order $e$. By the classification of finitely generated abelian groups, there exists a proper subgroup $M\le A$ for which $A = \langle a\rangle \oplus M$ as abelian groups. We know that ${\rm Core}(M) = 1$, so Lemma \ref{lem:inductive_H1ur_bound} with this $M$ and $H=1$ implies
	\begin{align*}
		|H^1_{ur}(k,A)| &\ll_{|A|,\epsilon} |H^1_{ur}(F,A/M)| \cdot |\disc(F/\Q)|^{\epsilon}.  
	\end{align*}
	Since $G_F$ acts trivially on $A$ and hence $A/M\cong \langle a\rangle$, we have
	\begin{align*}
		|H^1_{ur}(k,A)| &\ll_{|A|,\epsilon} |\Hom(\Cl_F,\langle a \rangle)| \cdot |\disc(F/\Q)|^{\epsilon} = |\Cl_F[e]| \cdot |\disc(F/\Q)|^{\epsilon}.  
	\end{align*}

	Part (iii): We are given an embedding $A'\hookrightarrow {\rm Ind}_F^k(A)$. 
	Lemmas \ref{lem:embedding_H1ur} and \ref{lem:induced} imply that
	\[
		|H^1_{ur}(k,A')| \ll_{|{\rm Ind}_F^k(A)|} |H^1_{ur}(k,{\rm Ind}_F^k(A))|=|H^1_{ur}(F,A)|
		=|\Hom(\Cl_F,A)|.
	\]
	\end{proof}

\section{The Pushforward Discriminant}\label{sec:pushforward_invariants}
	Let $G$ be a finite permutation group, and $T$ a normal subgroup of $G$.  We give $G/T$ the regular permutation action (i.e. by left mutiplication on the set of group elements).   
	We expressed the inputs of Theorem \ref{thm:main_abelian_on_top} in terms of the image
	\[
	q_*\Sur(G_k,G;X) = \{\pi\in \Sur(G_k,G/T) : \pi = q_*\psi\text{ for } \psi\in \Sur(G_k,G;X)\},
	\]
	which under the Galois correspondence is (up to multiplicity) the set of  $G/T$-extensions $L/k$ in $\bar{k}$ for which there exists a Galois $G$-extension $F/k$ with $F^T=L$ and $|\disc(F^{\Stab_G(1)}/\Q)|\le X$. The asymptotics of this particular set have not been studied previously to our knowledge.

	The primary difficulty is that ordering by the discriminant of a lift to a $G$-extension need not agree with a discriminant ordering for $G/T$-extensions. The purpose of this section is to define the pushforward discriminant $q_*\disc$ on $\Sur(G_k,G/T)$ in order to have an invariant we can relate to the discriminant of a $G$-extension lifting a $G/T$-extension. The point of the definition will be that
	\begin{equation}\label{E:push_forward}
	q_*\Sur(G_k,G;X) \subseteq \{\pi\in \Sur(G_k,G/T) : q_*\disc(\pi) \le X\}.
	\end{equation}

	For a prime ideal $\fp$ of $k$, 
	let $k_\fp$ denote the completion of $k$ at $\fp$.  For each prime ideal $\fp$ of $k$, we fix a choice of
	$k$-homomorphism $\bar{k}\ra\overline{k_\fp}$ giving a fixed choice of homomorphsim $G_{k_\fp}\ra G_k.$
	Given a $\psi\in\Sur(G_k,G)$, the relative discriminant ideal $\disc (\psi)$ is given as a product of local factors
	$$
	\disc(\psi) =\prod_\fp \fp^{f_\fp(\psi)},
	$$
	where the product is over prime ideals $\fp$ of $k$, and 
	${f_\fp(\psi)}$ is the local Artin conductor at $\fp$ of the composition of $\psi$ with the permutation representation of $G$. In particular, at a tame prime $\fp$ we have $f_\fp(\psi)=\ind(\psi(\tau_\fp)),$
	where $\tau_\fp$ is a generator of tame inertia.

	Let $q:G \ra H$ be a group homomorphism.
	For $\pi_\fp\in \Hom(G_{k_\fp},H)$ and $\pi\in\Hom(G_k,H)$, we define 
		\begin{align}\label{eq:pushforward_discriminant}
			q_*{f_\fp}(\pi_\fp) = \underset{\substack{\psi_\fp\in \Hom(G_{k_\fp},G)
			\\q\circ\psi_\fp=\pi_\fp}}{\min} {f_\fp}(\psi_\fp)\quad \quad\textrm{and}
			\quad \quad q_*\disc(\pi) =\prod_\fp \fp^{q_*f_\fp(\pi_p)},
		\end{align}
		where by convention $q_*{f_\fp}(\pi_\fp)=\infty$
		if there does not exist a $\psi_\fp\in \Hom(G_{k_\fp},G)$ lifting
		$\pi_\fp\in \Hom(G_{k_\fp},H)$.
	This immediately ensures \eqref{E:push_forward}.

	We define the pushforward of the index function to be
	\begin{align}\label{eq:pushforward_index}
		q_*\ind(gT) = \min_{hT=gT}\ind(h),
	\end{align}
	so that for tame places
	\[
		\nu_\fp(q_*\disc \pi) \ge q_*\ind(q_*\pi(\tau_\fp))
	\]
	for $\tau_\fp$ any generator of the tame inertia group at $\fp$. We then obtain the following conjecture following from a heuristic of Ellenberg and Venkatesh \cite[Question 4.3]{ellenberg-venkatesh2005}.

	\begin{conjecture}[The Weak Form of Malle's Conjecture for Pushforward Discriminants]\label{conj:WMpush}
		Let $k$ be a number field, $G$ a finite permutation group with normal subgroup $T\normal G$, and $q\colon G\to G/T$ the quotient map. Then
		\[
			\#\{\pi\in \Sur(G_k,G/T) : q_*\disc(\pi) \le X\} \ll_{\epsilon} X^{1/a(G-T) + \epsilon},
		\]
		where $a(G-T) = \min_{g\in G-T} \ind(g)$. 
	\end{conjecture}

	Ellenberg and Venkatesh's heuristic is known to hold for nilpotent groups, from which Conjecture \ref{conj:WMpush} for $G/T$ nilpotent follows. This follows from the discrimnant multiplicity conjecture for nilpotent groups \cite[Theorem 1.6]{Kluners2022}, or is proven directly by Alberts in \cite[Corollary 5.2]{alberts2020} (with $N=G$ a nilpotent group).

	\begin{remark}
		Ellenberg and Venkatesh also discuss a lower bound as part of their heuristic. However, the lifting condition in the pushforward discriminant makes it slightly larger than the general invariants considered by Ellenberg and Venkatesh. We cautiously expect that the lower bound $\gg_{\epsilon} X^{1/a(G-T) - \epsilon}$ should hold for each positive $\epsilon$, as the inequality $\nu_\fp(q_*\disc \pi) \ge \min_{\pi(I_p) = \langle gT\rangle} \ind(g)$ is an equality for a positive proportion of places  (namely, those congruent to $1$ mod $|G|$). For the purposes of this paper, we only require upper bounds.
	\end{remark}

\subsection{Imprimitive Extensions}
	Given a finite permutation group $G$, with subgroup $S=\Stab_G(1)$,
	a $G$-extension $K/k$ has a proper, non-trivial intermediate extension $L$ if and only if the is a subgroup $S'$ such that $S\lneq S' \lneq G.$  In this case,
	we can let $T=\cap_{g\in G} gS'g^{-1}$,  give $G/T$ the permutation action of left multiplication on the left cosets of $S'$, and $L/k$ is a $G/T$-extension.
	Let $q: G\ra G/T$.  Then we can compare $q_*\disc$ to the discriminant of a $G/T$-extension.

	\begin{proposition}\label{prop:imprim_beta}
	Let $k$ be a number field.
		Let $G,S,S',T,q$ be as just above.  Let $n=[G:S]$ and $m=[G:S']$.
		Then for all $X>0$,
		\[
			\Sur_{q_*\disc}(G_k,G/T;X) \subseteq \Sur(G_k,G/T;X^{\frac{m}{n}}),
		\]
		where $G/T$ is viewed as the permutation group in degree $m$, so the right-hand side is ordered with respect to $\disc_{G/T}$.
	\end{proposition}

	\begin{proof}
		Let $F$ be the field fixed by $\pi^{-1}(S')$. By the definition of $q_*\disc$ \eqref{eq:pushforward_discriminant}, it follows that for each prime ideal $\fp$ of $k$
		\[
			|\fp^{\nu_\fp(q_*\disc \pi)}| = \min_{\substack{\psi_{\fp} \in\Hom(G_{k_{\fp}},G)\\q\circ \psi_{\fp} = \pi_\fp}} |\disc(\psi_\fp)|.
		\]
		Let $L_{\psi_{\fp}}/k_{\fp}$ be the $G$-\'etale algebra corresponding to $\psi_\fp$. Notice that the subalgebra $L_{\psi_{\fp}}^{S'}$ fixed by $S'$ is necessarily the localization $F_{\fp}$. Then the proposition follows from
		\begin{equation*}
			|\disc(L_{\psi_\fp})|\ge |\disc(F_{\fp}/k_{\fp})^{[S':S]}|.
		\end{equation*}
		Thus, $|q_*\disc(\pi)| \ge |\disc(F/k)|^{[S':S]} = |\disc_{G/T}(\pi)|^{n/m}$ and the proposition follows.
	\end{proof}

\subsection{Corollary \ref{cor:main_abelian_on_top_imprimitive} follows from Theorem \ref{thm:main_abelian_on_top} and Corollary \ref{cor:inductive_H1ur_bounds}}\label{subsec:proving_abelian_on_top_imprimitive}
	Let $G$ be an imprimitive permutation group with tower type $(A,B)$, that is $G\subseteq A\wr B$ with $G$ surjecting onto $B$ and $A^m\cap G$ surjecting onto $A$ through each projection map. Suppose that
	\[
		\sum_{F\in \mathcal{F}_{m,k}(B;X)} |\Hom(\Cl_F,A)| \ll X^{\theta}.
	\]

	Let $S=\Stab_G(1)$, and $S'$ be the preimage of $\Stab_B(1)$ in $G$.
	Then let $T=A^m \cap G$, so $T=\cap_{g\in G} gS'g^{-1}.$
	We apply Proposition~\ref{prop:imprim_beta} to show that
	\[
		q_*\Sur(G_k,G;X) \subseteq \Sur(G_k,G/T; cX^{1/|A|})
	\]
	for some constant $c>0$ depending only on $[k:\Q]$ and $n$.

	The subgroup $A^m\le A\wr B$ carries the induced module structure by definition, that is
	\[
		A\wr B = {\rm Ind}_1^B(A) \rtimes B.
	\]
	Given any $F\in \mathcal{F}_{m,k}(B;X)$ corresponding to some $\pi\in \Sur(G_k,B)$, we then necessarily have and isomorphism of $G_k$-modules
	\[
		{\rm Ind}_1^B(A)(\pi) = {\rm Ind}_F^k(A),
	\]
	where $A$ carries the trivial $G_F$-action. We now see that our choice of $T = A^m\cap G = {\rm Ind}_1^B(A) \cap G$ necessarily admits an embedding $T(\pi) \hookrightarrow {\rm Ind}_F^k(A)$ as $G_k$-modules.  
	Corollary \ref{cor:inductive_H1ur_bounds}(iii) then  gives
	\[
		|H^1_{ur}(k,T(\pi))| \ll_{m,|A|} |\Hom(\Cl_F,A)|.
	\]
	Putting these together, we find that
	\begin{align*}
		\sum_{\pi\in q_*\Sur(G_k,G;X)}|H^1_{ur}(k,T(\pi))| &\ll_{m,|A|} \sum_{\pi\in \Sur(G_k,G/T; cX^{1/|A|})}|\Hom(\Cl_F,A)| \\
		&\ll_{m,|A|,k} X^{\theta/|A|}.
	\end{align*}

	Clearly $\theta < \frac{|A|}{a(A^m\cap G)}$ if and only if $\theta/|A| < 1/a(A^m\cap G) = 1/a(T)$, so the conclusions of Corollary \ref{cor:main_abelian_on_top_imprimitive} follow directly from the conclusions of Theorem \ref{thm:main_abelian_on_top}. 

\section{Groups concentrated in an Abelian Subgroup}\label{sec:abelian_on_top}
	When $T$ is an abelian group, Conjecture \ref{conj:twisted_number_field_counting} has been completely solved by Alberts and O'Dorney \cite{alberts-odorney2021}.
	In this section, we prove the upper bound of Conjecture \ref{conj:twisted_number_field_counting} for abelian $T$ \emph{with enough uniformity} for our desired applications.

	\begin{theorem}\label{thm:uniformity}
	Let $G$ be a transitive subgroup of degree $n$, $T\normal G$ an abelian normal subgroup with quotient map $q:G\to G/T$, and $\pi\in q_*\Sur(G_k,G)$. Then
	\[
	\#\{\psi\in q_*^{-1}(\pi) : |\disc_G(\psi)|\le X\} = O_{n,[k:\Q],\epsilon}\!\left(\frac{|H^1_{ur}(k,T(\pi))|}{(q_*\disc_G(\pi))^{1/a(T)-\epsilon}} X^{1/a(T)}(\log X)^{b(k,T(\pi))-1}\right)\!.
	\]
	\end{theorem}

	Together with Theorem \ref{thm:main_pointwise}, this will be sufficient to prove Theorem \ref{thm:main_abelian_on_top}.

	\begin{proof}[Proof of Theorem \ref{thm:main_abelian_on_top}]
		Let $T \normal G$ be an abelian normal subgroup for which the hypotheses of Theorem \ref{thm:main_abelian_on_top} is satisfied. We will prove that the hypotheses of Theorem \ref{thm:main_pointwise} are also satisfied with $a = a(T)$ and $b = \max_\pi b(k,T(\pi))$.

		Alberts gave a bijection between the fiber and ``surjective corssed homomorphisms valued in the Galois module $T(\pi)$ in \cite[Lemma 1.3]{alberts2021}, and further proves that this respects the coboundary relation in \cite[Lemma 3.5]{alberts2021} to conclude that
		\begin{align*}
			&\#\{\psi\in q_*^{-1}(\pi) : |\disc_G(\psi)|\le X\}\\
			&= |T(\pi)/T(\pi)^G| \cdot \#\{[f]\in H^1(k,T(\pi)) : f*\pi\text{ surjective},\ |\disc_G(f*\pi)|\le X\}.
		\end{align*}
		The asymptotic growth rate $c X^{1/a(T)}(\log X)^{b(k,T(\pi)) - 1}$ for this function is given directly by \cite[Theorem 1.1 and Corollary 1.2]{alberts-odorney2021} for $T$ abelian with no local restrictions. We remark that, for any $\pi$ with $b(k,T(\pi)) < b$, we may take $c(\pi) = 0$ so that Theorem \ref{thm:main_pointwise}(1) is verified for each fiber, even if that fiber does not contribute a positive proportion of extensions.

		\begin{remark}
			\cite{alberts-odorney2021} was originally published with an error in the main theorem, which has been corrected in the Corrigendum \cite{alberts-odorney2023}. This error applied to local restrictions - in certain cases (generalizing the Grunwald-Wang counterexample), the generating Dirichlet series cancels out completely and there are no elements of $H^1(k,T(\pi))$ satisfying that family of local conditions.

			In our setting, we are not considering any local restrictions whatsoever, which is equivalent to taking $L_{\fp} = H^1(k_{\fp},T(\pi))$ for all places ${\fp}$. This is a viable family of local restrictions in the sense of \cite{wood2009,alberts-odorney2023}, as certainly the trivial class $0$ satisfies these local conditions. For this reason, the results we are using from \cite{alberts-odorney2021} are correct as stated in the original publication.
		\end{remark}

		Theorem \ref{thm:main_pointwise}(2) follows from Theorem \ref{thm:uniformity}, with
		\[
			f(\pi) \ll_{n,[k:\Q],\epsilon} \frac{|H^1_{ur}(k,T(\pi))|}{(q_*\disc_G(\pi))^{1/a(T)-\epsilon}}
		\]
		if $\pi\in q_*\Sur(G_k,G)$, and $f(\pi) = 0$ if $\pi\not\in q_*\Sur(G_k,G)$ (as the fiber is empty in this case).

		Recall that we have assumed there is some $\theta \geq 0$ so that
			\[
				\sum_{\pi \in q_* \Sur(G_k,G;X)} |H^1_{ur}(k,T(\pi))|
					\ll_{n,k} X^{\theta}.
			\]
		From this, we find that
			\[
				\sum_{\pi \in q_* \Sur(G_k,G;X)} f(\pi)
					\ll_{n,k,\epsilon} \sum_{\pi \in q_* \Sur(G_k,G;X)} \frac{|H^1_{ur}(k,T(\pi))|}{(q_*\disc_G(\pi))^{1/a(T)-\epsilon}}
					\ll_{n,k,\epsilon} 1 + X^{\theta - 1/a(T) + \epsilon}.
			\]
		In particular, the criterion for convergence (Theorem~\ref{thm:main_pointwise}(3)) holds if $\theta < \frac{1}{a(T)}$.  Thus, we may apply Theorem~\ref{thm:main_pointwise} in this case, which yields Theorem~\ref{thm:main_abelian_on_top}(i).  
		
	If $\theta \ge 1/a(T)$, we bound the sum of the fibers directly as
	\begin{align*}
		\#\Sur(G_k,G;X) 
		&= \sum_{\pi\in q_*\Sur(G_k,G;X)} \#\{\psi\in q_*^{-1}(\pi) : |\disc_G(\psi)|\le X\}\\
		&\ll_{n,k,\epsilon} \sum_{\pi\in q_*\Sur(G_k,G;X)} f(\pi) X^{1/a(T)+\epsilon}\\
		&\ll_{n,k,\epsilon} X^{\theta+\epsilon},
	\end{align*}
	proving Theorem \ref{thm:main_abelian_on_top}(ii).
	\end{proof}

	The remainder of this section is dedicated to proving Theorem \ref{thm:uniformity}.

\subsection{Bounding by local factors}\label{subsec:upperboundMB}

	We will use the cohomological framework of \cite{alberts-odorney2021} to access the fibers, and bound them in terms of the Euler product of local factors
	\begin{align}\label{eq:MBseries}
		{\rm MB}_k(T,\pi;s) = \prod_{\fp} \frac{1}{|T|}\left(\sum_{\psi_{\fp}\in q_*^{-1}(\pi|_{G_{k_{\fp}}})} |\disc_G(\psi_{\fp})|^{-s}\right),
	\end{align}
	where $\disc_G(\psi_{\fp})$ is the discriminant of the $G$-\'etale algebra corresponding to $\psi_{\fp}$ over $k_{\fp}$, and the product is over all finite and infinite places of $k$. This is an analog to the Malle--Bhargava local series \cite{Bhargava2010, Wood2016}, and is equivalent to the Euler product appearing in \cite[Theorem 3.3]{alberts2021}.

	\begin{lemma}\label{lem:upperboundMBseries}
		Let $G$ be a transitive subgroup of degree $n$, $T\normal G$ an abelian normal subgroup with quotient map $q:G\to G/T$, and $\pi\in q_*\Sur(G_k,G)$.
		
		Let $\{a_m\}$ be the Dirichlet coefficients of ${\rm MB}_k(T,\pi;s)$, that is ${\rm MB}_k(T,\pi;s) = \sum a_m m^{-s}$. Then
		\[
			\#\{\psi\in q_*^{-1}(\pi) : |\disc_G(\psi)|\le X\} \le |H^1_{ur}(k,T(\pi))|\cdot |T[2]|^n\cdot \sum_{m\le X} a_m.
		\]
	\end{lemma}

	\begin{proof}
	Alberts gave a bijection between the fiber and a certain set of crossed homomorphisms in \cite[Lemma 1.3]{alberts2021}. Any nonempty fiber $q_*^{-1}(\pi)$ containing an element $\widetilde{\pi}$ is parametrized by crossed homomorphisms $Z^1(k,T(\widetilde{\pi}))$ valued in the Galois module $T(\widetilde{\pi})$ with action $x.t = \widetilde{\pi}(x) t \widetilde{\pi}(x)^{-1}$. Alberts used this to define a twisted version of the number field counting function predicting the asymptotic growth rate of the fibers via the set
	\begin{align*}
		\Sur(G_k,T,\widetilde{\pi};X):=\{f\in Z^1(k,T(\widetilde{\pi})) : f*\widetilde{\pi}\text{ surjective,}\ |\disc_G(f*\widetilde{\pi})|\le X\}.
	\end{align*}
	Here, $(f*\widetilde{\pi})(x) = f(x)\widetilde{\pi}(x)$ is the pointwise product of these maps and is necessarily a homomorphism.  We remark that, in the case that $T$ is abelian, the module $T(\widetilde{\pi})$ depends only on $\pi$ so we will often abuse notation write $T(\pi)$. While the set itself depends on the choice of lift $\widetilde{\pi}$, this set is in bijection with the fiber $\{\psi\in q_*^{-1}(\pi) : |\disc_G(\psi)| \le X\}$ so that the cardinality is independent of the choice of lift.

	The surjectivity and discriminant conditions are shown to factor through the coboundary relation in \cite[Lemma 3.5]{alberts2021}, which implies
	\begin{align*}
		&\#\{\psi\in q_*^{-1}(\pi) : |\disc_G(\psi)|\le X\}\\
			&\quad\quad\quad\quad = |T(\pi)/T(\pi)^G| \cdot \#\{[f]\in H^1(k,T(\pi)) : f\in \Sur(G_k,T,\widetilde{\pi};X)\}\\
			&\quad\quad\quad\quad \le \#\{[f]\in H^1(k,T(\pi)) : |\disc_G(f*\pi)|\le X\}.
	\end{align*}
	Thus, it suffices to bound the counting function
	\[
	H^1(k,T,\widetilde{\pi};X) := \{[f]\in H^1(k,T(\pi)) : |\disc_G(f*\widetilde{\pi})|\le X\}.
	\]
	This is precisely the type of counting function considered by Alberts--O'Dorney in \cite{alberts-odorney2021}, with no local restrictions and admissible ordering given by $\disc_{\widetilde{\pi}}(f) = \disc(f*\widetilde{\pi})$.

	We need to use the description of the generating Dirichlet series for $H^1(k,T,\widetilde{\pi};X)$ given by \cite[Theorem 2.3]{alberts-odorney2021}, which we summarize here: Let $H^1(\mathbb{A}_k,T(\pi))$ be the restricted direct product
	\[
	\left\{([f_{\fp}])\in \prod_{\fp} H^1(k_{\fp},T(\pi)) : [f_{\fp}] \in H^1_{ur}(k_{\fp},T(\pi))\text{ for all but finitely many }{\fp}\right\}.
	\]
	\cite[Theorem 2.3]{alberts-odorney2021} uses Poisson summation to prove that, for sufficiently nice functions $w:H^1(\mathbb{A}_k,T(\pi)) \to \mathbb{C}$, it follows that
	\[
	\sum_{f\in H^1(k,T(\pi))} w(f) = \frac{|H^0(k,T(\pi))|}{|H^0(k,T(\pi)^*)|}\sum_{h\in H^1(k,T(\pi)^*)}\hat{w}(h)
	\]
	where $\hat{w}$ is the Fourier transform of $w$ with respect the to Tate pairing, for $f\in H^1(k,T(\pi))$ we take $w(f) = w( (f|_{G_{k_{\fp}}})_{\fp} )$ for $(f|_{G_{k_{\fp}}})_{\fp}\in H^1(\mathbb{A}_K,T(\pi))$, and $T(\pi)^* = \Hom(T(\pi),\mu)$ is the Tate dual module of $T(\pi)$ with values in the group of roots of unity.

	Let $w(f) = |\disc_G(f*\widetilde{\pi})|^{-s}$ for some $s\in \mathbb{C}$. Alberts--O'Dorney show that \cite[Theorem 2.3]{alberts-odorney2021} applies to this function in \cite[Proposition 4.1]{alberts-odorney2021}. This function is multiplicative in the sense of \cite[Definition 3.1]{alberts-odorney2021}, which implies its Fourier transforms are Euler products. More precisely, the Fourier transforms are given by
	\[
		\hat{w}(h) = \prod_{\fp}\left(\frac{1}{|H^0(k,T(\pi))|}\sum_{[f]\in H^1(k_{\fp},T(\pi))}\langle f,h_{\fp}\rangle |\disc_G(f*\widetilde{\pi}_{\fp})|^{-s}\right)
	\]
	for each $h\in H^1(k,T(\pi)^*)$ with $h_{\fp} = h|_{G_{k_{\fp}}}$, $\widetilde{\pi}_{\fp} = \widetilde{\pi}|_{G_{k_{\fp}}}$, and
	\[
		\langle,\rangle : H^1(k_{\fp},T(\pi)) \times H^1(k_{\fp},T(\pi)^*) \to \mu
	\]
	the local Tate pairing.

	Moreover, this $w$ function is periodic with respect to the unramified coclasses by \cite[Proposition 4.1]{alberts-odorney2021}, which implies its Fourier transform has finite support
	\[
		H^1_{ur^*}(k,T(\pi)^*) = H^1(k,T(\pi)^*) \cap \prod_{\fp} H^1_{ur}(k_{\fp},T(\pi))^{\perp},
	\]
	the annihilator of the unramified coclasses in $H^1(k,T(\pi)^*)$ under the Tate pairing.

	All together, these facts give a concrete description of the generating series
	\[
	\sum_{f\in H^1(k,T(\pi))} |\disc_G(f*\widetilde{\pi})|^{-s} = \frac{|H^0(k,T(\pi))|}{|H^0(k,T(\pi)^*)|} \sum_{h\in H^1_{ur^*}(k,T(\pi)^*)} \hat{w}(h),
	\]
	with
	\[
	\hat{w}(h) = \prod_{\fp} \left(\frac{1}{|H^0(k,T(\pi))|}\sum_{[f]\in H^1(k_{\fp},T(\pi))}\langle f,h_{\fp}\rangle |\disc_G(f*\widetilde{\pi}_{\fp})|^{-s}\right).
	\]

	Let $\{a_m(h)\}$ be the Dirichlet coefficients for $\hat{w}(h)$, so that $\hat{w}(h) = \sum a_m(h) m^{-s}$ and
	\[
		\#H^1(k,T,\widetilde{\pi};X) = \frac{|H^0(k,T(\pi))|}{|H^0(k,T(\pi)^*)|} \sum_{h\in H^1_{ur^*}(k,T(\pi)^*)} \sum_{m\le X} a_m(h).
	\]
	The Tate pairing is valued in roots of unity, so in particular $|\langle f,h\rangle| = 1 = \langle f, 0\rangle$ for any $f,h$. This directly implies that the coefficients satisfy $|a_m(h)| \le a_m(0)$. Moreover, \cite[Proposition 3.6]{alberts2021}(ii) together with the bijection between crossed homomorphisms and fibers given by \cite[Lemma 1.3]{alberts2021} implies that $\hat{w}(0) = {\rm MB}_k(T,\pi;s)$, so that $a_m(0) = a_m$.

	Thus, we have shown
	\[
		\#H^1(k,T,\widetilde{\pi};X) \le \frac{|H^0(k,T(\pi))|}{|H^0(k,T(\pi)^*)|} |H^1_{ur^*}(k,T(\pi)^*)| \sum_{m\le X} a_m.
	\]

	Finally, we apply the the Greenberg--Wiles identity \cite[Theorem (8.7.9)]{neukirch-schmidt-wingberg2013cohomology} to control the size of the dual Selmer group $H^1_{ur^*}(k,T(\pi)^*)$. This identity states that the dual Selmer group is related to the usual Selmer group by
	\[
		\frac{|H^1_{ur}(k,T(\pi))|}{|H^1_{ur^*}(k,T(\pi)^*)|} = \frac{|H^0(k,T(\pi))|}{|H^0(k,T(\pi)^*)|}\prod_{\fp} \frac{|H^1_{ur}(k_{\fp},T(\pi))|}{|H^0(k_{\fp},T(\pi))|}.
	\]
	The product is supported only on infinite primes, of which there are at most $n$. We then conclude
	\begin{align*}
		\#H^1(k,T,\widetilde{\pi};X) &\le |H^1_{ur}(k,T(\pi))|\prod_{{\fp}\mid \infty}|H^0(k_{\fp},T(\pi))| \sum_{n\le X} a_m\\
		&\le |H^1_{ur}(k,T(\pi))| \cdot |T[2]|^n \sum_{m\le X} a_m.
	\end{align*}
	\end{proof}

\subsection{Complex Analysis}

	It now suffices to prove an upper bound for the sum of coefficients of ${\rm MB}_k(T,\pi;s)$, so that Theorem \ref{thm:uniformity} will follow from Lemma \ref{lem:upperboundMBseries}. We will do so by applying a Tauberian theorem to a smoothed sum of the coefficients, which means we need to understand the structure of ${\rm MB}_k(T,\pi;s)$ as a meromorphic function.

	It is proven in \cite[Theorem 3.3]{alberts2021} that ${\rm MB}_k(T,\pi;s)$ convergese absolutely on ${\rm Re}(s) > 1/a(T)$ and has a meromorphic continuation to an open neighborhood of ${\rm Re}(s) \ge 1/a(T)$ with a pole at $s=1/a(T)$ of order $b(k,T(\pi))$. We will need some more information in order to make the dependence on $\pi$ explicit, so we prove the following lemma constructing the meromorphic continuation.

	\begin{lemma}\label{lem:MBfactors}
	Let $G$ be a transitive subgroup of degree $n$, $T\normal G$ an abelian normal subgroup with quotient map $q:G\to G/T$, and $\pi\in q_*\Sur(G_k,G)$. Then there exist Dirichlet series $Q(T,\pi;s)$ and $G(T,\pi;s)$ and a Galois representation $\rho_{a(T)}$ for which
	\[
	{\rm MB}_k(T,\pi;s) = Q(T,\pi;s) L(a(T)s,\rho_{a(T)}) G(T,\pi;s),
	\]
	and such that
	\begin{enumerate}[(i)]
		\item For any integer $d\ge 0$, the $d^{\rm th}$ derivative of $Q(T,\pi;s)$ is bounded by
		\[
			|Q^{(d)}(T,\pi;s)| \ll_{n,d,[k:\Q],\epsilon} |q_*\disc(\pi)|^{-{\rm Re}(s) + \epsilon}
		\]
		for any $\epsilon > 0$ on the region ${\rm Re}(s) > 0$,
		\item For any integer $d\ge 0$, the $d^{\rm th}$ derivative of $G(T,\pi;s)$ is bounded by
		\[
			|G^{(d)}(T,\pi;s)| \ll_{n,d,\epsilon} 1
		\]
		on the region ${\rm Re}(s) > \frac{1}{a(T) + 1}+\epsilon$ for any $\epsilon > 0$, and
		\item For a positive integer $d>0$, the representation $\rho_d:G_k \to {\rm GL}(\mathbb{C}[A_d])$ is the permutation representation given by the Galois action on $A_d = \{t\in T : \ind(t) = d\}$ defined by $g: t \mapsto (\widetilde{\pi}(g)t\widetilde{\pi}(g)^{-1})^{\chi(g^{-1})}$ for some lift $\widetilde{\pi}\in q_*^{-1}(\pi)$.
	\end{enumerate}
	\end{lemma}

	The functions $Q(T,\pi;s)$ and $G(T,\pi;s)$ will be given explicitly in the proof. We chose to state the lemma in this way so that it can be more directly applied in the proof of Theorem \ref{thm:uniformity}. The proof is similar to that of \cite[Lemma 3.5]{alberts2021} and \cite[Corollary 3.3]{alberts2024}, although the detailed information we require for $Q(T,\pi;s)$ is not present in these pre-existing results.

	\begin{proof}
	We first consider the Euler factors for tame primes which are not ramified in $\pi$. The tame decomposition group at such $\fp$ has presentation
	\[
		G_{k_\fp}^{\rm tame} = \langle \tau_\fp,\Fr_\fp : \Fr_\fp \tau \Fr_\fp^{-1} \tau^{-|\fp|}\rangle.
	\]
	If $\fp\mid p$ for some rational prime $p$, local class field theory implies that the local cyclotomic character $G_{k_\fp}\to \Q_p^{\times}$ sends $\Fr_\fp\mapsto |\fp|$. Thus, for any prime $\fp\nmid |T|$ it follows that $\chi(\Fr_\fp) \equiv |\fp|\mod |T|$, so we may equivalently write
	\[
		G_{k_\fp}^{\rm tame} = \langle \tau_\fp,\Fr_\fp : \Fr_\fp \tau \Fr_\fp^{-1} \tau^{-\chi(\Fr_\fp)}\rangle.
	\]
	The Euler factors can then be written as
	\[
		\frac{1}{|T|}\sum_{\psi_\fp\in q_*^{-1}(\pi|_{G_{k_\fp}})}|\disc_G(\psi_\fp)|^{-s} = \frac{1}{|T|} \sum_{\substack{\tau,y\in G\\ y\tau y^{-1} \tau^{-\chi(\Fr_\fp)} = 1\\ \tau T = \pi(\tau_\fp)\\ yT = \pi(\Fr_\fp)}} |\fp|^{-\ind(\tau)}
	\]
	whenever $\fp\nmid |T|\infty$. The additional assumption that $\fp$ is unramified in $\pi$ implies $\pi(\tau_\fp) = 1$, so that
	\[
		\frac{1}{|T|}\sum_{\psi_\fp\in q_*^{-1}(\pi|_{G_{k_\fp}})}|\disc_G(\psi_\fp)|^{-s} = \frac{1}{|T|} \sum_{\substack{\tau\in T,\ ,y\in G\\ y\tau y^{-1} \tau^{-\chi(\Fr_\fp)} = 1\\ yT = \pi(\Fr_\fp)}} |\fp|^{-\ind(\tau)s}.
	\]
	Consider that $T$ abelian implies $y\tau y^{-1} = \widetilde{\pi}(\Fr_\fp)\tau\widetilde{\pi}(\Fr_\fp)$, as $yT = \widetilde{\pi}(\Fr_\fp)T = \pi(\Fr_\fp)$. Thus, if $\ind(\tau) = d$ then $y\tau y^{-1}\tau^{-\chi(\Fr_\fp)}=1$ if and only if $\tau = (\widetilde{\pi}(\Fr_\fp)\tau \widetilde{\pi}(\Fr_\fp)^{-1})\chi(\Fr_{\fp})$ is a fixed point of the permutation action on $A_d$. By the definition of $\rho_d$ as the permutation representation, this implies
	\begin{align*}
		\frac{1}{|T|}\sum_{\psi_\fp\in q_*^{-1}(\pi|_{G_{k_\fp}})}|\disc_G(\psi_\fp)|^{-s} &= \frac{1}{|T|} \sum_{d\ge 0}\sum_{\substack{y\in G\\yT=\pi(\Fr_\fp)}} \sum_{\substack{\tau\in A_d\\\tau\text{ fixed point}}} |\fp|^{-ds}\\
		&= \sum_{d\ge 0} {\rm tr}\rho_d(\Fr_\fp) |\fp|^{-ds}\\
		&= 1 + \sum_{d\ge 1} {\rm tr}\rho_d(\Fr_\fp) |\fp|^{-ds},
	\end{align*}
	where the last equality follows from $A_0 = \{1\}$.

	Consider that $\fp$ is ramified in $\pi$ if and only if $\fp\mid q_*\disc_G(\pi)$ by definition. We now set
	\begin{align*}
	Q(T,\pi;s) = \prod_{{\fp}\mid q_*\disc_G(\pi)|T|\infty}&\left(\frac{1}{|T|}\sum_{\psi_\fp\in q_*^{-1}(\pi|_{G_{k_\fp}})} |\disc_G(\psi_\fp)|^{-s}\right)\\
	&\cdot\det\left(I - \left(\rho_{a(T)}(\Fr_{\fp}) | \mathbb{C}[A_{a(T)}]^{I_{\fp}}\right) |{\fp}|^{-a(T)s}\right)
	\end{align*}
	and
	\[
	G(T,\pi;s) = \prod_{{\fp}\nmid q_*\disc_G(\pi)|T|\infty} \left(1 + \sum_{d\ge 1} {\rm tr} \rho_{d}(\Fr_{\fp}) |{\fp}|^{-ds}\right)\det\left(I - \rho_{a(T)}(\Fr_{\fp}) |{\fp}|^{-a(T)s}\right).
	\]
	By construction, these formally satisfy the relation
	\[
	{\rm MB}_k(T,\pi;s) = Q(T,\pi,s) L(a(T)s,\rho_{a(T)})G(T,\pi;s),
	\]
	and so this identity holds on the region of absolute convergence for ${\rm MB}_k(T,\pi;s)$ (i.e. for ${\rm Re}(s)>1/a(T)$). It now suffices to check the properties in parts (i) and (ii) (since part (iii) is just the definition for $\rho_d$).

	The Dirichlet series $Q(T,\pi;s)$ is in fact a Dirichlet polynomial, being a finite product of polynomials in $|{\fp}|^{-s}$. Writing $Q(T,\pi;s) = \sum_{\mathfrak{a}} \alpha_\pi(\mathfrak{a}) |\mathfrak{a}|^{-s}$, we immediately conclude that
	\begin{align*}
		|Q^{(d)}(T,\pi;s)| &= \left\lvert\sum_{\mathfrak{a}} \alpha_\pi(\mathfrak{a}) (-\log |\mathfrak{a}|)^d |\mathfrak{a}|^{-s}\right\rvert\\
		&\ll_{d,\epsilon} \sum_{\mathfrak{a}} |\alpha_\pi(\mathfrak{a})||\mathfrak{a}|^{-{\rm Re}(s)+\epsilon}\\
		&\ll_{d,\epsilon} \#\{\mathfrak{a} : \alpha_\pi(\mathfrak{a}) \ne 0\} \cdot \max_{\mathfrak{a}}|\alpha_\pi(\mathfrak{a})| \cdot \left(\min_{\alpha_\pi(\mathfrak{a})\ne 0} |\mathfrak{a}|\right)^{-{\rm Re}(s) + \epsilon}
	\end{align*}
	on the region ${\rm Re}(s) > 0$. Thus, it suffices to give bounds for these three factors.
	\begin{itemize}
		\item We first bound the length of the sum, i.e. the number of $\mathfrak{a}$ for which $\alpha_\pi(\mathfrak{a}) \ne 0$, by bounding the number of terms in each Euler factor. If $\fp$ is a tamely prime in $\pi$, then there are at most
		\[
			|\Hom(G_{k_{\fp}},G)|\cdot (\dim\rho_{a(T)} + 1) \le |G|^3
		\]
		terms. If $\fp\mid |T|$ is a wildly ramified prime, then there are similarly at most
		\[
			|\Hom(G_{k_{\fp}},G)|\cdot (\dim\rho_{a(T)} + 1) \le |G|^{d(G_{k_{\fp}}) + 1},
		\]
		where $d(G_{k_{\fp}})$ is the number of generators for the decomposition group at $\fp$, a number depending only on $[k:\Q]$. If $\fp\mid \infty$ is an infinite prime, then $\disc_G(\psi_\fp) = 1$ by convention, so this contributes at most to the coefficients themselves and not the number of terms. Overall, this implies
		\begin{align*}
			\#\{\mathfrak{a} : \alpha_{\pi}(\mathfrak{a}) \ne 0\} &\le \prod_{\fp\mid |T|} |G|^{d(G_{k_{\fp}}) + 1} \prod_{\substack{\fp\nmid |T|\\ \fp\mid q_*\disc(\pi)}}|G|^3\\
			&\ll_{n,[k:\Q]} |G|^{3\omega(q_*\disc(\pi))}\\
			&\ll_{n,[k:\Q],\epsilon} |q_*\disc(\pi)|^{\epsilon}.
		\end{align*}

		\item Next, we bound the values of the function $|\alpha_\pi(\mathfrak{a})|$. For each finite prime $\fp$ the coefficient of $|\disc_G(\psi_\fp)|$ is $1$, while the coefficient for $|\fp|^{-ds}$ in the determinant is bounded in absolute value by $|G|$ by $\rho_d$ a permutation representation (so any matrix in its image is a permutation matrix of dimension $|A_d|\le |G|$). Distributing implies that the coefficient is no more than $1$ times $|\Hom(G_{k_{\fp}},G)|\cdot |G|$ (an upper bound for the number of possible products of terms). In this case, we must also consider the infinite places, where we note that if $\fp\mid \infty$ then
		\[
			\left(\frac{1}{|T(\pi)|}\sum_{f\in Z^1(k_{\fp},T(\pi))} |\disc_G(f*\widetilde{\pi}|_{G_{k_{\fp}}})|^{-s}\right) = \frac{|\Hom(G_{k_{\fp}},G)|}{|T|} \le |G|
		\]
		by the infinite decomposition groups all being cyclic. Thus, we can use a similar upper bound to the last bullet point to show that
		\begin{align*}
			\max_{\mathfrak{a}} |\alpha_{\pi}(\mathfrak{a})| &\le \prod_{\fp\mid \infty} |G| \prod_{\fp\mid |T|} |G|^{d(G_{k_{\fp}}) + 1} \prod_{\substack{\fp\nmid |T|\\ \fp\mid q_*\disc(\pi)}}|G|^3\\
			&\ll_{n,\epsilon} |q_*\disc(\pi)|^{\epsilon}.
		\end{align*}

		\item Finally, we determine the smallest integer in the support of $\alpha_\pi$. This is the product of the smallest degree terms from each Euler factor. If $p$ is not ramified in $\pi$, then certainly exists an unramified lift $\psi_\fp\in q_*^{-1}(\pi|_{G_{k_\fp}})$ (because $\Gal(k_\fp^{ur}/k_\fp) = \langle \Fr_\fp\rangle$ is a free group) which satisfies $\disc_G(\psi_\fp) = 1$. In these cases the Euler factor would have a constant term. if $\fp$ is ramified in $\pi$, then the minimum degree term is given by
		\begin{align*}
			\min_{\psi_\fp\in q_*^{-1}\pi|_{G_{k_\fp}}} |\disc_G(\psi_\fp)|.
		\end{align*}
		Appealing to the definition of the pushforward discriminant \eqref{eq:pushforward_discriminant}, this is given by
		\begin{align*}
			\min_{\psi_\fp\in q_*^{-1}\pi|_{G_{k_\fp}}} |\disc_G(\psi_\fp)| & = \min_{\psi_\fp\in q_*^{-1}\pi|_{G_{k_\fp}}} |\fp|^{f_\fp(\psi_\fp)}\\
			&= |\fp|^{q_*f_\fp(\psi_\fp)}\\
			&= |\fp|^{\nu_{\fp}(q_*\disc_G(\pi)}.
		\end{align*}
		Multiplying these together, we have shown that $|q_*\disc(\pi)|=\min_{\alpha_\pi(\mathfrak{a})\ne 0} |\mathfrak{a}|$.
	\end{itemize}
	All together, we have proven that
	\[
		Q^{(d)}(T,\pi;s) \ll_{n,d,[k:\Q],\epsilon} |q_*\disc_G(\pi)|^{-{\rm Re}(s) + 3\epsilon}
	\]
	on the region ${\rm Re}(s) > 0$, so replacing $\epsilon$ with $\epsilon / 3$ concludes the proof of part (i).

	Part (ii) is proven similarly. Write $G(T,\pi;s) = \sum_{\mathfrak{a}} \beta_\pi(\mathfrak{a}) \mathfrak{a}^{-s}$. $G(T,\pi;s)$ is a product over tame places, so by the same argument as above each coefficient in the Euler product at $\fp$ is bounded above by $|G|^{3}$. Moreover, the smallest degree term appearing in this Euler product is $|\fp|^{-(a(T)+1)s}$, as the $|\fp|^{-a(T)s}$ terms cancel out after distributing. Let $f$ be the characteristic function supported on ideals of $k$ for which $\fp\mid \mathfrak{a}\Rightarrow \nu_{\fp}(\mathfrak{a}) \ge a(T) + 1$. Then
	\[
		|\beta_\pi(\mathfrak{a})| \le |G|^{3\omega(\mathfrak{a})}f(\mathfrak{a}) \ll_{n,\epsilon} f(\mathfrak{a}) |\mathfrak{a}|^{\epsilon}.
	\]
	For any $d\ge 0$, we have
	\[
		G^{(d)}(T,\pi;s) = \sum_{\mathfrak{a}} \beta_\pi(\mathfrak{a}) (-\log |\mathfrak{a}|)^d |\mathfrak{a}|^{-s}.
	\]
	The corresponding absolute series is bounded by
	\begin{align*}
		\sum_{\mathfrak{a}} |\beta_\pi(\mathfrak{a}) (\log |\mathfrak{a}|)^d |\mathfrak{a}|^{-s}| &\ll_{n,d,\epsilon} \sum_{\mathfrak{a}} f(\mathfrak{a}) |\mathfrak{a}|^{\epsilon} |\mathfrak{a}|^{\epsilon} |\mathfrak{a}|^{-{\rm Re}(s)}\\
		&= \prod_{\fp} \left(1 + \sum_{e=a(T) + 1}^{\infty} |\fp|^{e(-{\rm Re}(s)+2\epsilon)}\right),
	\end{align*}
	where converges absolutely on the region ${\rm Re}(s) > \frac{1}{a(T) + 1}$ and is independent of $\pi$.
	\end{proof}

	This is sufficient to apply a Tauberian theorem to ${\rm MB}_k(T,\pi;s)$. In order to carry the dependence on $\pi$ through this argument, we need more information about the Artin $L$-function $L(a(T)s,\rho_{a(T)})$. Luckily, permutation representations are particularly nice.

	\begin{lemma}\label{lem:permRep}
	Let $A$ be a finite set with a (left) $G$ action and $\mathbb{C}[A]$ the corresponding $G$-module. Then
	\[
	\mathbb{C}[A] = \bigoplus_{i=1}^n \mathbb{C}[G/H_i]
	\]
	for $H_i$ the sequence of stabilizers of the $G$-orbits of $A$.

	In particular, if $G = G_k$ and $\rho:G_k\to {\rm GL}(\mathbb{C}[A])$ is the corresponding permutation representation, then
	\[
		L(s,\rho) = \prod_{i=1}^n \zeta_{k_i}(s),
	\]
	where $k_i$ is the field fixed by the stabilizer $H_i\le G_k$.
	\end{lemma}

	\begin{proof}
	Decompose $A$ into a disjoint union of $G$-orbits $\mathcal{O}_1$,...,$\mathcal{O}_n$. Then
	\[
	\mathbb{C}[A] = \bigoplus_{i=1}^n \mathbb{C}[\mathcal{O}_i].
	\]
	For each element $x\in\mathcal{O}_i$ and $g\in G$, we necessarily have that $g.x = gh.x$ for every $h\in \Stab(x)$. If we fix a base point of each orbit $x_i$, then
	\[
	\mathbb{C}[\mathcal{O}_i] = \mathbb{C}[G/\Stab(x_i)].
	\]
	We remark that changing the base point of $\mathcal{O}_i$ only changes the stabilizer up to conjugation, and $\mathbb{C}[G/H]\cong \mathbb{C}[G/H^g]$ as $G$-modules.

	In particular, for $G = G_k$ this implies
	\[
	\rho = \sum_{i=1}^n \ind_{H_i}^{G_k}(1_{H_i}),
	\]
	where $1_{H_i}$ is the trivial representation on $H_i$ so that
	\[
	L(s,\rho) = \prod_{i=1}^n L(s,\ind_{H_i}^{G_k}(1_{H_i})).
	\]
	Artin $L$-functions are invariant under under representations, so we have proven that
	\[
	L(s,\rho) = \prod_{i=1}^n L(s,1_{H_i}) = \prod_{i=1}^n \zeta_{k_i}(s).
	\]
	\end{proof}

	Finally, we will require an analog to Lemma \ref{lem:MBfactors}(i,ii) for the Dedekind zeta function. It turns out that knowing an analogous upper bound at $s=1$ will suffice, so we prove the following lemma on the Laurent expansion.

	\begin{lemma}\label{lem:DedekinZetaLaurent}
		Let $K$ be a number field of degree $n \geq 2$.  Let $\zeta_K(s)$ denote its Dedekind zeta function, and let $c_{-1}, c_0, \dots$ be its Laurent series coefficients about $s=1$, i.e.
			\[
				\zeta_K(s)
					= \frac{c_{-1}}{s-1} + c_0 + c_1 (s-1) + \dots.
			\]
		Then for every $r \geq -1$, we have $c_r = O_{n,r}((\log |\disc(K/\Q)|)^{n+r+2})$.
	\end{lemma}
	\begin{proof}
		For convenience, set $\Delta = \max\{ |\disc(K/\Q)|, e^4 \}$.  We recall the convexity bound for $\zeta_K(s)$ in the following form.  Let $\delta$ be a real number such that $0 < \delta < \frac{1}{4}$.  Then for every $s$ with real part $\sigma$ between $-\delta$ and $1+\delta$, we have
		\[
			\left| \frac{s-1}{s+1} \zeta_K(s) \right|
				\ll_n \delta^{-1} \Delta^{\frac{1+\delta-\sigma}{2}} (1+|t|)^n (\log \Delta)^n.
		\]
		(This is standard, but see \cite[Lemma~4.1]{LemkeOliver2024}, for example.) 
		On the circle $|s-1| = \delta$, we therefore find that
			\[
				|\zeta_K(s)| 
					\ll_n \delta^{-2} \Delta^{\delta} (\log \Delta)^n.
			\]
		By the Cauchy integral formula, we then have
			\[
				c_r
					= \frac{1}{2\pi i} \oint_{|s-1| = \delta} \frac{\zeta_K(s)}{(s-1)^{r+1}}\, ds
					\ll_{n} \delta^{-2-r}\Delta^{\delta} (\log \Delta)^n. 
			\]
		In particular, choosing $\delta = 1/\log \Delta$, we find that
			\[
				c_r
					\ll_n (\log \Delta)^{n+r+2}.
			\]
	\end{proof}

\subsection{Proving Theorem \ref{thm:uniformity}}
	Lemmas \ref{lem:upperboundMBseries}, \ref{lem:MBfactors}, \ref{lem:permRep}, and \ref{lem:DedekinZetaLaurent} give us enough information to perform a contour shifting argument that keeps track of the dependence on $\pi$.

	\begin{proof}[Proof of Theorem \ref{thm:uniformity}]
	Following Lemma \ref{lem:upperboundMBseries} we have
	\[
	\#\{\psi\in q_*^{-1}(\pi) : |\disc_G(\psi)| \le X\} \ll_{n} |H^1_{ur}(k,T(\pi))| \sum_{j\le X} a_j,
	\]
	so it suffices to bound $\sum_{j\le X} a_j$.

	By Perron's formula we have for $c=1/a(T)$
	\begin{align*}
		\sum_{j \leq X} a_j &\leq \sum_{j=1}^\infty a_j e^{1-\frac{j}{X}}\\
		&= \frac{e}{2\pi i} \int_{c+\epsilon-i\infty}^{c+\epsilon+i\infty} {\rm MB}_k(T,\pi;s)\cdot \Gamma(s)\cdot X^s\,ds.
	\end{align*}
	We next shift the contour integral to ${\rm Re}(s)=c-\epsilon$. Lemma \ref{lem:MBfactors}(i,ii) when $d=0$, Lemma \ref{lem:permRep} and the convexity bound for Dedekind zeta functions together imply that on the region $c-\epsilon \le {\rm Re}(s) \le c+\epsilon$
	\[
		|{\rm MB}_k(T,\pi;s)| \ll_{n,[k:\Q],\epsilon} |s-1/a(T)|^{-b(k,T(\pi))}|q_*\disc_G(\pi)|^{-c+\epsilon}(1+|t|)^{O_n(1)}.
	\]
	The rapid decay of $\Gamma(s)$ in vertical strips then implies that
	\[
		\lim_{t\to \infty} \frac{e}{2\pi i} \int_{c-\epsilon\pm it}^{c+\epsilon\pm it} {\rm MB}_k(T,\pi;s)\cdot \Gamma(s)\cdot X^s\,ds = 0.
	\]
	Thus, the Cauchy residue theorem implies
	\begin{align*}
		\sum_{j\le X} a_j &\le \Res({\rm MB}_k(T,\pi;s) \Gamma(s) X^s)_{s = 1/a(T)} + \frac{e}{2\pi i} \int_{c-\epsilon-i\infty}^{c-\epsilon+i\infty} {\rm MB}_k(T,\pi;s)\cdot \Gamma(s)\cdot X^s\,ds.
	\end{align*}
	Once again, the rapid decay of $\Gamma(s)$ combined with the upper bounds for $|{\rm MB}_k(T,\pi;s)|$ imply that the integral is bounded by
	\begin{align*}
		\frac{e}{2\pi i} \int_{c-\epsilon-i\infty}^{c-\epsilon+i\infty} {\rm MB}_k(T,\pi;s)\cdot \Gamma(s)\cdot X^s\,ds \ll_{n,[k:\Q],\epsilon} |q_*\disc_G(\pi)|^{-1/a(T)+ \epsilon} X^{1/a(T)-\epsilon}.
	\end{align*}
	It now suffices to bound the residue. We will do this in terms of the factorization given by Lemma \ref{lem:MBfactors} and Lemma \ref{lem:permRep}
	\begin{align*}
		\Res({\rm MB}_k(T,\pi;s) \Gamma(s) X^s)_{s = 1/a(T)} &= \Res\left(Q(T,\pi;s)\prod_i \zeta_{k_i}(s) G(T,\pi;s) \Gamma(s) X^s\right)_{s = 1/a(T)},
	\end{align*}
	where $k_i$ are the fields fixed by the stabilizers of the orbits in $A_{a(T)} = \{t\in T : \ind(t) = a(T)\}$.

	We now express the residue in terms of Laurent coefficients. Let $q_r$ be the Laurent coefficients of $Q(T,\pi;s)$ at $s=1/a(T)$, $c_{r,i}$ the Laurent coefficients of $\zeta_{k_i}(s)$ at $s=1$, $g_r$ the Laurent coefficients of $G(T,\pi,s)$ and $s=1/a(T)$, and $\gamma_r$ the Laurent coefficients of $\Gamma(s)$ at $s=1/a(T)$. Then
	\begin{align*}
		&\Res({\rm MB}_k(T,\pi;s) \Gamma(s) X^s)_{s = 1/a(T)} \\
			&\quad\quad\quad\quad= \sum_{\substack{r_Q + r_G + \sum_i r_i + r_\Gamma + r_X = -1\\r_X\ge 0}} q_{r_Q} \left(\prod_i \frac{c_{r_i,i}}{a(T)}\right) g_{r_G} \gamma_{r_\Gamma} X^{1/a(T)} (\log X)^{r_X}.
	\end{align*}
	Given that $Q(T,\pi;s)$, $G(T,\pi;s)$, and $\Gamma(s)$ are holomorphic at $s=1/a(T)$, we may also restrict this sum to $r_Q,r_G,r_\Gamma \ge 0$. We also know that $c_{r,i} = 0$ if $r < -1$, so we can restrict to $r_i \ge -1$. The equation $r_Q+r_G+\sum_i r_i + r_\Gamma + r_X = -1$ together with the lower bounds $r_Q,r_G,r_\Gamma,r_X\ge 0$, $r_i\ge -1$ imply the upper bounds $r_Q,r_G,r_\Gamma,r_X,r_i\le b(k,T(\pi)) - 1$. This is because the largest possible negative contribution from the left hand side is $\sum_i -1 = -b(k,T(\pi))$.

	Bounding above by the largest term times the length of the series, this gives an upper bound of the form
	\begin{align*}
		\le \frac{b(k,T(\pi))^{b(k,T(\pi)) + 4}}{a(T)^{b(k,T(\pi))}} \max_{\substack{0\le r_Q,r_G,r_\Gamma,r_X\le b(k,T(\pi)) - 1\\-1\le r_i \le b(k,T(\pi))-1}} |q_{r_Q}| \left(\prod_i |c_{r,i}|\right) |g_{r_G}| |\gamma_{r_\Gamma}| X^{1/a(T)} (\log X)^{r_X}.
	\end{align*}
	Lemma \ref{lem:MBfactors}(i) implies that
	\[
		q_{r} \ll_{n,r,[k:\Q],\epsilon} |q_*\disc_G(\pi)|^{-1/a(T) + \epsilon},
	\]
	while Lemma \ref{lem:MBfactors}(ii) implies that
	\[
		g_r \ll_{n,r,\epsilon} 1.
	\]
	Lemma \ref{lem:DedekinZetaLaurent} states explicitly that
	\[
		c_{r,i} \ll_{[k_i:\Q],r} (\log|\disc(k_i/\Q)|)^{[k_i:\Q]+r+2}.
	\]
	We know that $k_i$ is fixed by $\ker\pi$, so in particular any prime that ramifies in $k_i$ necessarily ramifies in $\pi$, and so divides $q_*\disc_G(\pi)$. We also know $[k_i:\Q] \le |G|[k:\Q]$ is bounded in terms of $n$ (by $|G| \le n!$) and $[k:\Q]$. This implies that, after appropriately adjusting the value of $\epsilon$,
	\[
		c_{r,i} \ll_{n,[k:\Q],r,\epsilon} |\disc(k_i/\Q)|^{\epsilon(n![k:\Q]+r+2)} \ll_{n,[k:\Q],r,\epsilon} |q_*\disc_G(\pi)|^{\epsilon}.
	\]
	Lastly, $\gamma_r\ll_r 1$ because $\Gamma(s)$ is independent of all other parameters. Putting these all together gives an upper bound for the residue of the form
	\begin{align*}
		&\ll_{n,b(k,T(\pi)),[k:\Q],\epsilon} \frac{b(k,T(\pi))^{b(k,T(\pi)) + 4}}{a(T)^{b(k,T(\pi))}} \max_{0\le r_X\le b(k,T(\pi)) - 1} |q_*\disc_G(\pi)|^{-1/a(T) + 2\epsilon}X^{1/a(T)} (\log X)^{r_X}\\
		&\ll_{n,a(T),b(k,T(\pi)),[k:\Q],\epsilon} |q_*\disc_G(\pi)|^{-1/a(T) + 2\epsilon}X^{1/a(T)} (\log X)^{b(k,T(\pi))-1}.
	\end{align*}
	We know that $a(T) \le n$ and $b(k,T(\pi)) \le |T| \le n!$ are both bounded in terms of $n$. Thus, we have shown
	\[
		\Res({\rm MB}_k(T,\pi;s) \Gamma(s) X^s)_{s = 1/a(T)} \ll_{n,[k:\Q],\epsilon}|q_*\disc_G(\pi)|^{-1/a(T) + \epsilon}X^{1/a(T)} (\log X)^{b(k,T(\pi))-1},
	\]
	the same upper bound as the integral term.

	Put together, we have proven that
	\begin{align*}
		\sum_{j \leq X} a_j & \ll_{n,[k:\Q],\epsilon} |q_*\disc_G(\pi)|^{-1/a(T)+\epsilon} X^{1/a(T)}(\log X)^{b(k,T(\pi))-1}.
	\end{align*}
	Multiplying by $|H^1_{ur}(k,T(\pi))|$ gives the required bound for the fiber, concluding the proof.
	\end{proof}

\section{Examples}\label{sec:examples}

	In this section we give proofs for the examples given in the introduction. This includes a summary of the results of 
	Corollary~\ref{cor:compute}
	and the \verb^Magma^ code used to produce them, as well as proofs for the infinite families of examples following from the statements of Theorem \ref{thm:main_S3_wreath} and Theorem \ref{thm:main_abelian_on_top}.

\subsection{Computations for Groups of Degree up to $23$}

	We now describe in more detail the computation leading to Corollary~\ref{cor:compute}, which reports on the number of permutation groups of degree $\leq 23$ to which our methods apply.  We begin by creating a list of the $4953$ permutation groups of degree up to $23$, along with computing a preliminary exponent $\alpha$ for each such group $G$ so that the number of $G$-extensions is at most $O_{\epsilon}(X^{\alpha+\epsilon})$.  This exponent $\alpha$ is the smallest of the following, when applicable: if $G$ is nilpotent, the reciprocal of the index of $G$ \cite{alberts2020}; the Schmidt bound \cite{schmidt1995}; Bhargava's improvement to the Schmidt bound for primitive groups \cite{bhargava2024}; and bounds coming from the degrees of algebraically independent invariants of $G$ \cite{LO-uniform}.  We also check at this stage whether Conjecture~\ref{conj:number_field_counting} is known a priori for $G$, including for nilpotent groups subject to Corollary~\ref{cor:nilpotent}, groups of the form $S_3 \times A$, $S_4 \times A$, and $S_5 \times A$ for $A$ abelian \cite{masri_thorne_tsai_wang2020}, and $S_3$ in its degree $6$ regular representation.

	We now iterate through this list multiple times, aiming to improve the recorded bound when possible and to detect whether Conjecture~\ref{conj:number_field_counting} is now known for $G$ based on our techniques.  This process of iterative improvement has two pieces.  First, we check whether there is another faithful permutation representation of $G$ in our database, and if so, whether ``swapping'' to that representation yields stronger bounds on our given representation.  Second, we iterate over the normal abelian subgroups $G$ to see whether we obtain an improved upper bound or a proof of Conjecture~\ref{conj:number_field_counting} for $G$ by applying Theorem~\ref{thm:main_abelian_on_top}.  It is this step that is the computational bottleneck.  To simplify this computation in practice, we use Lemma~\ref{lem:H1_first_bound} to compute upper bounds on $H_\text{ur}^1(k,T(\pi))$ rather than the full inductive machinery built in Section~\ref{sec:classgrp}, and we restrict this process to solvable groups $G$.  Finally, once an updated bound on $G$-extensions has been computed, we check whether the group $S_3 \wr G$ is in our database, and if so, whether Theorem~\ref{thm:main_S3_wreath} yields stronger results for it.

	The code used to execute this computation and the resulting database are available at \cite{codeurl}.  However, because of the simplifications described above, this computation is ad hoc.  We consider it an interesting challenge, both theoretically and computationally, to expand its scope.

\subsection{Groups of degree $6$}

	The first degree that we prove new results for is $6$. We describe the known results for degree $6$ in Table \ref{tb:deg6} to showcase the smallest examples of our methods. To summarize, we prove Conjecture \ref{conj:number_field_counting} for four new groups in degree $6$ over an arbitrary base field. Our main results also prove upper bounds, which we optimized for this table. We prove the weak form of Malle's conjecture for three additional new groups in degree $6$. We do not prove any lower bounds in this paper, but we include the currently known lower bounds in the table for the sake of completeness.

	The groups labelled as $nTd$ refer to the group \verb|TransitiveGroup(n,d)| in \verb^Magma^'s database of transitive permutation groups. We also include a classical label for each group.

	\begin{table}[!htbp]
		\begin{center}
		\begin{tabular}{|c|l| p{0.85in} p{2.1in} p{1.4in} |}
			\hline
			\multicolumn{2}{|l|}{Group} & Known & Known & Reference(s) \\
			\multicolumn{2}{|l|}{} & Asymptotic & Bounds & \\\hline\hline
			\multicolumn{5}{c}{concentrated}\\\hline
			6T\ref{it:6T1}&$C_6$& $c_1X^{1/3}$ & & \cite{wright1989}\\
			6T\ref{it:6T3}& $S_3\times C_2$& $c_3 X^{1/2}$ & & \cite{masri_thorne_tsai_wang2020} over $\Q$\\
			6T\ref{it:6T4}&$A_4$& $c_4 X^{1/2}$ & & Thm \ref{thm:main_abelian_on_top}\\
			6T\ref{it:6T5}&$C_3\wr C_2$ & $c_5 X^{1/2}\!\log X$ & & Cor \ref{cor:iterated_cyclic_wreath}\\
			6T\ref{it:6T6}&$C_2\wr C_3$& $c_6 X$ & & \cite{kluners2012}\\
			6T\ref{it:6T8}&  $S_4$ & $c_8X^{1/2}$ & & Thm \ref{thm:main_abelian_on_top}\\
			6T\ref{it:6T9}& ${\rm Hol}(S_3)\cong S_3^2$ & & $X^{1/2}\!\ll\! N(X)\!\ll_{\epsilon}\! X^{1/2+\epsilon}$ & \cite{alberts2021},Thm \ref{thm:main_abelian_on_top} \\
			6T\ref{it:6T10}&  $(C_3)^2\rtimes C_4$& &$X^{1/2}\!\ll\! N(X)\!\ll_{\epsilon}\! X^{1/2+\epsilon}$ & \cite{alberts2021},Thm \ref{thm:main_abelian_on_top}\\ 
			6T\ref{it:6T11}&  $C_2\wr S_3$ & $c_{11}X$ & &\cite{kluners2012}\\
			6T\ref{it:6T13}& $S_3\wr C_2$& $c_{13}X$ & & Cor \ref{cor:S3_wreath}\\
			6T\ref{it:6T14}& $S_5$& &$X^{1/2}\ll N(X) \ll X$ & \cite{bhargava-shankar-wang2015}\\ 
			\hline
			\multicolumn{5}{c}{non-concentrated}\\\hline
			6T\ref{it:6T2}& $S_3$ & $c_2 X^{1/3}$ & & \cite{bhargava-wood2007,belabas2010discriminants}\\
			6T\ref{it:6T7}& $S_4$ & & $X^{1/2}\!\ll\! N(X) \!\ll_{\epsilon}\! X^{1/2+\epsilon}$ & \cite{alberts2021},Thm \ref{thm:main_abelian_on_top}\\ 
			6T\ref{it:6T12}& $A_5$ & & $X^{\frac{59}{1920}-\epsilon}\!\ll_{\epsilon}\! N(X) \!\ll\! X$ & \cite{pierce_turnage-butterbaugh_wood2021,bhargava-shankar-wang2015}\\ 
			6T\ref{it:6T15}& $A_6$& &$X^{\frac{359}{7200}-\epsilon}\!\ll_{\epsilon}\! N(X)\!\ll_{\epsilon}\! X^{3/2+\epsilon}$ & \cite{pierce_turnage-butterbaugh_wood2021,LO-uniform}\\
			6T\ref{it:6T16}& $S_6$& &$X^{7/10}\!\ll\! N(X)\!\ll\! X^2$ & \cite{Bhargava2022,schmidt1995}\\ 
			\hline
		\end{tabular}
		\end{center}
		\caption{Table of Degree $6$ Transitive Groups, where $N(X):=\#\mathcal{F}_{6,k}(G;X)$}
		\label{tb:deg6}
	\end{table}

	Among the concentrated groups of degree $6$, there are now only three groups for which Conjecture \ref{conj:number_field_counting} is not yet known.
	\begin{itemize}
		\item The minimal index elements of 6T9 generate a normal subgroup $A_3\rtimes S_3 \normal {\rm Hol}(S_3)$, where $S_3$ acts on the alternating group by conjugation.
		\item The minimal index elements of 6T10 generate the normal subgroup $C_3^2 \rtimes C_2\normal C_3^2 \rtimes C_4$, where the $C_2$ action on $C_3^2$ is the dihedral action.
		\item The minimal index elements of 6T14 generate the alternating group $A_5\normal S_5$.
	\end{itemize}
	The minimal index elements of these groups are not abelian, and the groups themselves are not wreath products. Nevertheless, we are able to prove the weak form of Malle's conjecture for both 6T9 and 6T10 using Theorem \ref{thm:main_abelian_on_top}.

	We are also able to prove the weak form of Malle's conjecture for the non-concentrated group 6T7. This group is abstractly isomorphic to $S_4$, and our ability to prove such a strong result is due to known results for the first moment of $2$-torsion in the class group of cubic extensions. This example demonstrates an important idea: while our methods are only able to give an asymptotic for concentrated groups, the main results of this paper still yield very good upper bounds for many non-concentrated groups. 

	We give the proofs of the new results below. For all but 6T8, 6T9, and 6T10, the proof is either a direct citation of previous references or a direct application of one of the Theorems or Corollaries in the introduction. The groups 6T9 and 6T10 are the smallest examples in which we apply our inductive techniques twice to prove the best possible upper bound.

	\begin{enumerate}[(6T1)]
		\item\label{it:6T1} $C_6$ is abelian, so Conjecture \ref{conj:number_field_counting} follows  from \cite{wright1989}.
		
		\item\label{it:6T2} $S_3$ in degree $6$, Conjecture \ref{conj:number_field_counting} was proven in \cite{bhargava-wood2007,belabas2010discriminants}.
		
		\item\label{it:6T3} $S_3\times C_2$, Conjecture \ref{conj:number_field_counting} was proven over $\Q$ in \cite{masri_thorne_tsai_wang2020}. 
		
		\item\label{it:6T4} $A_4$ in degree $6$, Conjecture \ref{conj:number_field_counting} will follow from Corollary \ref{cor:main_abelian_on_top_imprimitive}. $A_4$ in degree $6$ is an imprimitive group that is realized as a subgroup of $C_2\wr C_3$. Taking $T = C_2^3 \cap A_4 = V_4\normal A_4$, we find that $a(T) = 2$. There is only a single nontrivial conjugacy class in $T$, so $b(K,T(\pi)) = 1$ for any $\pi:G_K\to A_4/T=C_3$.

		It now suffices to find some $\theta < 1 = |C_2|/a(T)$ for which
		\[
			\sum_{F\in\mathcal{F}_{3,k}(C_3;X)} |\Hom(\Cl_F,C_2)| \ll X^{\theta}.
		\]
		We can do this by bounding the length of the sum and the summands independently. We appeal to the bound of $2$-torsion in the class group proven in \cite{bhargava_shankar_taniguchi_thorne_tsimerman_zhao_2020} to prove
		\begin{align*}
			\sum_{F\in\mathcal{F}_{3,k}(C_3;X)} |\Hom(\Cl_F,C_2)| & = \sum_{F\in\mathcal{F}_{3,k}(C_3;X)} |\Cl_F[2]|\\
			&\ll_{[k:\Q],\epsilon} \sum_{F\in\mathcal{F}_{3,k}(C_3;X)} |\disc(F/\Q)|^{\frac{1}{2}-\frac{1}{2[F:\Q]}+\epsilon}\\
			&\ll_{[k:\Q],\epsilon} X^{\frac{1}{2} - \frac{1}{6[k:\Q]}+\epsilon} \#\mathcal{F}_{3,k}(C_3;X).
		\end{align*}
		It is known that $\#\mathcal{F}_{3,k}(C_3;X)\ll X^{1/2+\epsilon}$, so that we can take $\theta = 1 - \frac{1}{6[k:\Q]}+\epsilon$. For $\epsilon$ sufficiently small it follows that $\theta < 1$, so Conjecture \ref{conj:number_field_counting} follows from Corollary \ref{cor:main_abelian_on_top_imprimitive}.

		\item\label{it:6T5} $C_3 \wr C_2$ is Kl\"uners' original counterexample to Malle's Conjecture \cite{kluners2005}. Conjecture \ref{conj:number_field_counting} is proven directly in Corollary \ref{cor:iterated_cyclic_wreath} and Corollary \ref{cor:Malle_counter_example}, although the power of $\log X$ is not explicitly computed. Theorem \ref{thm:main_abelian_on_top} shows that $b= \max_\pi b(k,C_3^2(\pi))$ where the maximum is taken over $\pi\in \Sur(G_k,C_2)$. One can directly calculate that $b(k,C_3^2(\pi)) = 1$ if $k(\zeta_3)$ is not fixed by $\ker\pi$, and $b(k,C_3^2(\pi)) = 2$ if $k(\zeta_3)$ is fixed by $\pi$. The computation is done explicitly in \cite{alberts2021} (in the example following Proposition 3.7), and is essentially the same as the computations done by Kl\"uners in \cite{kluners2005} to show that Malle's prediction is incorrect for this group.

		\item\label{it:6T6} $C_2\wr C_3$, Conjecture \ref{conj:number_field_counting} was first proven in \cite{kluners2012}. This is also a subcase of Corollary \ref{cor:iterated_cyclic_wreath}.

		\item\label{it:6T7} $S_4$ in its first degree $6$ representation has a maximal abelian normal subgroup $T = V_4$, which has $a(T) = 2$. The lower bound was proven in \cite[Corollary 1.7]{alberts2021}, and is conjecturally sharp up to logs.

		For the upper bound, we appeal to Corollary \ref{cor:main_abelian_on_top_imprimitive} as $S_4$ in this representation is realized as a subgroup of $C_2\wr S_3$ with $T=C_2^3 \cap G$. It now suffices to find some $\theta$ for which
		\[
			\sum_{F\in\mathcal{F}_{3,k}(S_3;X)} |\Hom(\Cl_F,C_2)| \ll X^{\theta}.
		\]
		This will follow from \cite[Theorem 2]{bhargava-shankar-wang2015} on the number of $S_4$-extensions with restricted local behavior. Indeed, the sum
		\[
			\sum_{F\in\mathcal{F}_{3,k}(S_3;X)} |\Cl_F[2]|
		\]
		is precisely equal to $N_{4,\Sigma}(k,X)$, where $\Sigma$ is the set of local specifications that requires $L_p/k_p$ to have no ramification of type $(12)(34)$. Thus, it follows directly from \cite[Theorem 2]{bhargava-shankar-wang2015} that we may take $\theta = 1 = |C_2|/a(T)$.

		\item\label{it:6T8} $S_4$ in its second degree $6$ representation has a maximal abelian normal subgroup $T = V_4$, which has $a(T) = 2$. There is only a single nontrivial conjugacy class in $T$, so $b(K,T(\pi)) = 1$ for any $\pi:G_K\to S_4/T=S_3$. This group is also imprimitive (realized as a subgroup of $S_3\wr C_2$), however in this case $T$ is not an imprimitive kernel. For this reason, we need to use the full strength of Theorem \ref{thm:main_abelian_on_top}.
		
		By calculating the indices of elements in \verb^Magma^, we determine that
		\begin{align*}
			\ind_6((1234)) &= 3\\
			\ind_6((123)) &= 4\\
			\ind_6((12)) &= 3\\
			\ind_6((12)(34)) &= 2
		\end{align*}
		Taking the quotient to $S_4/T = S_3$, we compute the pushforward indices as in \eqref{eq:pushforward_index} to be
		\begin{align*}
			q_*\ind_6((12)) &= 3\\
			q_*\ind_6((123)) &= 4.
		\end{align*}
		This agrees with the degree $6$ representation for $S_3$, so we conclude that $q_*\disc_{6T8}(\pi) \asymp \disc_{6T2}(\pi)$. Thus
		\[
			q_*\Sur(G_k,6T8;X) \ll \Sur(G_k,6T2;X) \ll X^{1/3}.
		\]
		The bounds for $|H^1_{ur}(k,T(\pi))|$ are entirely analogous to 6T4, so we compute
		\begin{align*}
			\sum_{\pi\in q_*\Sur(G_k,6T8;X)} |H^1_{ur}(k,T(\pi))| &\ll \sum_{\pi\in q_*\Sur(G_k,6T8;X)} |\disc_{3T2}(\pi)|^{\frac{1}{2} - \frac{1}{6[k:\Q]}+\epsilon}\\
			&\ll \sum_{\pi\in \Sur(G_k,6T2;X)} |\disc_{3T2}(\pi)|^{\frac{1}{2} - \frac{1}{6[k:\Q]}+\epsilon}\\
			&\ll \sum_{F\in \mathcal{F}_{6,k}(6T2;X)} |\disc_{3T2}(F^{\Stab_{3T2}(1)}/\Q)|^{\frac{1}{2} - \frac{1}{6[k:\Q]}+\epsilon}
		\end{align*}
		This function is given entirely by the distribution of $S_3$-extensions, using a mix of the sextic and cubic discriminants. There are likely multiple ways to bound this sum, we choose to do so by partitioning according to the quadratic resolvent of $F$. This is analogous to taking $T=C_3\normal S_3$ in Theorem \ref{thm:main_abelian_on_top}.

		Suppose $M\le F$ is the quadratic resolvent. Then comparing indices implies
		\begin{align*}
			\disc_{6T2}(F/\Q) &\asymp \disc_{3T2}(F^{\rm Stab_{3T2}(1)}/\Q)^2 \disc(M/\Q).
		\end{align*}
		Partitioning by the quadratic resolvant implies
		\begin{align*}
			&\sum_{F\in \mathcal{F}_{6,k}(6T2;X)} |\disc_{3T2}(F^{\Stab_{3T2}(1)}/\Q)|^{\frac{1}{2} - \frac{1}{6[k:\Q]}+\epsilon}\\
			&\ll X^{\frac{1}{4} - \frac{1}{12[k:\Q]}+\epsilon}\sum_{F\in \mathcal{F}_{6,k}(6T2;X)} |\disc(F^{A_3}/\Q)|^{-\frac{1}{4} + \frac{1}{12[k:\Q]}+\epsilon}.
		\end{align*}
		Let $\pi_M:G_k\to \Gal(M/k) \cong C_2$ be the (unique) surjective homomorphism corresponding to a quadratic field $M$ and let $q:S_3\to S_3/A_3 \cong C_2$ be the quotient map. Partitioning the sum with respect to the fibers of $q$ gives, up to a constant multiple,
		\[
			X^{\frac{1}{4} - \frac{1}{12[k:\Q]}+\epsilon}\sum_{M\in \mathcal{F}_{2,k}(C_2;X^{1/3})} |\disc(M/\Q)|^{-\frac{1}{4} + \frac{1}{12[k:\Q]}+\epsilon}\#\{\psi\in q_*^{-1}(\pi_M) : |\disc_{6T2}(\psi)|\le X\},
		\]
		For each fixed $M$, we can bound this uniformly in terms of Theorem \ref{thm:uniformity} to get
		\begin{align*}
			&\ll X^{\frac{1}{4} - \frac{1}{12[k:\Q]}+\epsilon}\sum_{M\in \mathcal{F}_{2,k}(C_2;X^{1/3})} |\disc(M/\Q)|^{-\frac{1}{4} + \frac{1}{12[k:\Q]}+\epsilon}\frac{|\Cl_M[3]|}{\disc(M/\Q)^{3/4-\epsilon}} X^{1/4+\epsilon},
		\end{align*}
		noting that $A_3\subseteq 6T2$ has $a(A_3) = 4$ and $q_*\disc(\pi) \asymp \disc(M/\Q)^3$. Simplifying, this is given by
		\begin{align*}
			&\ll X^{\frac{1}{2} - \frac{1}{12[k:\Q]}+\epsilon}\sum_{M\in \mathcal{F}_{2,k}(C_2;X^{1/3})} |\disc(M/\Q)|^{-1 + \frac{1}{12[k:\Q]}+\epsilon}|\Cl_M[3]|\\
			&\ll X^{\frac{1}{2} - \frac{2}{36[k:\Q]}+\epsilon}\sum_{M\in \mathcal{F}_{2,k}(C_2;X^{1/3})} |\disc(M/\Q)|^{-1}|\Cl_M[3]|.
		\end{align*}
		Datskovsky--Wright \cite{DW88} proved that the average size of $3$-torsion in class groups over relative quadratic extensions is constant. Thus, Abel summation implies
		\[
			X^{\frac{1}{2} - \frac{2}{36[k:\Q]}+\epsilon}\sum_{M\in \mathcal{F}_{2,k}(C_2;X^{1/3})} |\disc(M/\Q)|^{-1}|\Cl_M[3]| \ll X^{\frac{1}{2} - \frac{2}{36[k:\Q]}+\epsilon}\log X.
		\]
		
		All together, making $\epsilon$ slightly larger we conclude that
		\[
			\sum_{\pi\in q_*\Sur(G_k,6T8;X)} |H^1_{ur}(k,T(\pi))| \ll X^{\frac{1}{2} - \frac{1}{18[k:\Q]}+\epsilon},
		\]
		so we can take $\theta = \frac{1}{2} - \frac{1}{18[k:\Q]}+\epsilon$. For $\epsilon$ sufficiently small, we certainly have $\theta < 1/2 = 1/a(T)$ so that Conjecture \ref{conj:number_field_counting} follows from Theorem \ref{thm:main_abelian_on_top}(i).

		\item\label{it:6T9} Consider $T = C_3^2 \normal S_3^2$ in degree $6$. A \verb^Magma^ search indicates that $a(T) = 2$, so that the lower bound is proven by the first author in \cite[Corollary 1.7]{alberts2021}. This is conjecturally sharp.

		We apply Theorem \ref{thm:main_abelian_on_top} to prove the upper bound, so that it suffices to show
		\[
			\sum_{\pi\in q_*\Sur(G_k,6T9;X)}|H^1_{ur}(k,T(\pi))| \ll X^{1/2+\epsilon},
		\]
		as $\theta = 1/2+\epsilon > 1/2 = 1/a(T)$.

		By calculating the indices of all elements outside $T$, we conclude that $q_*\ind_{6T9}(g) \ge 2 = \ind_{V_4}(g)$ for all $g\in V_4\cong S_3^3/ T$. Thus, $q_*\disc_{6T9}\asymp \disc_{V_4}$ and we can prove
		\[
			q_*\Sur(G_k,6T9;X) \ll \Sur(G_k,V_4;X) \ll X^{1/2}.
		\]
		We now bound the size of the summands by taking Lemma \ref{lem:inductive_H1ur_bound} with $M$ the diagonal subgroup of $T$. This is a core free subgroup, and is normalized by the $H\le S_3^2$ the pullback of the diagonal in the quotient $\Delta \subseteq V_4 = S_3^2/T$. Thus
		\begin{align*}
			|H^1_{ur}(k,T(\pi))| &\ll |H^1_{ur}(F(\pi)^{\Delta}, C_3)|\disc(F/\Q)^{\epsilon}\\
			&\ll |\Cl_{F(\pi)^{\Delta}}[3]|\disc(F/\Q)^{\epsilon},
		\end{align*}
		where $F(\pi)$ is the field of definition of $M$ (so necessarily a $V_4$-extension). We then evaluate
		\begin{align*}
			\sum_{\pi\in q_*\Sur(G_k,6T9;X)}|H^1_{ur}(k,T(\pi))| &\ll \sum_{\pi\in \Sur(G_k,V_4;X)}|\Cl_{F(\pi)^{\Delta}}[3]|\disc(F/\Q)^{\epsilon}\\
			&\ll X^{\epsilon}\sum_{\pi\in \Sur(G_k,V_4;X)}|\Cl_{F(\pi)^{\Delta}}[3]|.
		\end{align*}
		We decompose the sum over $V_4$-extensions by fibering over $\Delta\normal V_4$. Let $\overline{q}:V_4\to V_4/\Delta$ be the quotient map. This implies
		\begin{align*}
			\sum_{\pi\in \Sur(G_k,V_4;X)}|\Cl_{F(\pi)^{\Delta}}[3]| &= \sum_{\overline{\pi}\in \overline{q}_*\Sur(G_k,V_4;X)}|\Cl_{F(\overline{\pi})}[3]| \#\{\psi\in \overline{q}_*^{-1}(\overline{\pi}) : |\disc_{V_4}(\psi)|\le X\},
		\end{align*}
		We directly apply Theorem \ref{thm:uniformity} to bound the summands with uniform dependence on $\overline{\pi}$, proving that
		\begin{align*}
			\sum_{\pi\in \Sur(G_k,V_4;X)}|\Cl_{F(\pi)^{\Delta}}[3]| &\ll_{[k:\Q],\epsilon}\sum_{\overline{\pi}\in \overline{q}_*\Sur(G_k,V_4;X)}|\Cl_{F(\overline{\pi})}[3]| \frac{|H^1_{ur}(k,\Delta(\overline{\pi}))|}{|\overline{q}_*\disc(\overline{\pi})|^{1/2-\epsilon}}X^{1/2+\epsilon}.
		\end{align*}
		Now, $\Delta\normal V_4$ is a central subgroup, so $\Delta(\overline{\pi}) = \Delta$ carries the trivial action and the numerator can be bounded by $|\Cl_k[2]|\ll_k 1$. For the denominator, the fact that all nonidentity elements of $V_4$ have index $2$ implies $\overline{\pi}_*\disc(\overline{\pi}) = \disc(\overline{\pi})^2$. Converting to a sum over quadratic field via the Galois correspondence, we have shown that up to a constant multiple
		\begin{align*}
			\sum_{\pi\in q_*\Sur(G_k,6T9;X)}|H^1_{ur}(k,T(\pi))| \ll X^{1/2+\epsilon}\sum_{F\in \mathcal{F}_{2,k}(C_2;X^{1/2})} |\Cl_{F}[3]|\cdot |\disc(F/k)|^{-1+\epsilon}.
		\end{align*}
		Datskovsky--Wright \cite{DW88} proved that the average $3$-torsion of the class groups of relative quadratic extensions is constant. Abel summation then implies
		\begin{align*}
			\sum_{\pi\in q_*\Sur(G_k,6T9;X)}|H^1_{ur}(k,T(\pi))| &\ll X^{1/2+2\epsilon}.
		\end{align*}
		Replacing $\epsilon$ with $\epsilon/2$ concludes the proof.

		\item\label{it:6T10} The group $6T10$ is $C_3^2\rtimes C_4$ with the faithful action. Take $T = C_3^2 \normal 6T10$. A \verb^Magma^ search indicates that $a(T) = 2$, so that the lower bound is proven by the first author in \cite[Corollary 1.7]{alberts2021}. This is conjecturally sharp.

		We apply Theorem \ref{thm:main_abelian_on_top} to prove the upper bound, so that it suffices to show
		\[
			\sum_{\pi\in q_*\Sur(G_k,6T10;X)}|H^1_{ur}(k,T(\pi))| \ll X^{1/2+\epsilon},
		\]
		as $\theta = 1/2+\epsilon > 1/2 = 1/a(T)$.

		By calculating the indices of all elements outside of $T$, we conclude that
		\[
			q*\ind_{6T10}(g) =\begin{cases}
				4 & |\langle g\rangle| = 4\\
				2 & |\langle g\rangle| = 2.
			\end{cases}
		\]
		As the quotient group $C_3^2\rtimes C_4/C_3^2 = C_4$ is abelian (and therefore nilpotent), we use the upper bounds proven in \cite{alberts2020} for nilpotent groups ordered by arbitrary invariants to conclude
		\[
			q_*\Sur(G_k,6T10;X) \ll X^{1/2+\epsilon}.
		\]
		Next, we apply Lemma \ref{lem:inductive_H1ur_bound} to bound the summands with core-free subgroup $M=C_3\times 1 \le T$. $M$ is stabilized by $H=C_2\le C_4$ under the semidirect product action. Thus,
		\begin{align*}
			|H^1_{ur}(k,T(\pi))| &\ll |H^1_{ur}(F(\pi)^{C_2},C_3)|\cdot |\disc(F(\pi)/\Q)|^{\epsilon}\\
			&\ll |\Cl_{F(\pi)^{C_2}}[3]|\cdot |\disc(F(\pi)/\Q)|^{\epsilon}\\
			&\ll |\disc(F(\pi)^{C_2}/\Q)|^{1/2+\epsilon}|\disc(F(\pi)/\Q)|^{\epsilon},
		\end{align*}
		where $F(\pi)$ is the field fixed by $\ker\pi$ for $\pi:G_k\to C_4$, and $F(\pi)^{C_2}$ is the quadratic subfield. Thus
		\begin{align*}
			\sum_{\pi\in q_*\Sur(G_k,6T10;X)} |H^1_{ur}(k,T(\pi))|&\ll X^{\epsilon}\sum_{\pi\in \Sur_{q_*\disc}(G_k,C_4;X)} |\disc(F(\pi)^{C_2}/\Q)|^{1/2+\epsilon}.
		\end{align*}
		Next, we partition the sum according to the normal subgroup $C_2\normal C_4$, similar to 6T9. Let $\overline{q}:C_4\to C_4/C_2$ be the quotient map. This gives an upper bound of the form
		\begin{align*}
			\ll \sum_{\bar{\pi}\in \bar{q}_*\Sur_{q_*\disc}(G_k,C_4;X)} |\disc(\bar{\pi})|^{1/2+\epsilon}\#\{\psi\in \overline{q}_*^{-1}(\overline{\pi}) : |q_*\disc(\psi)|\le X\}.
		\end{align*}
		We can directly use Theorem \ref{thm:uniformity} to give an upper bound
		\begin{align*}
			&\ll \sum_{\bar{\pi}\in \bar{q}_*\Sur_{q_*\disc}(G_k,C_4;X)} |\disc(\bar{\pi})|^{1/2+\epsilon}\frac{|H^1_{ur}(k,C_2(\bar{\pi}))|}{|\bar{q}_*q_*\disc(\bar{\pi})|^{1/2-\epsilon}}X^{1/2+\epsilon}.
		\end{align*}
		The subgroup $C_2\normal C_4$ is central, and therefore $C_2(\bar{\pi})=C_2$ carries the trivial action. Thus, the numerator is bounded above by $|\Cl_k[2]|\ll_k 1$. The group $C_2$ has only one nontrivial element, so by checking the weight of that element we find that
		\[
			\bar{q}_*q_*\disc_{6T10}(\overline{\pi}) \asymp \disc_{C_2}(\bar\pi)^4.
		\]
		Thus, we can finally bound the sum by
		\begin{align*}
			\sum_{\pi\in q_*\Sur(G_k,6T10;X)} |H^1_{ur}(k,T(\pi))|&\ll X^{1/2+\epsilon}\sum_{\bar{\pi}\in \Sur(G_k,C_2;X^{1/4})} |\disc(\bar{\pi})|^{-3/2+\epsilon}.
		\end{align*}
		Via the Galois correspondence and Abel summation, we conclude that
		\begin{align*}
			\sum_{\pi\in q_*\Sur(G_k,6T10;X)} |H^1_{ur}(k,T(\pi))|&\ll X^{1/2+\epsilon}\sum_{F\in\mathcal{F}_{2,k}(C_2;X^{1/4})} |\disc(F/k)|^{-3/2+\epsilon}\\
			&\ll X^{1/2+\epsilon}.
		\end{align*}

		\item\label{it:6T11} $C_2\wr S_3$, Conjecture \ref{conj:number_field_counting} was first proven in \cite{kluners2012}. This is also a subcase of Corollary \ref{cor:cyclic_wreath}.

		\item\label{it:6T12} For $A_5$ in degree $6$, we check the indices in \verb^Magma^ to prove that
		\[
			\disc_{5T4}\ll \disc_{6T12} \ll \disc_{5T4}^2,
		\]
		where $5T4$ is the group $A_5$ in the degree $5$ representation. Thus,
		\[
			\#\mathcal{F}_{5,k}(A_5;X^{1/2}) \ll \#\mathcal{F}_{6,k}({\rm 6T12};X) \ll \#\mathcal{F}_{5,k}(A_5;X).
		\]
		The lower bound is then given by \cite{pierce_turnage-butterbaugh_wood2021} and the upper bound by \cite{bhargava-shankar-wang2015}.

		\item\label{it:6T13} $S_3\wr C_2$ is proven directly by Corollary \ref{cor:S3_wreath}.

		\item\label{it:6T14} For $S_5$ in degree $6$, we check the indices in \verb^Magma^ to prove that
		\[
			\disc_{5T5}\ll \disc_{6T14} \ll \disc_{5T5}^2,
		\]
		where $5T5$ is the group $S_5$ in the degree $5$ representation. Thus,
		\[
			\#\mathcal{F}_{5,k}(S_5;X^{1/2}) \ll \#\mathcal{F}_{6,k}({\rm 6T14};X) \ll \#\mathcal{F}_{5,k}(S_5;X).
		\]
		The asymptotic for $\#\mathcal{F}_{5,k}(S_5;X)\asymp X$ is given by \cite{bhargava-shankar-wang2015}.

		\item\label{it:6T15} For $A_6$ in degree $6$ \cite{pierce_turnage-butterbaugh_wood2021} prove the best known lower bound while Lemke Oliver gives the best known upper bound \cite{LO-uniform}.

		\item\label{it:6T16} For $S_6$ in degree $6$ \cite{Bhargava2022} proves the best known lower bound while Schmidt's trivial bound \cite{schmidt1995} is the best known upper bound.
	\end{enumerate}

\subsection{Nilpotent Groups}

	A \verb^Magma^ search reveals numerous new groups for which Corollary \ref{cor:nilpotent} proves Conjecture \ref{conj:number_field_counting}. In particular, all $2{,}685{,}340$ groups in degree 32 in Theorem \ref{cor:compute} for which we prove Conjecture \ref{conj:number_field_counting} are nilpotent.

	In the introduction, we referred to ${\rm Hol}(D_4) = D_4\rtimes \Aut(D_4)$ in degree $8$ as one such new example, which we elaborate on before proving Corollary \ref{cor:nilpotent}. Theorem \ref{thm:main_abelian_on_top} implies the asymptotic
	\[
		\#\mathcal{F}_{8,k}({\rm Hol}(D_4);X) \sim c(k,{\rm Hol}(D_4)) X^{1/2}\log X
	\]
	for some positive constant $c(k,{\rm Hol}(D_4)) > 0$. This follows from a \verb^Magma^ search through the elements of ${\rm Hol}(D_4)$, expressed as \verb|TransitiveGroup(8,26)|. This search confirms that $a({\rm Hol}(D_4)) = 2$, that there are two conjugacy classes of minimum index elements, and that
	\[
	T:=\langle g\in {\rm Hol}(D_4) : \ind(g) = 2\rangle \cong C_2^3
	\]
	is abelian. Given that the cyclotomic character acts trivially on group elements of order $2$, it follows that the power of $\log X$ given by Theorem \ref{thm:main_abelian_on_top} agrees with Malle's original prediction
	\[
	\max_{\pi} b(k,T(\pi)) = b(k,{\rm Hol}(D_4)) = 2.
	\]

	We now prove Corollary \ref{cor:nilpotent} via Theorem \ref{thm:main_abelian_on_top}. 

	\begin{proof}[Proof of Corollary \ref{cor:nilpotent}]
	If $G$ is a nilpotent group whose minimal elements all commute, choose
	\[
	T = \langle g\in G-\{1\} : \ind(g) = a(G)\rangle.
	\]
	This is a normal subgroup of $G$, as the index function is constant on conjugacy classes. It is abelian by assumption, so we can apply Theorem \ref{thm:main_abelian_on_top}.

	Alberts proves in \cite{alberts2020} a generalized version of Malle's predicted weak upper bound for nilpotent groups to arbitrary admissible invariant, including the pushforward discriminant. It follows from \cite[Theorem 1.6]{Kluners2022} or \cite[Corollary 5.2]{alberts2020} that $G$ nilpotent implies
	\[
		\#q_*\Sur(G_k,G;X) \le \#\Sur_{q_*\disc_G}(G_k,G/T;X) \ll_{\epsilon} X^{1/a(G-T)+\epsilon}.
	\]
	Given that $G$ is nilpotent, it follows that $T(\pi)$ is a nilpotent module for any $\pi:G_K\to G$. The inductive class group bounds in Corollary \ref{cor:inductive_H1ur_bounds}(i) then imply
	\begin{align*}
	\sum_{\pi\in q_*\Sur(G_k,G;X)} |H^1_{ur}(k,T(\pi))| \ll_{\epsilon} \sum_{\pi\in q_*\Sur(G_k,G;X)} |\disc(F(\pi)/\Q)|^{\epsilon}
	\end{align*}
	for $F(\pi)$ the field fixed by $\pi^{-1}(\Stab_G(1)T)$. As we have the freedom to choose $\epsilon$ as small as we like and $\fp\mi \disc(F(\pi)/\Q)$ if and only if $\fp\mid q_*\disc(\pi)$, it follows that there exists an $\epsilon' > 0$ for which we can bound
	\[
	|\disc(F(\pi)/\Q)|^{\epsilon'} \ll_{\epsilon} |q_*\disc(\pi)|^{\epsilon}.
	\]
	Thus, we conclude that
	\begin{align*}
		\sum_{\pi\in q_*\Sur(G_k,G;X)} |H^1_{ur}(k,T(\pi))| \ll_{\epsilon} \sum_{\pi\in q_*\Sur(G_k,G;X)} |q_*\disc_G(\pi)|^{\epsilon} \ll_{\epsilon} X^{1/a(G-T)+\epsilon}
	\end{align*}
	so that we can take $\theta = 1/a(G-T)+\epsilon$ in Theorem \ref{thm:main_abelian_on_top}.

	It is clear that $\theta < 1/a(G)$ by our choice of $T$, so Theorem \ref{thm:main_abelian_on_top}(i) implies Corollary \ref{cor:nilpotent}.
	\end{proof}

\subsection{Wreath Products}

	We present a proof of Corollary \ref{cor:cyclic_wreath} in this section. Our methods give us access to wreath products by abelian groups of slightly larger rank as well. We give a complete statement below:

	\begin{corollary}\label{cor:abelian_generalized_wreath}
		Let $G=A \wr B\subseteq S_{nm}$ be a wreath product of an abelian group $A$ of cardinality $n$ with a transitive subgroup $B$ of degree $m$.

		Suppose $k$ is a number field which has at least one $B$-extension and for which there exists a $\delta > 0$ such that
		\[
			\#\mathcal{F}_{m,k}(B;X) \ll_{k} X^{1 + \frac{1}{\ell-1} - \frac{d(A)}{2}-\delta},
		\]
		where $\ell$ is the smallest prime dividing $n$ and $d(A)$ the minimum number of generators of $A$ as an abstract group. Then Conjecture \ref{conj:number_field_counting} holds for $G$ over $k$.
	\end{corollary}

	It is clear that Corollary \ref{cor:cyclic_wreath} is an immediate consequence by $d(C_n) = 1$. Moreover, as long as $d(A) < 2 + \frac{2}{\ell-1}$ we can give further examples by taking $B$ to be nilpotent with $a(B)$ sufficiently large.

	\begin{proof}
	This will follow directly from Corollary \ref{cor:main_abelian_on_top_imprimitive}, as $A\wr B$ is an imprimitive group.

	Using Minkowski's bound on the size of the class group we find that
	\begin{align*}
		\sum_{F\in \mathcal{F}_{m,k}(B;X)} |\Hom(\Cl_F,A)| &\ll_{k,\epsilon} \sum_{F\in \mathcal{F}_{m,k}(B;X)} |\disc(F/k)|^{d(A)/2+\epsilon}\\
		&\ll_{k,\epsilon} X^{d(A)/2+\epsilon}\#\mathcal{F}_{m,k}(B;X).
	\end{align*}
	By assumption, it follows that
	\begin{align*}
		\sum_{F\in \mathcal{F}_{m,k}(B;X)} |\Hom(\Cl_F,A)| &\ll_{k,\epsilon} X^{1 + \frac{1}{\ell-1}-\delta}
	\end{align*}
	for some $\delta>0$ (where we choose $\epsilon$ sufficiently small to get canceled out). Thus, we can take $\theta = 1 + \frac{1}{\ell-1} - \delta$.

	The minimum index elements of $A^m$ are precisely the conjugates of $(a_1,1,1,...,1)$ with $a_1\in A$ a minimum index element. Thus,
	\[
		\frac{|A|}{a(A^m)} = \frac{|A|}{a(A)} = \frac{|A|}{|A|\frac{\ell-1}{\ell}} = \frac{\ell}{\ell-1} = 1 +\frac{1}{\ell-1}.
	\]
	We have now shown that $\theta < |A|/a(A^m)$, so the result follows from Corollary \ref{cor:main_abelian_on_top_imprimitive}(i).
	\end{proof}

\subsection{Iterated Wreath Products of Cyclic Groups}

	We address Corollary \ref{cor:iterated_cyclic_wreath} separately, as this example showcases the inductive power of our methods. We first prove the following upper bound results for iterated wreath products, which we will use as input in Corollary \ref{cor:main_abelian_on_top_imprimitive}.

	\begin{corollary}\label{cor:upper_bound_iterated_wreath}
		Let $k$ be a number field and $G = C_{n_1}\wr C_{n_2} \wr \cdots \wr C_{n_r}$. Then
		\[
		\#\Sur(G_k,G;X)\ll \begin{cases}
			X^{\frac{3}{2n_1}+\epsilon} & 2\nmid n_1\\
			X^{\frac{2}{n_1}} & 2\mid n_1.
		\end{cases}
		\]
	\end{corollary}

	\begin{proof}
	We prove this by inducting on $k$. If $k=1$, these groups are abelian and the result follows from \cite{wright1989}. For $k>1$, write $G = C_{n_1}\wr H$ where $H$ is an iterated wreath product of cyclic groups of length $k-1$. For the sake of convenience, we can weaken the inductive hypothesis to
	\[
		\#\mathcal{F}_{n_2\cdots n_k,k}(H;X) \ll X.
	\]

	$G$ is certainly an imprimitive group, so we apply Corollary \ref{cor:main_abelian_on_top_imprimitive}. Minkowski's bound and the inductive hypothesis imply that
	\begin{align*}
		\sum_{F\in \mathcal{F}_{n_2\cdots n_k,k}(H;X)} |\Hom(\Cl_F,C_{n_1})| &\ll X^{1/2}\#\mathcal{F}_{n_2\cdots n_k,k}(H;X) \ll X^{3/2},
	\end{align*}
	We will compare this to
	\[
		\frac{n_1}{a(C_{n_1})} = \frac{\ell}{\ell-1}
	\]
	for $\ell$ the smallest prime dividing $n_1$.

	If $n_1$ is odd, then $\theta = 3/2 \ge \frac{\ell}{\ell-1} = \frac{n_1}{a(C_{n_1})}$. Thus, Corollary \ref{cor:main_abelian_on_top_imprimitive}(ii) implies
	\[
		\#\mathcal{F}_{n_1\cdots n_k,k}(G;X) \ll X^{\frac{\theta}{n_1}+\epsilon} = X^{\frac{3}{2n_1}+\epsilon}.
	\]
	Meanwhile if $n_1$ is even, then $\theta = 3/2 < 2 = \frac{n_1}{a(C_{n_1})}$. Thus, Corollary \ref{cor:main_abelian_on_top_imprimitive}(i) implies
	\[
		\#\mathcal{F}_{n_1\cdots n_k,k}(G;X) \ll X^{\frac{1}{a(C_{n_1})}+\epsilon} = X^{\frac{2}{n_1}+\epsilon}.
	\]
	\end{proof}

	We can now prove the shortest cases of Corollary \ref{cor:iterated_cyclic_wreath} as a consequence of Corollary \ref{cor:cyclic_wreath}.

	\begin{proof}[Proof of Corollary \ref{cor:iterated_cyclic_wreath}(a,b)]
	Suppose first that $n_2 > 2$. If $n_2 > 2$ is even then certainly $2/n_2 < 1/2$, while if $n_2> 2$ is odd $3/2n_2 < 1/2$. The result then follows from Corollary \ref{cor:cyclic_wreath}, as Corollary \ref{cor:upper_bound_iterated_wreath} implies
	\begin{align*}
		\#\mathcal{F}_{n_2\dots n_r,k}(B;X) \ll X^{1/2} \le X^{\frac{1}{2} + \frac{1}{\ell-1}-\delta}
	\end{align*}
	for $\ell$ the smallest prime dividing $n_1$. In fact, the same argument applies if $n_1=n_2=2$ as part of (c).

	If $n_1,n_2,...,n_r$ are all powers of $2$, then $G$ is a $2$-group (and therefore nilpotent) with minimum index elements landing in $C_{n_1}^{n_2\cdots n_r}$ and the result follows from Corollary \ref{cor:nilpotent}.
	\end{proof}

	The remaining cases of Corollary \ref{cor:iterated_cyclic_wreath} with $n_2 = 2$ require a closer study of $|H^1_{ur}(k,T(\pi))|$, which we give in the following lemma:

	\begin{lemma}\label{lem:odd_part_inductive_H1_bounds}
		Let $F/k$ be a $G=C_2\wr B$-extension in degree $2m$. Then
		\[
			|H^1_{ur}(k,{\rm Ind}_F^k(C_{n}))| \ll_{|G|,\epsilon} |\Cl_E[2^{\nu_2(n)}]| \cdot |\Cl_E[2]|^{\nu_2(n)} \cdot |\Cl_F[n_{\rm odd}]| \cdot |\disc(F/\Q)|^{\epsilon},
		\]
		where $\nu_2(n)$ is the order to which $2$ divides $n$, $E$ is the index two subfield of $F/k$ fixed by the normal subgroup $C_2^m\normal C_2\wr B$, and $n_{\rm odd}$ is the odd part of $n$.
	\end{lemma}

	\begin{proof}
	Induced modules and cohomology groups respect direct sum decompositions, so we can write
	\[
		H^1_{ur}(k,{\rm Ind}_F^k(C_{n})) = H^1_{ur}(k,{\rm Ind}_F^k(C_{2^{\nu_2(n)}}))\oplus H^1_{ur}(k,{\rm Ind}_F^k(C_{n_{\rm odd}})).
	\]
	We will bound the two factors separately. Lemma~\ref{lem:induced} implies
	\[
		|H^1_{ur}(k,{\rm Ind}_F^k(C_{n_{\rm odd}}))| =|\Hom(\Cl_F,C_{n_{\rm odd}})| = |\Cl_F[n_{\rm odd}]|.
	\]
	For the even factor, we need to further decompose the module. Taking $E$ to be the index $2$ subfield of $F/k$, we have
	\[
		{\rm Ind}_F^k(C_{2^{\nu_2(n)}}) = {\rm Ind}_F^E({\rm Ind}_E^k(C_{2^{\nu_2(n)}}))).
	\]
	Lemma \ref{lem:induced} gives that
	\[
		|H^1_{ur}(k,{\rm Ind}_F^k(C_{2^{\nu_2(n)}}))| = |H^1_{ur}(E,{\rm Ind}_F^E(C_{2^{\nu_2(n)}}))| .
	\]
	There is an isomorphism of abstract groups ${\rm Ind}_F^E(C_{2^{\nu_2(n)}})\cong C_{2^{\nu_2(n)}}\times C_{2^{\nu_2(n)}}$, where $G_E$ acts by permuting the coordinates. The diagonal subgroup is then a submodule with the trivial action. Taking $M$ to be this diagonal in Lemma \ref{lem:inductive_H1ur_bound} implies
	\begin{align*}
		|H^1_{ur}(E,{\rm Ind}_F^E(C_{2^{\nu_2(n)}}))| &\ll_{|G|,\epsilon} |H^1_{ur}(E,C_{2^{d_2}})|\cdot |H^1_{ur}(E,C_{2^{d_2}}(-1))| \cdot |\disc(F/\Q)|^{\epsilon}\\
		&\ll_{|G|,\epsilon} |\Cl_E[2^{\nu_2(n)}]| \cdot |H^1_{ur}(E,C_{2^{d_2}}(-1))|\cdot |\disc(F/\Q)|^{\epsilon},
	\end{align*}
	where $\Gal(F/E)$ acts on $C_{2^{\nu_2(n)}}(-1)$ via the dihedral action $\sigma.a=a^{-1}$. This group has a central subgroup isomorphic to $C_2$, with quotient $C_{2^{\nu_2(n)-1}}(-1)$. Iterating Lemma \ref{lem:inductive_H1ur_bound} with $M$ being this central subgroup implies
	\begin{align*}
		|H^1_{ur}(k,{\rm Ind}_F^k(C_{2^{d_2}}))| &\ll_{|G|,\epsilon} |\Cl_E[2^{\nu_2(n)}]| \cdot |H^1_{ur}(E,C_{2})|^{\nu_2(n)}\cdot |\disc(F/\Q)|^{\epsilon}\\
		&\ll_{|G|,\epsilon} |\Cl_E[2^{\nu_2(n)}]| \cdot |\Cl_E[2]|^{\nu_2(n)}\cdot |\disc(F/\Q)|^{\epsilon}.
	\end{align*}
	This concludes the proof.
	\end{proof}

	We can now prove the remaining cases of Corollary \ref{cor:iterated_cyclic_wreath}.

	\begin{proof}[Proof of Corollary \ref{cor:iterated_cyclic_wreath}(c,d)]
	We take $n_2=2$, and write $n_1 = 2^{d_2}3^{d_3}$. We will prove the result using Theorem \ref{thm:main_abelian_on_top}. Take $T = C_{n_1}^{n_2\cdots n_r}$ as a normal subgroup in $G$ and take $B=C_{n_3}\wr \cdots \wr C_{n_r}$ so that $G/T = C_2\wr B$. (If $r=2$, then we take $B=1$). We know that the wreath action realizes $T$ as the induced module ${\rm Ind}_1^{C_2\wr B}(T)$. Thus, for any $\pi\in \Sur(G_k,C_2\wr B)$
	\[
		T(\pi) = {\rm Ind}_{F(\pi)}^k(C_{n_1}),
	\]
	where $F(\pi)/k$ is the extension fixed by $\pi^{-1}(\Stab_{C_2\wr B}(1))$ and $G_{F(\pi)}$ acts trivially on $T$. Thus, we are interested in bounding
	\begin{align*}
		\sum_{\pi\in q_*\Sur(G_k,G;X)}|H^1_{ur}(k,T(\pi))| \ll_{n_1,\epsilon} \sum_{\pi\in q_*\Sur(G_k,G;X)} |H^1_{ur}(k,{\rm Ind}_{F(\pi)}^k(C_{n_1}))|\cdot|\disc(F(\pi)/\Q)|^{\epsilon}.
	\end{align*}
	By Proposition \ref{prop:imprim_beta},
	\[
		q_*\Sur(G_k,G;X) \subseteq \Sur(G_k,C_2\wr B, cX^{1/n_1})
	\]
	for some constant $c$ depending only on $[k:\Q]$ and $n_1$. Up to a constant multiple, we have bounded
	\begin{align*}
		\sum_{\pi\in q_*\Sur(G_k,G;X)}|H^1_{ur}(k,T(\pi))| \ll_{n_1,\epsilon} \sum_{F\in \mathcal{F}_{2n_3\cdots n_r,k}(C_2\wr B; cX^{1/n_1})} |H^1_{ur}(k,{\rm Ind}_F^k(C_{n_1}))|X^{\epsilon}.
	\end{align*}
	We now partition this sum according to the index $2$-subfield $E$ fixed by $C_2^{n_3\cdots n_r}$, yielding
	\begin{align*}
		\ll_{n_1,\epsilon}\sum_{E\in \mathcal{F}_{n_3\cdots n_r,k}(B; cX^{1/2n_1})}\sum_{F\in \mathcal{F}_{2,E}(C_2;c^{1/2}X^{1/n_1}/|\disc(E/\Q)|^2)} |H^1_{ur}(k,{\rm Ind}_F^k(C_{n_1}))|X^{\epsilon}.
	\end{align*}
	We can bound this above using Lemma \ref{lem:odd_part_inductive_H1_bounds}.
	\begin{align*}
		\ll_{|G|,\epsilon}\sum_{E\in \mathcal{F}_{n_3\cdots n_r,k}(B; cX^{1/2n_1})}\sum_{F\in \mathcal{F}_{2,E}(C_2;c^{1/2}X^{1/n_1}/|\disc(E/\Q)|^2)} |\Cl_E[2^{d_2}]|\cdot |\Cl_E[2]|^{d_2} \cdot |\Cl_F[3^{d_3}]| \cdot X^{\epsilon}.
	\end{align*}
	If $d_3=0$, we can bound the sum over $F$ directly by
	\[
		\#\mathcal{F}_{2,E}(C_2;Y) \ll_{k,[E:k],\epsilon} |\disc(E/\Q)|^{\epsilon}  \cdot |\Cl_E[2]| \cdot Y.
	\]
	This follows, for example, from Theorem \ref{thm:uniformity} for $C_2$ with the trivial action. If $d_3=1$, then we can bound the sum over $F$ using \cite[Corollary 3.2]{LOWW} as
	\[
		\sum_{F\in \mathcal{F}_{2,E}(C_2;Y)} |\Cl_F[3]| \ll_{[E:\Q],\epsilon} |\disc(E/\Q)|^{1+\epsilon} \cdot |\Cl_E[2]|^{2/3} \cdot Y.
	\]
	These produce the upper bound
	\begin{align*}
		\ll_{|G|,[k:\Q],\epsilon}\sum_{E\in \mathcal{F}_{n_3\cdots n_r,k}(B; cX^{1/2n_1})}|\Cl_E[2^{d_2}]|\cdot |\Cl_E[2]|^{d_2+1-d_3/3} \cdot |\disc(E/\Q)|^{d_3-2+\epsilon} \cdot X^{1/n_1+\epsilon}
	\end{align*}
	for $d_3\in \{0,1\}$. Minkowski's bound on the size of the class group gives an upper bound of the form
	\begin{align*}
		\ll_{|G|,[k:\Q],\epsilon}\sum_{E\in \mathcal{F}_{n_3\cdots n_r,k}(B; cX^{1/2n_1})}|\disc(E/\Q)|^{\frac{1}{2}d_2 + \frac{5}{6}d_3 - 1 +\epsilon} \cdot X^{1/n_1+\epsilon}.
	\end{align*}
	Corollary \ref{cor:upper_bound_iterated_wreath} and Abel summation then produce the upper bound $X^{\max\{\theta,1/n_1\}}$ for
	\[
		\theta = \frac{1}{n_1n_3} + \frac{d_2}{4n_1} + \frac{5d_3}{12n_1} - \frac{1}{2n_1} + \frac{1}{n_1} + \epsilon.
	\]
	(Unless $r=2$ so that $B=1$, then $\theta = 0$ and we are done.) Noting that $1/n_1 < 1/a(T) = 1/a(C_{n_1})$, it suffices to determine when $\theta < 1/a(T)$. If $d_2 = 0$ and $d_3=1$ (so $n_1=3$), then we get $\theta = 1/n_3 - 1/36 < 1/2 = 1/a(T)$, as $a(T) = a(C_3)$ in this case and $n_3\ge 2$. Otherwise, $d_2 > 0$ and $1/a(T) = 2/n_1$. It follows that $\theta < 2/n_1$ if and only if
	\[
		d_2 < 6 - \frac{5d_3}{3} + \frac{4}{n_3}.
	\]
	The result then follows by plugging in each of $d_3=0$ and $d_3=1$.
	\end{proof}

\subsection{Wreath products by $S_3$ in the wreath representation}

	Corollary \ref{cor:S3_wreath} follows directly from Theorem \ref{thm:main_S3_wreath}.

	Indeed, using the assumption $\#\mathcal{F}_{m,k}(B;X) \ll X^{\frac{5}{3} + \frac{1}{3r[K:\Q]} - \delta}$ for some $\delta > 0$, we can bound the class group using the $2$-torsion bounds of \cite{bhargava_shankar_taniguchi_thorne_tsimerman_zhao_2020}
	\begin{align*}
	\sum_{F\in \mathcal{F}_{m,k}(B;X)} |\Cl_{F}[2]|^{2/3}&\ll \sum_{F\in \mathcal{F}_{m,k}(B;X)} |\disc(F/\Q)|^{\frac{1}{3} - \frac{1}{3m[K:\Q]}+\epsilon}\\
	&\ll X^{\frac{1}{3} - \frac{1}{3m[K:\Q]} + \epsilon}\#\mathcal{F}_{m,k}(B;X)\\
	&\ll X^{2 + \epsilon - \delta}.
	\end{align*}
	Taking $\epsilon < \delta$, we can choose $\theta = 2 + \epsilon - \delta < 2$. Corollary~\ref{cor:S3_wreath}(i) then follows from Theorem \ref{thm:main_S3_wreath}(i).

	For Corollary \ref{cor:S3_wreath}(ii), suppose that $B$ is primitive and that there exists some constant $\beta$ so that
		\[
			\#\mathcal{F}_{m,k}(B;X)
				\ll_{m,k} X^{\beta}.
		\]
	From \cite[Corollary~7.4]{LemkeOliver2024}, it follows that
		\[
			\sum_{F \in \mathcal{F}_{m,k}(B;X)} |\Cl_F[2]|
				\ll_{m,k,\epsilon} X^{\frac{1}{2} + \beta \cdot \left(1 - \frac{1}{2m-1}\right) + \epsilon}.
		\]
	Hence, on using H\"older's inequality, we find
		\begin{align*}
			\sum_{F \in \mathcal{F}_{m,k}(B;X)} |\Cl_F[2]|^{2/3}
				&\leq \left( \sum_{F \in \mathcal{F}_{m,k}(B;X)} |\Cl_F[2]| \right)^{2/3} \left( \sum_{F \in \mathcal{F}_{m,k}(B;X)} 1 \right)^{1/3} \\
				&\ll_{m,k,\epsilon} X^{\frac{1}{3} + \beta \cdot \frac{6m-5}{6m-3} + \epsilon}.
		\end{align*}
	It follows that if there is some $\delta > 0$ so that $\beta = \frac{5}{3} + \frac{10}{18m-15} - \delta$, then the exponent above is strictly less than $2$, and the result follows from Theorem~\ref{thm:main_S3_wreath}(i).

\subsection{Trace $0$ Semidirect Products}

	We now prove Corollary \ref{cor:trace_0_semidirect} from Corollary \ref{cor:main_abelian_on_top_imprimitive}.

	Given $G = W\rtimes B$ for $W\le \F_p^m$ the trace zero subspace and $B$ a transitive group of degree $m$ is given as an explicit subgroup of the wreath product $C_p\wr B$. Moreover, the degree $pm$ representation realizes $G$ as a permutation subgroup of $C_p\wr B$, and so $G$ is imprimitive and we can apply Corollary \ref{cor:main_abelian_on_top_imprimitive} with $A = C_p$. We then bound
	\begin{align*}
		\sum_{F\in \mathcal{F}_{m,k}(B;X)} |\Hom(\Cl_F,C_p)| &\ll_{k,\epsilon} \sum_{F\in \mathcal{F}_{m,k}(B;X)} |\disc(F/k)|^{1/2+\epsilon}\\
		&\ll_{k,\epsilon} X^{1/2+\epsilon} \#\mathcal{F}_{m,k}(B;X).
	\end{align*}
	By assumption, we have bounded
	\begin{align*}
		\sum_{F\in \mathcal{F}_{m,k}(B;X)} |\Hom(\Cl_F,C_p)| &\ll_{k,\epsilon} X^{\frac{1}{2} + \frac{1}{2(p-1)}-\delta}
	\end{align*}
	for some $\delta > 0$. Meanwhile, the minimum index elements of $W$ are permutations of $(a,a^{-1},1,1,...)$, which have index $2(p-1)$. Thus,
	\begin{align*}
		\frac{|C_p|}{a(G\cap C_p^m)} = \frac{p}{a(W)} = \frac{p}{2(p-1)} = \frac{1}{2} + \frac{1}{2(p-1)}.
	\end{align*}
	Thus, $\theta < |C_p|/a(G\cap C_p^m)$ so the result follows from Corollary \ref{cor:main_abelian_on_top_imprimitive}(i).

\section{Further Applications to Concentrated Groups}\label{sec:conditional_examples}

	We use the discussion in this section to frames our method as a general approach to Conjecture \ref{conj:number_field_counting} for any concentrated group.

	The compounding phenomenon of these methods can make it difficult to see from the statements of the main theorems exactly which groups are actually covered by our main results. Our main results are explicitly apply to groups in the following families:
	\begin{itemize}
		\item if $G=S_3\wr B$ for some transitive group $B$, then Theorem \ref{thm:main_S3_wreath} may be applicable.
		\item if $G$ is concentrated in an abelian normal subgroup, i.e. all the elements of minimal index commute with each other, then Theorem \ref{thm:main_abelian_on_top} may be applicable.
	\end{itemize}
	These are purely group theoretic conditions which, in particular, imply that $G$ is concentrated and indicate when we might expect our methods to apply in the future.

	\begin{data}
		Among the $40238$ transitive groups of degree $\le 31$,
		\begin{enumerate}[(i)]
			\item $39770$ are concentrated,
			\item $166$ are of the form $S_3\wr B$, and
			\item $30691$ are concentrated in an abelian normal subgroup.
		\end{enumerate}
	\end{data}

	On the one hand, existing conjectures suggest that improved bounds for the (average) size of class group torsion and the number of $G/T$-extension should exist which we can use as input for Theorem \ref{thm:main_S3_wreath} and Theorem \ref{thm:main_abelian_on_top}. The \emph{$\ell$-torsion conjecture} predicts that $|\Cl_F[\ell]|\ll_{\epsilon} |\disc(F/\Q)|^{\epsilon}$ as $F$ varies over any family of number fields $\mathcal{F}$ with bounded degree (this is generally regarded as a folklore conjecture, see \cite{pierce_turnage-butterbaugh_wood2021} by Pierce, Turnage--Butterbaugh, and Wood for a good introduction). Meanwhile, we already discussed Conjecture \ref{conj:WMpush} for an upper bound on the number of $G/T$-extensions order by the pushforward discriminant following from the discussion in \cite[Question 4.3]{ellenberg-venkatesh2005}.

	In the context of Theorem \ref{thm:main_S3_wreath} these conjectures would imply that
	\[
		\sum_{F\in \mathcal{F}_{m,k}(B;X)} |\Cl_F[2]|^{2/3} \ll_{m,k} X^{1/a(B)+\epsilon},
	\]
	so that we can take $\theta = 1/a(B)$. By definition $1/a(B) \le 1 < 2$, so Conjecture \ref{conj:number_field_counting} would follow.

	Similarly, in the context of Theorem \ref{thm:main_abelian_on_top}, these conjectures would imply that
	\[
		\sum_{\pi\in q_*\Sur(G_k,G;X)}|H^1_{ur}(k,T(\pi))| \ll_{\epsilon} X^{1/a(G-T) + \epsilon},
	\]
	so that we can take $\theta = 1/a(G-T)$. If $G$ is concentrated in $T$, then $\theta < 1/a(G)$ by definition an Conjecture \ref{conj:number_field_counting} would follow.

	On the one hand we argue that our method is, in principle, applicable to any concentrated group. Theorem \ref{thm:main_abelian_on_top} can be extended to allow $T$ to be nonabelian as long as Conjecture \ref{conj:twisted_number_field_counting} is known for such $T$ with sufficient uniformity. This gives a roadmap for proving Conjecture \ref{conj:number_field_counting} for any concentrated group, through proving new cases for Conjecture \ref{conj:twisted_number_field_counting}.

\section{A Cute Extension}\label{sec:cute}

	As a demonstration of the general nature of our methods, we prove the following cute result:

	\begin{theorem}\label{thm:cute}
		Let $G$ be a group with a nontrivial abelian normal subgroup and $k$ a number field which has at least one $G$-extension. Then there exists an admissible ordering of $G$-extensions, $\inv$, for which there are positive constants $b,c > 0$ such that
		\[
			\#\Sur_{\inv}(G_k,G;X) \sim c X(\log X)^{b-1}.
		\]
	\end{theorem}

	In particular, this includes all solvable groups, and many more groups besides! This is a cute application of our methods. While it showcases the general framework to which our methods apply, the admissible invariant needed for this specific result is often very far from a discriminant ordering.

	\begin{proof}
		Let $T\normal G$ be a nontrivial abelian normal subgroup. Take the admissible invariant determined by the weight function
		\[
			{\rm wt}(g) = \begin{cases}
				1 & g\in T-\{1\}\\
				\ind_{|G|}(g) & g\not\in T,
			\end{cases}
		\]
		where $\ind_{|G|}$ is the index function for $G$ in the regular representation, that is
		\[
			\inv(\pi) = \prod_{\fp} \fp^{{\rm wt}(\pi(\tau_\fp)}
		\]
		for $\tau_\fp$ a generator of tame inertia.

		We follow along the proof of Theorem \ref{thm:main_abelian_on_top}. Alberts--O'Dorney work at the level of a general admissible invariant in \cite{alberts-odorney2021}, and it follows directly from their work that
		\[
			\#\{\psi\in q_*^{-1}(\pi) : |\inv(\psi)|\le X\} \sim c X (\log X)^{b(\pi)-1}
		\]
		for some positive constants $b(\pi),c > 0$, where $\pi\in q_*\Sur(G_k,G/T)$ and $\widetilde{\pi}\in q_*^{-1}(\pi)$. In particular, $b(\pi) \le |T|$ is necessarily bounded. This verifies Theorem \ref{thm:main_pointwise}(1), where we set $a=1$, $b=\max_\pi b(\pi)$, and $c(\pi)$ is the $c$ above if $b(\pi) = b$ and $c(\pi) = 0$ otherwise.

		Next, we prove a uniform upper bound for these fibers analogous to Theorem \ref{thm:uniformity}. The start of the proof is the same: we bound
		\[
			\#\{\psi\in q_*^{-1}(\pi) : |\inv(\psi)|\le X\} \ll_{|T|} |H^1_{ur}(k,T(\pi))| \sum_{j\le X} a_j
		\]
		by the same argument as the one proving Lemma \ref{lem:upperboundMBseries}, where $a_j$ are the Dirichlet coefficients of
		\[
			{\rm MB}_{k,\inv}(T,\pi;s) = \prod_{\fp} \left(\frac{1}{|T|}\sum_{\psi_\fp\in q_*^{-1}(\pi|_{G_{k_\fp}})} |\inv(\psi_\fp)|^{-s}\right).
		\]
		The analog of Lemma \ref{lem:MBfactors} shows that
		\[
			{\rm MB}_{k,\inv}(T,\pi;s) = Q_{\inv}(T,\pi,s)L(s,\rho_{T}) G_{\inv}(T,\pi;s),
		\]
		where $Q_{\inv}(T,\pi;s)$ and $G_\inv(T,\pi;s)$ obey the same bounds as the ones listed in Lemma \ref{lem:MBfactors} with $\disc_G$ replaced by $\inv$, $\ind$ with ${\rm wt}$, and $a(T)$ replaced by $1$ everywhere they appear.

		Finally, we use the same smoothed Perron formula and shifted contour argument to bound
		\begin{align*}
			\sum_{j\le X} a_j(0) &\le \sum_{j=1}^{\infty} a_j(0) e^{1-\frac{j}{X}}\\
			&=\frac{e}{2\pi i} \int_{1+\epsilon-i\infty}^{1+\epsilon - i\infty} \hat{w}(0) \Gamma(s) X^s ds\\
			&=\underset{s=1}{\rm Res}(\hat{w}(0)\Gamma(s)X^s) + O_{n,[k:\Q],\epsilon}\left(|q_*\inv(\pi)|^{-1+\epsilon}X^{1-\epsilon}\right)\\
			&=O_{n,[k:\Q],\epsilon}\left(|q_*\inv(\pi)|^{-1+\epsilon}X^{1-\epsilon}\right).
		\end{align*}

		Thus, we have shown that
		\[
			\{\psi\in q_*^{-1}(\pi) : |\inv(\psi)|\le X\} \ll_{|T|,[k:\Q],\epsilon} \frac{|H^1_{ur}(k,T(\pi))|}{|q_*\inv(\pi)|^{1-\epsilon}}X(\log X)^{b(\pi)-1}.
		\]
		This gives Theorem \ref{thm:main_pointwise}(2).

		We remark that $q_*\inv \asymp (\disc_{G/T})^{|T|}$ with $G/T$ being viewed in the regular representation. Theorem \ref{thm:main_pointwise}(3) now follows, as the uniform coefficients satisfy
		\begin{align*}
			\sum_{\pi\in q_*\Sur_{\inv}(G_k,G;X)} \frac{|H_{ur}^1(k,T(\pi))|}{|q_*\inv(\pi)|^{1-\epsilon}} &\ll \sum_{F\in \mathcal{F}_{|G/T|,k}(G/T;X^{\frac{1}{|T|}})} |\disc(F/\Q)|^{d(\hat{T})/2 - |T| + \epsilon}
		\end{align*}
		following from the upper bound proven in Lemma \ref{lem:H1_first_bound}. By Abel summation, this is further bounded by
		\begin{align*}
			\sum_{\pi\in q_*\Sur_{\inv}(G_k,G;X)} \frac{|H_{ur}^1(k,T(\pi))|}{|q_*\inv(\pi)|^{1-\epsilon}} \ll& X^{\frac{d(\hat{T})}{2|T|} - 1 + \epsilon}\#\mathcal{F}_{|G/T|,k}(G/T;X^{\frac{1}{|T|}})\\
			&+ \left(|T| - \frac{d(\hat{T})}{2}\right)\int_1^{X^{1/|T|}}\! t^{d(\hat{T})/2-|T|-1+\epsilon} \#\mathcal{F}_{|G/T|,k}(G/T,t)dt\\
			\ll& X^{\frac{d(\hat{T})}{2|T|} + \frac{1}{|T|} - 1 + \epsilon} + O(1),
		\end{align*}
		where $d(\hat{T})$ is the number of generators of $\hat{T}$ and the last line uses the bounds for Galois extensions in \cite[Proposition 1.3]{ellenberg_venkatesh2006}, which imply that $\#\mathcal{F}_{|G/T|,k}(G/T;X) \ll X$.
		
		The fact that $d(\hat{T}) \le |T|-1$ implies
		\[
			\frac{d(T)}{2|T|} + \frac{1}{|T|} - 1 + \epsilon < -\frac{1}{2}+\epsilon < 0.
		\]
		Thus, we have shown that the sum
		\[
			\sum_{\pi\in q_*\Sur_{\inv}(G_k,G)} \frac{|H_{ur}^1(k,T(\pi))|}{|q_*\inv(\pi)|^{1-\epsilon}} = O(1)
		\]
		is convergent, so the result follows from the conclusion to Theorem \ref{thm:main_pointwise}.
	\end{proof}

\bibliographystyle{alpha}
\bibliography{references,MyLibrary}

\end{document}